\documentclass[reqno]{amsart}
\usepackage{amsfonts}
\usepackage{amsmath}
\usepackage{amsthm, amscd}
\usepackage{tikz-cd}
\usepackage{mathtools}
\usepackage{amssymb}
\usepackage{mathrsfs}
\usepackage{enumerate}
\usepackage[alphabetic]{amsrefs}
\usepackage[mathscr]{euscript}
\usepackage{etex}
\usepackage{hyperref}
\usepackage{svg}
\hypersetup{
  colorlinks = true,
  linkcolor  = black
}
\usepackage{color}

\usepackage{tikz}
\usepackage{xypic}

\allowdisplaybreaks

\theoremstyle{plain}\newtheorem{Theorem}{Theorem}[section]
\theoremstyle{plain}\newtheorem{Corollary}[Theorem]{Corollary}
\theoremstyle{plain}\newtheorem{Lemma}[Theorem]{Lemma}
\theoremstyle{plain}\newtheorem{Definition}[Theorem]{Definition}
\theoremstyle{plain}\newtheorem{Conjecture}[Theorem]{Conjecture}
\theoremstyle{plain}
\theoremstyle{plain}
\theoremstyle{plain}
\theoremstyle{plain}\newtheorem{Example}[Theorem]{Example}
\theoremstyle{plain}
\theoremstyle{plain}\newtheorem*{Theorem*}{Theorem}

\makeatletter
\newtheorem*{rep@theorem}{\rep@title}
\newcommand{\newreptheorem}[2]{%
\newenvironment{rep#1}[1]{%
 \def\rep@title{#2 \ref{##1}}%
 \begin{rep@theorem}}%
 {\end{rep@theorem}}}
\makeatother
\theoremstyle{plain}\newreptheorem{theorem}{Theorem}

\theoremstyle{remark}\newtheorem{remark}[Theorem]{Remark}
\theoremstyle{remark}

\numberwithin{equation}{section}

\newcommand{\bC}{\mathbb{C}}

\newcommand{\bK}{\mathbb{K}}
\newcommand{\bP}{\mathbb{P}}
\newcommand{\bQ}{\mathbb{Q}}
\newcommand{\bR}{\mathbb{R}}
\newcommand{\bS}{\mathbb{S}}
\newcommand{\bV}{\mathbb{V}}
\newcommand{\bZ}{\mathbb{Z}}

\newcommand{\clB}{\mathcal{B}}
\newcommand{\clC}{\mathcal{C}}
\newcommand{\clD}{\mathcal{D}}
\newcommand{\clF}{\mathcal{F}}

\newcommand{\clI}{\mathcal{I}}
\newcommand{\clJ}{\mathcal{J}}
\newcommand{\clL}{\mathcal{L}}
\newcommand{\clM}{\mathcal{M}}
\newcommand{\clN}{\mathcal{N}}
\newcommand{\clO}{\mathcal{O}}
\newcommand{\clP}{\mathcal{P}}
\newcommand{\clR}{\mathcal{R}}
\newcommand{\clS}{\mathcal{S}}

\newcommand{\clU}{\mathcal{U}}
\newcommand{\clW}{\mathcal{W}}
\newcommand{\clX}{\mathcal{X}}
\newcommand{\clY}{\mathcal{Y}}
\newcommand{\clZ}{\mathcal{Z}}

\newcommand{\coker}{{\normalfont\text{coker }}}
\newcommand{\ind}{{\normalfont\text{ind }}}
\newcommand{\SU}{{\normalfont\text{SU}}}
\newcommand{\delbar}{\bar\partial}
\newcommand{\sC}{\mathscr{C}}
\newcommand{\sM}{\mathscr{M}}

\begin{document}

\author{Shaoyun Bai}
\address{Department of Mathematics, Columbia University, New York 10027, USA}
\email{sb4841@columbia.edu}
\author{Mohan Swaminathan}
\address{Department of Mathematics, Stanford University, California 94305, USA}
\email{mohans@stanford.edu}
\title[Towards an extension of Taubes' Gromov invariant]{Bifurcations of embedded curves and towards an extension of Taubes' Gromov invariant to Calabi--Yau $3$-folds}
\begin{abstract}
    We define an integer-valued virtual count of embedded pseudo-holomorphic curves of two times a primitive homology class and arbitrary genus in symplectic Calabi--Yau $3$-folds, which can be viewed as an extension of Taubes' Gromov invariant. The construction depends on a detailed study of bifurcations of moduli spaces of embedded pseudo-holomorphic curves which is partially motivated by Wendl's recent solution of Bryan--Pandharipande's super-rigidity conjecture.
\end{abstract}

\maketitle

\setcounter{tocdepth}{1}

\section{Introduction}\label{intro}

Let $(X,\omega)$ be a symplectic Calabi--Yau $3$-fold, i.e., it is a closed symplectic manifold of (real) dimension $6$ with $c_1(TX,\omega) = 0\in H^2(X,\bZ)$. After choosing an $\omega$-compatible almost complex structure $J \in \clJ(X,\omega)$, the genus $g$ Gromov--Witten invariant of $(X,\omega)$ of class $0 \neq A \in H_2(X,\bZ)$ with no marked points, denoted by $\text{GW}_{A,g}(X)$, is defined by certain counts of stable $J$-holomorphic maps into $X$ and it is independent of $J$. In general, $\text{GW}_{A,g}(X) \in \bQ$ because some $J$-holomorphic maps may possess non-trivial automorphism groups. 

On the other hand, Taubes' Gromov invariants \cite{Taubes-Gr} define integer-valued invariants for symplectic $4$-manifolds by counting \emph{embedded} pseudo-holomorphic curves. The construction heavily relies on intersection theory in dimension $4$. The goal of this paper is to partially generalize Taubes' Gromov invariants to define invariants of symplectic Calabi--Yau $3$-folds by studying moduli spaces of embedded pseudo-holomorphic curves.

\subsection{Statement of the main Theorem}
We start by introducing some notions important for our discussion. The term \emph{$J$-curve} throughout this paper stands for a simple $J$-holomorphic map with smooth and connected domain into an almost complex manifold $(X,J)$. Any closed connected embedded $2$-dimensional almost complex submanifold $C\subset X$, gives rise to a $J$-curve unique up to the reparametrization of the domain. In the following, unless otherwise specified, $(X, \omega)$ will denote a symplectic Calabi--Yau $3$-fold.
\begin{Definition}\label{emb-wall}
    Define $\normalfont\clW_\text{emb}\subset\clJ(X,\omega)$ to be the subset of almost complex structures $J$ for which there exists either a simple and non-embedded $J$-curve or two simple $J$-curves with distinct but intersecting images.
\end{Definition}
It is a basic fact (see \cite[Theorem 2.4.3]{DW-20} for example) that $\clW_\text{emb}$ is of codimension $2$ in $\clJ(X,\omega)$. Thus, along generic paths $\gamma\subset\clJ(X,\omega)$, we may assume that all simple curves are embedded and pairwise disjoint by appealing to the Sard--Smale theorem.\footnote{Strictly speaking, in order to apply the Sard--Smale theorem, we need to pass to appropriate Banach completions, refer to Section \ref{technical-transversality-remarks} for details.}
\begin{Definition}[Normal deformation operator]\label{normal-operator}
    Let $J\in\clJ(X,\omega)$ and let $C\subset X$ be a smooth closed embedded $J$-curve, i.e. the tangent bundle of $C$, written as $T_C$, is invariant under the $J$--action. The normal bundle $N_C = T_X|_C/T_C$ is then naturally a complex vector bundle. The linearized Cauchy--Riemann operator on $T_X|_C$ descends to an operator
    \begin{align}\label{eqn:normalCR}
        D^N_{C,J}:\Omega^0(C,N_C)\to\Omega^{0,1}(C,N_C)
    \end{align}
    called the \emph{normal deformation operator} of the embedded curve $C$.
\end{Definition}
Given an embedded $J$-curve $C$, the Fredholm index of the operator \eqref{eqn:normalCR} is $0$. When $\ker D^N_{C,J} = 0$, and thus $\coker D^N_{C,J} = 0$, recall that we can define the sign of $C$ as a point in the moduli space of $J$-curves as follows.
\begin{Definition}[Sign]\label{count-sign-def}
    Let $C\subset X$ be an embedded $J$-curve such that $D^N_{C,J}$ is an isomorphism. We then define the sign of the curve $C$, denoted by
    \begin{align}
        \normalfont\text{sgn}(C)\in\{-1,+1\}
    \end{align}
    as follows. We deform $D^N_{C,J}$ in a $1$-parameter family to a $\bC$-linear operator which is an isomorphism and define the sign of $C$ according to the parity of the associated spectral flow\footnote{While the term ``spectral flow" is standard, it is a slight abuse of terminology in this context as our operators are not self-adjoint and there are no eigenvalues involved. More explicitly, to define $\text{sgn}(C)$, we take a \emph{generic} $1$-parameter family $\{D_t\}_{t\in[0,1]}$ of $\bR$-linear Cauchy--Riemann operators with $D_0 = D^N_{C,J}$ and $D_1$ a $\bC$-linear isomorphism and count the (parity of the) finite number of $t\in(0,1)$ for which $D_t$ fails to be an isomorphism.}. Alternatively, we consider the straight line homotopy $\{D_t\}_{0\le t\le 1}$ from $D_0 = D^N_{C,J}$ to its $\bC$-linear part $D_1$ and trivialize the determinant line bundle
    \begin{align}
        t\mapsto\det(D_t) = \det\ker D_t\otimes(\det\coker D_t)^*
    \end{align}
    and define the sign of $C$ to be the sign of the resulting isomorphism between $\det(D_0) = \det(0)\otimes\det(0)^* = \bR$ and $\det(D_1)$ with the complex orientation.
\end{Definition}
The notion of super-rigidity requires that any holomorphic multiple cover of any $J$-curve $C \subset X$ does not admit any non-trivial normal deformation. More precisely, we have the following definition.
\begin{Definition}[Super-rigidity]\label{super-rigidity-def}
    Given an embedded $J$-curve $C\subset X$ as in Definition \ref{normal-operator}, we say that it is \emph{super-rigid} if for every non-constant holomorphic stable map $\varphi:C'\to C$, the pullback operator
    \begin{align}
        \varphi^*D^N_{C,J}:\Omega^0(C',\varphi^*N_C)\to\Omega^{0,1}(\tilde C',\tilde\varphi^*N_C)
    \end{align}
    is injective where $\tilde C'\to\tilde C'$ is the normalization of $C'$ and $\tilde\varphi:\tilde C'\to C$ is the map induced by $\varphi$. We say that a given almost complex structure $J\in\clJ(X,\omega)$ is \emph{super-rigid} if and only if $J \in \clJ(X, \omega) \setminus \normalfont\clW_\text{emb}$ and every embedded $J$-curve is super-rigid.
\end{Definition}

\begin{remark}\label{rem:nodal-CR}
    We recall the definition of $\varphi^*D^N_{C,J}$ in detail for the reader's convenience. When $C'$ is non-singular, the pullback operator $\varphi^*D^N_{C,J}$ is uniquely characterized as follows. If $D^N_{C,J} - \delbar = A\in\Omega^{0,1}(U,\text{End}_{\bR}(N_C))$ in a trivialization of $N_C$ over an open subset $U\subset C$, where $\delbar$ is the standard Cauchy--Riemann operator under this trivialization, then $\varphi^*D^N_{C,J} = \delbar + \varphi^*A$ in the corresponding trivialization of $\varphi^*N_C$ over $\varphi^{-1}(U)$. Holomorphicity of $\varphi$ ensures that $\varphi^*A$ is a $1$-form of type $(0,1)$. When $C'$ is nodal, the operator $\varphi^*D^N_{C,J}$ is (by definition) the restriction of the operator $\tilde\varphi^*D^N_{C,J}$ to the subspace $\Omega^0(C',\varphi^*N_C)\subset\Omega^0(\tilde C',\tilde\varphi'^*N_C)$ consisting of those smooth sections of $\tilde\varphi^*N_C\to\tilde C'$ which descend to continuous sections of $\varphi^*N_C\to C'$, i.e., have matching values at the two pre-images in $\tilde C'$ of each nodal point in $C'$. A nearly identical discussion can be used to define the linearized Cauchy--Riemann operator on $\varphi^*T_X$ for the stable map $\varphi$.
\end{remark}

Recent work of Wendl \cite{Wendl-19} shows that the set of super-rigid compatible almost complex structures is a Baire subset of $\clJ(X,\omega)$, verifying the generic super-rigidity conjecture of Bryan--Pandharipande \cite{bryan-pand}. An important consequence of super-rigidity is the following finiteness result. Suppose $J$ is super-rigid, then given any class $A \in H_2(X, \bZ)$ and $g \in \bZ_{\geq 0}$, there are only finitely many embedded $J$-curves of class $A$ and genus $g$ and they have disjoint images in $X$, see \cite[Corollary 1.4]{Wendl-19}. We remark that the work of Doan--Walpuski \cite[Theorem 1.6]{doan2021castelnuovo} proves this finiteness result with only an \emph{a priori} energy bound, without any assumption on the genus, though we do not need this stronger result in the present work.

Let $\overline{H}_2(X,\bZ)$ denote $H_2(X,\bZ)$ modulo its torsion subgroup\footnote{We work modulo torsion purely for notational convenience. With a bit more notation, it is possible to state all our results with $H_2(X,\bZ)$ instead of $\overline{H}_2(X,\bZ)$.}. Given a homology class $A \in \overline{H}_2(X, \bZ)$, we say it is \emph{primitive} if $\omega(A) > 0$ and we cannot write $A = kB $ for any $B \in \overline{H}_2(X, \bZ)$ and integer $k \geq 2$. For a fixed primitive homology class $A$, the works of Zinger \cite{zinger2011comparison} and Doan--Walpuski \cite{doan2019counting} show that the signed count (c.f. Definition \ref{count-sign-def}) of embedded $J$-curves of class $A$ and genus $h \geq 0$ for $J \in \clJ(X, \omega)\setminus \normalfont\clW_\text{emb}$ defines an invariant of $(X, \omega)$ independent of $J$. If $A$ is not primitive, although there are only finitely many embedded $J$-curves of class $A$ and genus $h$ for a super-rigid almost complex structure $J$, the direct signed count of all such $J$-curves may not define an invariant of $(X, \omega)$ because \emph{bifurcations} could happen when varying the almost complex structures. The main result of this paper is to define a virtual count of embedded $J$-curves of two times a primitive homology class for \emph{all genera}. We adopt the notation $\text{Gr}$ in order to be consistent with Taubes' pioneering work in dimension $4$ \cite{Taubes-Gr}.
\begin{Theorem}\label{thm:Gr-invariance}
Fix a primitive class $A\in\overline{H}_2(X,\bZ)$ and an integer $h\ge 0$. For a super-rigid $J\in\clJ(X,\omega)$, define the \emph{virtual count of embedded genus $h$ curves of class $2A$} to be the integer
  \begin{align}\label{Gr-def}
    \normalfont\text{Gr}_{2A,h}(X,\omega, J) = \sum_{C':\,2A,h}\text{sgn}(C') + \sum_{g\le h}\,\sum_{C:\,A, g}\text{sgn}(C)\cdot\text{w}_{2,h}(D^N_{C,J})
  \end{align}
where the first sum is over embedded $J$-curves $C'$ of genus $h$ and class $2A$, the second sum is over all genera $0\le g\le h$ and embedded $J$-curves $C$ of genus $g$ and class $A$, $\normalfont\text{sgn}$ is as in Definition \ref{count-sign-def} and the integers $\normalfont\text{w}_{2,h}$ are as in Definition \ref{lwc-double}.

Then, $\normalfont\text{Gr}_{2A,h}(X,\omega, J)$ is independent of the choice of super-rigid $J$ and 
\begin{align}
    \normalfont\text{Gr}_{2A,h}(X,\omega) := \normalfont\text{Gr}_{2A,h}(X,\omega, J)
\end{align} 
defines a symplectic invariant of $(X, \omega)$.
\end{Theorem}
In other words, we construct correction terms to the naive signed count of embedded $J$-curves of class $2A$ and genus $h$ from embedded $J$-curves of class $A$ and genus $g \leq h$ to get a well-defined invariant. For a thorough discussion of the correction terms in \eqref{Gr-def}, see Section \ref{lwc-development}.

\begin{remark}[Symplectic deformation invariance]
    In fact, $\normalfont\text{Gr}_{2A,h}$ is a symplectic deformation invariant of $(X,\omega)$. To see this, one can use the fact that any $1$-parameter symplectic deformation $\{\omega_t\}$ can be accompanied by a compatible family $\{J_t\}$ of almost complex structures.
\end{remark}

\subsection{Ingredients of the proof}
The proof of Theorem \ref{thm:Gr-invariance} (carried out in Section \ref{Gr-invariance-proof}) relies on a detailed study of bifurcations of moduli spaces of embedded $J$-curves when varying $J$ in a generic $1$-parameter family. To be more precise, although super-rigidity is a generic condition for $J \in \clJ(X, \omega)$, there are codimension $1$ subsets of $\clJ(X, \omega)$ where super-rigidity fails to hold. Given a path of almost complex structures $\gamma: [-1,1] \to \clJ(X, \omega)$ such that both $\gamma(-1)$ and $\gamma(1)$ are super-rigid, we need to understand the bifurcations of moduli spaces of embedded pseudo-holomorphic curves when $\gamma$ intersects these codimension $1$ subsets. 

The first key ingredient is a \emph{necessary} condition for the occurrence of a bifurcation. For a more general and complete statement, see Theorem \ref{compactness-result}.
\begin{Theorem}[special case of Theorem \ref{compactness-result}]\label{thm:intro-compa}
    Given a sequence $t_k \rightarrow t'$ in $(-1,1)$, define $J_k := \gamma(t_k)$ and $J := \gamma(t')$ so that $J_k \rightarrow J$ in the $C^{\infty}$-topology. Suppose $h \in \bZ_{\geq 0}$ and we have a sequence of genus $h$ embedded $J_k$-holomorphic curves $\varphi_k: \Sigma^{\prime}_k \to X$ converging to a stable $J$-holomorphic map $C' \xrightarrow{\varphi} C \subset X$ in the Gromov topology. Here $C$ is an embedded $J$-curve of genus $g \leq h$ and $\varphi: C' \to C$ is a stable genus $h$ map of degree $k$ with target $C$. 
    
    Then we either have $g=h, k=1$ and $\varphi: C' \to C$ is an isomorphism, or the pullback map $\varphi^*:\ker D^N_{C,J}\to\ker\varphi^*D^N_{C,J}$ is only injective but not surjective.
\end{Theorem}
Theorem \ref{thm:intro-compa} allows us to exclude the possibility of bifurcations from nodal curves (possibly with ghost components) into embedded curves once the path $\gamma$ is chosen generically. When the homology class represented by $\varphi_k: \Sigma^{\prime}_k \to X$ is primitive, the degree $k$ is then necessarily equal to $1$ and Theorem \ref{thm:intro-compa} implies that the limit $\varphi$ must be an isomorphism. In particular, the stable map $C' \xrightarrow{\varphi} C \subset X$ does not have any ghost irreducible component. Combined with Ionel--Parker's construction of Kuranishi models for birth-death bifurcations of embedded $J$-curves \cite[Theorem 6.2]{IP-GV}, one can prove the aforementioned result of Zinger and Doan--Walpuski regarding counting embedded $J$-curves representing primitive homology classes without appealing to Gromov--Witten theory as in \cite{zinger2011comparison} or the delicate gluing analysis as from \cite{doan2019counting}, see also Remark \ref{compactness-remark}. The full power of Theorem \ref{thm:intro-compa} is used in Section \ref{sec_GT}, which tells us ultimately we just need to worry about bifurcations from (branched) double covers with smooth domains in order to check the invariance of $\normalfont\text{Gr}_{2A,h}(X,\omega)$.

The second key ingredient is a \emph{sufficient} condition for bifurcations to occur. We construct Kuranishi models of moduli spaces of embedded curves based on a detailed analytical argument. The technical condition of ``elementary wall type" includes all the cases in the construction of $\normalfont\text{Gr}_{2A,h}(X,\omega)$, and in fact is satisfied by a large class of examples, see Examples \ref{exa_bif_1}--\ref{exa_bif_4}.
\begin{Theorem}[see Theorem \ref{main-bifurcation}]\label{thm_int_1}
    Let $\{J_t\}_{t\in[-1,1]}$ be a generic path in $\clJ(X,\omega)$. Assume that there exists an embedded rigid $J_0$-curve $\Sigma \hookrightarrow X$ along with a $d$-fold genus $h$ branched multiple cover $\varphi:\Sigma'\to\Sigma$ which has non-trivial normal deformations. If this cover determines an elementary wall type (see Definition \ref{elem-wall}), then $\normalfont\text{Aut}(\varphi)\subset\bZ/2\bZ$ and, moreover, the change in the signed count of embedded curves of genus $h$ and class $d[\Sigma]$ near $\varphi$ is given by $\pm 2/|\normalfont\text{Aut}(\varphi)|$.
\end{Theorem}
The notion of ``generic path" deserves some further comments. The work of Wendl \cite{Wendl-19}, recalled in Section \ref{super-rigidity-recall}, specifies the codimension $1$ subsets (``walls"\footnote{These consist of almost complex structures violating super-rigidity (Definition \ref{super-rigidity-def}).}) of $\clJ(X, \omega)$ where a bifurcation could occur. Away from a subset of higher codimension, these subsets have manifold structures and we require that the path $\{J_t\}_{t\in[-1,1]}$ only intersects this manifold locus and the intersection is transverse. However, this notion of transversality is \emph{not} sufficient for our discussion. In fact, we define another subset of the walls with positive codimension, called the \emph{degeneracy locus}, in Definition \ref{bad-stratum-defined} and the path $\{J_t\}_{t\in[-1,1]}$ is required to avoid the degeneracy locus as well. The condition of ``elementary wall type" is needed in this step to guarantee the existence and abundance of such paths. Avoiding the degeneracy locus allows us to guarantee the non-vanishing of certain Taylor coefficients of the Kuranishi maps which leads to Theorem \ref{thm_int_1}. 

For instance, suppose we have a path $\{J_t\}_{t \in [-1,1]} \subset \clJ(X, \omega)$ and an embedded $J_0$-holomorphic torus $T \subset X$ with $\ker(D_{T,J_0}^{N}) = 0$, i.e. the moduli space of $J_0$-curves is cut out transversely at $T$. Suppose there exists an unbranched holomorphic double cover from another torus $\varphi: T' \to T$ such that the normal deformation operator $\varphi^*D_{T,J_0}^{N}$ satisfies $\dim_{\bR} \ker(\varphi^*D_{T,J_0}^{N}) = 1$.
Using Kuranishi reduction, one can show that the parametric moduli space of $J_t$-holomorphic genus $1$ maps in class $2[T]$ (near the $J_0$-holomorphic map $T'\xrightarrow{\varphi} T\subset X$) is locally modeled on 
\begin{equation}\label{kur-local-model}
    x \cdot (a x^2 + bt + \cdots) = 0 \text{ for some } a,b \in \bR
\end{equation}
for $(x,t) \in \bR \times [-1,1]$ with a $\bZ/2\bZ$-action given by $(x,t) \to (-x, t)$. Wendl's result allows us to ensure that $b \neq 0$ by choosing the path $\{J_t\}_{t \in [-1,1]}$ generically. However, this is \emph{not} sufficient for our application. We argue that one can choose the path generically so that we have $a \neq 0$ as well. When $a$ and $b$ are both non-zero, the higher order terms in \eqref{kur-local-model} can be removed by a suitable coordinate change and the relevant moduli space here is indeed described by the ``pitchfork bifurcation".

We remark that our arguments fill in a technical gap in Taubes' paper \cite[Lemma 5.18]{Taubes-Gr} and the applicability of our result goes beyond the bifurcation coming from double unbranched covers of tori as in \cite{Taubes-Gr}. As far as we know, this is the first result studying equivariant bifurcations of pseudo-holomorphic curves in dimension $6$.

The above two ingredients give a complete characterization of the moduli spaces of embedded pseudo-holomorphic curves of class $2A$ and any genus along a generic path of almost complex structures connecting two given super-rigid almost complex structures. The third key ingredient is the construction of the correction terms in \eqref{thm:Gr-invariance}, which are called \emph{linear wall crossing invariants} in Definition \ref{lwc-double}. These correction terms are constructed by studying the spaces of $\bR$-linear Cauchy--Riemann operators on the normal bundles of embedded $J$-curves in a symplectic Calabi--Yau $3$-fold and they should be thought of as a version of ``equivariant spectral flow". This algebraic invariant is used to compensate the change in the signed count of embedded curves of class $2A$ when a bifurcation happens and it serves as the last piece of the construction of $\normalfont\text{Gr}_{2A,h}(X,\omega)$. We remark that these linear wall-crossing invariants only suffice for our definition of $\normalfont\text{Gr}_{2A,h}(X,\omega)$. We plan to investigate generalization of these invariants for more general homology classes in future work. See \cite[Section 6.2]{bai2020equivariant} for related constructions in a different context, in which more complicated corrections are presented due to the occurrence of more complicated bifurcations.

Constructions of this kind probably date back to Walker's extension of the $\SU(2)$-Casson invariant to rational homology $3$-spheres \cite{casson-walker} and similar ideas appear in other literature of gauge theory, including the $\SU(3)$-Casson invariants for integer homology $3$-spheres \cite{boden1998the} and the Seiberg--Witten invariants for homology $S^1 \times S^3$ \cite{mrowka2011seiberg}. Recent work of Eftekhary \cite{eftekhary2020counting} applies similar ideas to a counting problem of periodic orbits. However, in our case, even the definition of the wall crossing numbers requires some work (see Section \ref{lwc-development}).

\subsection{Relations with other work}
There are some works in the literature (besides Taubes' Gromov invariant) which have close relations with the current work.
\subsubsection{Gopakumar--Vafa invariants}
Given a symplectic Calabi--Yau $3$-fold $(X, \omega)$, it was conjectured by Gopakumar--Vafa \cite{gopakumar1998m} that the Gromov--Witten generating function of $X$ satisfies the formula
\begin{align}\label{eqn_GV}
    \sum_{\substack{A \neq 0 \\ g \geq 0}} \text{GW}_{A,g}(X) t^{2g-2} q^{A} = \sum_{\substack{A \neq 0 \\ h \geq 0}} n_{A ,h}(X) \sum_{k=1}^{\infty} \frac{1}{k} \left(2\sin\left(\frac{kt}{2}\right)\right)^{2h-2} q^{kA},
\end{align}
where the ``BPS numbers" $n_{A ,h}(X)$ are integers and $n_{A ,h}(X) = 0$ for a fixed $A \in H_2(X, \bZ)$ if $h$ is sufficiently large. The integrality the BPS numbers was proved by Ionel--Parker \cite{IP-GV} and the finiteness has been established by the recent work of Doan--Ionel--Walpuski \cite{doan2021gopakumar}.

However, Ionel--Parker's proof of $n_{A ,h}(X) \in \bZ$ does not illustrate the geometric meaning of these Gopakumar--Vafa invariants. Their proof proceeds as follows. \cite[Equation (8.1)]{IP-GV} shows that the Gromov--Witten generating function satisfies
\begin{align}\label{eqn:elem_clust}
    \sum_{A \neq 0, g \geq 0} \text{GW}_{A,g}(X) t^{2g-2} q^{A} = \sum_{A \neq 0, g \geq 0} e_{A,g}(X) \cdot \text{GW}^\text{elem}_g (q^A, t)
\end{align}
where $\text{GW}^\text{elem}_g$ is some universal generating series and $e_{A,g}(X) \in \bZ$ is regarded as a ``\emph{virtual count of embedded clusters in $X$}". They also show that $n_{A ,h}(X)$ and $e_{A,g}(X)$ actually determine each other via the formula \cite[Equation (8.9)]{IP-GV}
\begin{align}\label{e-vs-GV}
    n_{A ,h}(X) = \sum_{\substack{A = dB \\ B \in \bZ_{>0}}} \sum_{g=0}^h n_{d,h}(g) \cdot e_{B,g}(X)
\end{align}
where $n_{d,h}(g) \in \bZ$ are certain universal coefficients determined by $d,h,g$. Therefore the integrality of $n_{A ,h}(X)$ is a result of integrality of $e_{A,g}(X)$.
\begin{Conjecture}\label{conj:intro}
The invariant $\normalfont\text{Gr}_{2A,h}(X,\omega)$ is equal to the virtual count of embedded clusters $e_{2A, h}$ from \eqref{eqn:elem_clust} for all $h \geq 0$ where $A$ is a primitive homology class.
\end{Conjecture}
We show that Conjecture \ref{conj:intro} is true for $h=0$, see Theorem \ref{thm_GWGV}. The proof relies on a wall-crossing result of virtual Euler numbers, which might be of independent interest.
\begin{Theorem}[see Theorem \ref{euler-change-thm}]\label{thm_int_2}
    Suppose $N\to C$ is a rank $2$ complex vector bundle of degree $2g-2$ on a curve $C$ of genus $g$. Let $\clD = \{D_t\}_{t\in[-1,1]}$ be a $1$-parameter family of Cauchy--Riemann operators on $N$. Fix integers $d\ge 2$, $h\ge g$ and consider the moduli stack $\overline\clM_h(C,d)$ of stable genus $h$ maps to $C$ of degree $d$. Assume further that $\clD$ is $(d,h)$-rigid (see Definition \ref{coker-bundle-def}) except over a finite set
    \begin{align}\label{eqn_intro_delta}
        \Delta\times\{0\}\subset\clM_h(C,d)\times[-1,1]
    \end{align}
    corresponding to holomorphic maps with smooth domains and that $\clD$ is a simple wall-crossing of type $(d,h)$ at $t=0$ (see Definition \ref{simple-linear-wall-crossing}).
    Then, for small $t\ne 0$, the assignment 
    \begin{align}\label{intro-cokernel}
        [f:C'\to C]\in\overline\clM_h(C,d)\mapsto\coker(f^*D_t)
    \end{align}
    gives a well-defined vector bundle on $\overline\clM_h(C,d)$ equipped with a natural orientation.\footnote{The orientation here is defined using the modified determinant line \eqref{eqn:modified-orientation-line} and differs from the usual orientation on \eqref{intro-cokernel} when $\normalfont\text{sgn}(D_0) = -1$. See Remark \ref{further-orientation-explanation} for details.} Denote the Euler number of this bundle by $e_{d,h}(C,D_\pm)$ for small $\pm t>0$. Then, we have
    \begin{align}
        \normalfont e_{d,h}(C,D_+) - e_{d,h}(C,D_-) = \sum_{p\in\Delta}\frac{2\cdot\text{sgn}(\clD,p)}{|\text{Aut}(p)|}
    \end{align}
    where the numbers $\normalfont\text{sgn}(\clD,p)\in\{-1,+1\}$ are determined by the local behavior of $\clD$ near $p$.
\end{Theorem}
The main obstacle to proving Conjecture \ref{conj:intro} in general is to drop the assumption that $\Delta$ is a finite subset and corresponds to stable maps with smooth domains in the above statement.

When $(X,\omega)$ is further assumed to be a smooth projective variety, Maulik--Toda \cite{maulik2018gopakumar} propose a sheaf-theoretic definition of the Gopakumar--Vafa invariants, based on earlier works of Hosono--Saito--Takahashi \cite{hosono2001relative} and Kiem--Li \cite{kiem2012categorification}. Despite the beauty of the definition, it is not known how to prove the deformation invariance of Maulik--Toda's ansatz and the proof of formula \eqref{eqn_GV} using this definition seems to be out of reach in general. Moreover, the sheaf-theoretic definition relies crucially on the algebraicity of $X$, making it hard to generalize it to general symplectic Calabi--Yau $3$-folds (equipped with not necessarily integrable almost complex structures). Conjecture \ref{conj:intro} proposes a geometric definition of the Gopakumar--Vafa invariants of two times a primitive homology class via the formula \eqref{e-vs-GV} by counting embedded pseudo-holomorphic curves. 
\begin{remark}
Conjecture 7.4.5 in Cox--Katz's book \cite{cox1999mirror} conjectures that the genus $0$ BPS numbers should come from virtual counts of embedded genus $0$ $J$-curves, with correction terms coming from genus $0$ $J$-curves of lower degree. Later on, Katz proposes an algebro-geometric counterpart of this conjecture in \cite{katz2008genus} using Donaldson--Thomas invariants. Our conjecture and result are consistent with these expectations.
\end{remark}

\subsubsection{Counting calibrated submanifolds}
For a given symplectic manifold $(X, \omega)$ and an $\omega$-compatible almost complex structure $J$, almost complex submanifolds, in particular embedded $J$-curves, define \emph{calibrated submanifolds} of $X$. Another example of calibrated submanifolds are associative submanifolds of $7$-dimensional Riemannian manifolds with holonomy group $G_2$.

There is a proposal from Doan--Walpuski \cite{DWas} which defines weights from counting solutions to gauge-theoretic equations on the associative submanifolds such that the weighted count of associatives defines an invariant. In this work, a symplectic definition of Pandharipande--Thomas invariant is discussed, based on a weighted count of embedded pseudo-holomorphic curves with weights again coming from counts of gauge-theoretic equations.

Our formula \eqref{Gr-def} could also be viewed as a weighted count of embedded pseudo-holomorphic curves which defines an invariant of $(X, \omega)$. This definition bypasses the difficulty coming from compactness issues of moduli spaces of generalized Seiberg--Witten equations which enters the proposal in \cite{DWas} and has a closer relation with the Gromov--Witten invariants.

\subsubsection{Study of multi-valued sections}
Given an embedded $J$-curve $C$ in a symplectic Calabi--Yau $3$-fold and a holomorphic (branched) cover $\varphi: C' \to C$, any non-zero element $s \in \ker(\varphi^* D_{C,J}^N)$ defines a \emph{multi-valued section} of the normal bundle $N$ of $C$ in $X$ by associating any point $x \in C$ with the values of $s$ over the set $\varphi^{-1}(x)$. Our bifurcation result (Theorem \ref{thm_int_1}) can be interpreted as a sufficient condition to perturb a (holomorphic) multi-valued section into a ``nearby" embedded pseudo-holomorphic curve in $X$. Moreover, the discussion in Section \ref{lwc-development} which is required for the definition of linear wall crossing numbers studies some generic behaviors of these multi-valued sections.

As speculated in \cite{donaldson2019deformations}, $2$-valued harmonic $1$-forms are important for special Lagrangian geometry. Our study of multi-valued sections might have some applicability in this context, see Remark \ref{rmk_slag} and Remark \ref{rmk_mult_slag}.

\subsection{Technical remarks concerning transversality arguments}\label{technical-transversality-remarks}

To carry out the numerous generic transversality arguments in this paper (and in many other places in the theory of pseudo-holomorphic curves), one needs to appeal to the Sard--Smale theorem (for Fredholm maps of second countable \emph{Banach} manifolds) after showing that certain universal moduli spaces are smooth (respectively, $C^\infty$ subvarieties as in \cite[Appendix C]{Wendl-19}) by showing that the linearized operators associated to their defining equations are surjective (respectively, satisfy a lower bound on their rank).

Let $f:\clX\to\clY$ be the map for which we wish to prove that the set $\clY_\text{reg}\subset\clY$ of regular values is a Baire subset (with $\clY$ being a space of smooth objects and both $\clX$, $\clY$ being second countable). In our applications, this will be the space of smooth almost complex structures on a manifold or Cauchy--Riemann operators on a fixed bundle or spaces of paths/homotopies in the aforementioned spaces or, in our most complicated situation, a space of embedded pseudo-holomorphic curves equipped with the data of a branched cover of the domain (Definition \ref{bad-stratum-defined}). The first step of the proof is always to establish \emph{formal regularity}, i.e., surjectivity (or rank lower bounds) for the linearized operators of the defining equations of $\clX$ over $\clY$ at the level of smooth objects (i.e. without taking any Banach space completions or restricting to any Banach subspace).

Now, for most applications (including all the applications in this paper), one wishes to prove $C^\infty$-genericity statements (e.g. ``simple $J$-holomorphic maps are unobstructed for a $C^\infty$-generic $J$"). However, the space $\clY$ does not form a Banach manifold and thus, we cannot directly appeal to the Sard--Smale theorem. Fortunately, there is a well-known technique to overcome this difficulty (i.e. to upgrade formal regularity into genericity of regular values) which we briefly outline below. For the complete details and non-trivial examples of the application of this technique, the reader is referred to the very clear and detailed exposition in \cite{Wendl-19} (specifically the proofs of Theorems 5.26 and D therein). Going from formal regularity to the statement that regular values are generic involves showing that $\clY_\text{reg}$ is a countable intersection of dense open subsets of $\clY$. 

For this, we use the popular trick (due to Taubes) of exhausting $\clX$ by a sequence of subspaces $\clX^1\subset\clX^2\subset\cdots$ such that $f|_{\clX^N}:\clX^N\to\clY$ is proper for all $N$. We then define $\clY^N_\text{reg}\subset\clY$ to be the subset of points $y\in\clY$ for which every point in $f^{-1}(y)\cap\clX^N$ is regular. Now, since regularity is an open property on $\clX$ and $f|_{\clX^N}$ is proper, it follows that $\clY^N_\text{reg}$ is open. As a consequence, $\clY_\text{reg} = \bigcap_{N\ge 1}\clY^N_\text{reg}$ is a countable intersection of open subsets of $\clY$. The exact construction of the subsets $\clX^N\subset\clX$ is straightforward but somewhat tedious and we refer the reader to the proofs of Theorem 5.26 and Theorem D in \cite{Wendl-19} for details.

 We are left to argue that each $\clY^N_\text{reg}\subset\clY$ is dense. For this, we fix any reference point $y_\text{ref}\in\clY$ introduce Floer's $C_\epsilon$ spaces centered at $y_\text{ref}$, which we denote by $\clY_\epsilon(y_\text{ref})$. These spaces are Banach manifolds, consist entirely of smooth objects close to $y_\text{ref}$ but carry a topology depending on the parameter $\epsilon$ which is much finer than the $C^\infty$ topology. For any $x_\text{ref}\in \clX_\text{ref} = f^{-1}(y_\text{ref})$, formal regularity and the implicit function theorem imply that for all sufficiently small $\epsilon$, the space $\clX_\epsilon(y_\text{ref}) = f^{-1}(\clY_\epsilon(y_\text{ref}))$ is a Banach manifold (or $C^\infty$ subvariety of expected codimension) near $x_\text{ref}$. By the second countability of $\clX_\text{ref}$, we can choose a uniform $\epsilon$ which works for all $x_\text{ref}\in\clX_\text{ref}$. Thus, we obtain an open neighborhood $\clX_\epsilon(y_\text{ref})^\text{reg}$ of $\clX_\text{ref}$ in $\clX_\epsilon(y_\text{ref})$ such that the restriction of $f$ as a map $\clX_\epsilon(y_\text{ref})^\text{reg}\to\clY_\epsilon(y_\text{ref})$ is a smooth map of Banach manifolds. Now, the Sard--Smale theorem shows that we can find a sequence $y_k\in\clY_\epsilon(y_\text{ref})$ of regular values of this map converging to $y_\text{ref}$ in $\clY$. For any $N\ge 1$, the properness of $\clX^N$ over $\clY$ and the openness of $\clX_\epsilon(y_\text{ref})^\text{reg}\subset\clX_\epsilon(y_\text{ref})$ now show that we must have $y_k\in\clY^N_\text{reg}$ for all sufficiently large $k$ (depending on $N$). Since $y_\text{ref}$ was arbitrary, this shows that each $\clY^N_\text{reg}$ is dense in $\clY$. 

In view of the above, we adopt the convention of only proving formal regularity in all transversality proofs that arise in this paper (and omitting the use of Floer's $C_\epsilon$ spaces and Taubes' trick to upgrade from formal regularity to genericity of regular values). Consistent with this, we use the notions of codimension of a subset of $\clY$ (or a $C^\infty$ subvariety of $\clY$), the index of the map $f$ etc in this formal sense as a shorthand for the more detailed proof scheme above. The corresponding notions for maps of Banach manifolds have standard definitions (a subset is of codimension $\ge k$ if it is covered by the images of countably many Fredholm maps with indices $\le -k$; $C^\infty$ subvarieties of codimension $\ge k$ are defined as in \cite[Appendix C]{Wendl-19} and we note that they are, in particular, subsets of codimension $\ge k$ by \cite[Proposition C.1]{Wendl-19}). The above formal notions will translate to the corresponding notions for Banach manifolds when we apply the implicit function theorem in the above proof scheme to a suitable $C_\epsilon$ space.

\subsection{Structure of the paper}

We end the introduction by outlining the structure of this paper. In Section \ref{recall}, we recall the necessary background on (twisted) Cauchy--Riemann operators and super-rigidity (following \cite{Wendl-19}). Section \ref{compactness-section} establishes a necessary condition for bifurcations to occur (Theorem \ref{compactness-result}) while Section \ref{bifurcation} gives a sufficient condition (Theorem \ref{main-bifurcation}). In Section \ref{sec_GT}, we apply the results of the previous two sections to construct the extension of Taubes' Gromov invariant  for symplectic Calabi--Yau $3$-folds (counting curves of genus $h\ge 0$ lying in two times a primitive homology class) and prove the main result of this paper (Theorem \ref{thm:Gr-invariance}). Finally, in Section \ref{sec_GV}, we prove a result on the virtual Euler numbers of certain obstruction bundles (Theorem \ref{euler-change-thm}) which allows us to identify our invariant with the Gopakumar--Vafa invariant when the genus $h$ is equal to $0$ (Theorem \ref{thm_GWGV}). 

\subsection*{Acknowledgements.} We would like to thank Chris Wendl for helpful correspondences and for pointing out an error in an earlier version of Lemma \ref{ker-coker-final}. We would also like to thank Eleny Ionel for a careful reading of an earlier version of this paper and for several useful suggestions on improving the exposition. We are also grateful to Aleksander Doan, Sridhar Venkatesh and Thomas Walpuski for useful discussions. Moreover, the authors would like to thank their Ph.D. advisor John Pardon for his constant support and encouragement. Finally, the authors thank the referees for their meticulous reading and suggestions, which improved the exposition substantially. 

\subsection*{Conventions.} Unless otherwise stated, all vector spaces, vector bundles and their determinants are defined over $\mathbb{R}$. Dimensions, Hom-spaces and tensor products without any subscripts should be understood to be over ${\mathbb{R}}$. $\overline{H}_2(X,\bZ)$ will denote $H_2(X,\bZ)$ modulo torsion.

\section{Review of embedded curves and super-rigidity}\label{recall}

In this section, we will summarize the background on super-rigidity in a form relevant for our discussion in sections \ref{bifurcation}--\ref{sec_GT}. The notation closely follows that of \cite[Section 2]{Wendl-19}. We will begin by recalling some basic facts about branched covers of curves and Cauchy--Riemann operators (in \textsection\ref{covers-recall}) and then move onto super-rigidity (in \textsection\ref{super-rigidity-recall}).

\subsection{Cauchy--Riemann operators and Galois covers}\label{covers-recall}

We will consistently use the word `curve' to mean `irreducible smooth projective algebraic curve over $\bC$' (i.e. a closed connected smooth Riemann surface). A `morphism' between curves will mean a `non-constant holomorphic map'.

\begin{Definition}[Function field]
    Let $\Sigma$ be a curve. Define $\bC(\Sigma)$ to be the \emph{function field} of $\Sigma$ (i.e., the field of meromorphic functions on $\Sigma$). Given a degree $d\ge 1$ morphism of curves $\varphi:\Sigma'\to\Sigma$, we get the associated degree $d$ field extension $\varphi^*:\bC(\Sigma)\hookrightarrow\bC(\Sigma')$. It is a basic fact from algebraic geometry that this is a contravariantly functorial one-to-one correspondence (i.e. an equivalence of categories) between curves and fields of transcendence degree $1$ over $\bC$.
\end{Definition}
\begin{Definition}[Galois closure]
    Let $\varphi:\Sigma'\to\Sigma$ be a morphism of curves. We say that $\varphi$ is \emph{Galois} if and only if the associated field extension is Galois. In general, $\varphi$ may not be Galois and we can form a Galois closure, a morphism $\tilde\Sigma\to\Sigma'\xrightarrow{\varphi}\Sigma$, by taking a Galois closure of the function field extension $\varphi^*:\bC(\Sigma)\to\bC(\Sigma')$. For a Galois cover $\tilde\Sigma\to\Sigma$, we define the Galois group $\normalfont\text{Gal}(\tilde\Sigma/\Sigma)$ to be the group of automorphisms of $\tilde\Sigma$ over $\Sigma$ (or equivalently the Galois group of the associated function field extension).
\end{Definition}

Given a morphism $\varphi:\Sigma'\to\Sigma$ of curves, we call its critical points as \emph{ramification points} and its critical values as \emph{branch points}. The map $\varphi$ restricts to a covering map with connected domain over the complement of (the pre-images of) the branch points, and this is called \emph{the associated covering space}. It is a basic fact that $\varphi:\Sigma'\to\Sigma$ is Galois if and only if the associated covering space is normal. Moreover, the associated covering space of a Galois closure of $\varphi$ is a normal closure of the associated covering space of $\varphi$.

\begin{Definition}[Branched covers and Galois groups]\label{cover-combinatorial-type}
    Let the following data be given: integers $g\ge 0$ and $d\ge 2$ and \emph{branching data} ${\bf b} = ({\bf b}_1,\ldots,{\bf b}_r)$ where each ${\bf b}_i = (b_i^1,\ldots,b_i^{q_i})$ is a tuple of positive integers adding up to $d$ with $b_i^1\ge\cdots\ge b_i^{q_i}$ and $b_i^1\ge 2$. Define 
    \begin{align}
        b = \sum_{i=1}^r\sum_{j=1}^{q_i}(b_i^j-1)
    \end{align}
    and define $h$ using the formula $2h - 2 = b + d(2g-2)$. Assume that $h$ is a non-negative integer. Now, let $\varphi:\Sigma'\to\Sigma$ be a morphism of curves of \emph{type} $(g,{\bf b})$, i.e., $h = \normalfont\text{genus}(\Sigma')$, $g = \normalfont\text{genus}(\Sigma)$, $d = \deg(\varphi)$ and $\varphi$ has exactly $r$ branched points $w_1,\ldots,w_r\in\Sigma$ with $\varphi^{-1}(w_i) = \{\zeta_i^1,\ldots,\zeta_i^{q_i}\}$ and $\varphi$ is $b_i^j$-to-$1$ in a punctured neighborhood of each $\zeta_i^j\in\Sigma'$. We define the \emph{generalized automorphism group} of $\varphi$ to be the Galois group of a Galois closure of $\varphi$. The isomorphism class of the generalized automorphism group is independent of the choice of Galois closure.
    
    Given a finite group $G$ and a genus $g$ curve $\Sigma$, define $\clM^d_{{\bf b},G}(\Sigma)$ to be the orbifold of pairs $(\varphi:\Sigma'\to\Sigma,\iota)$, up to isomorphism, where $\varphi$ is a morphism of type $(g,{\bf b})$ and $\iota$ is an isomorphism of $G$ with the generalized automorphism group of $\varphi$ (defined using a Galois closure $\tilde\Sigma\to\Sigma'\xrightarrow{\varphi}\Sigma$). Under $\iota$, $\normalfont\text{Gal}(\tilde\Sigma/\Sigma')\subset\text{Gal}(\tilde\Sigma/\Sigma)$ corresponds to $H\subset G$ which is called \emph{subgroup of $G$ corresponding to $\varphi$}. The orbifold structure is explained in Lemma \ref{orbifold-of-G-covers} below.
\end{Definition}

\begin{Lemma}\label{orbifold-of-G-covers}
    $\clM^d_{{\bf b},G}(\Sigma)$ is naturally an orbifold.
\end{Lemma}
\begin{proof}
    A point of $\clM^d_{{\bf b},G}(\Sigma)$ explicitly consists of a morphism $\varphi:\Sigma'\to\Sigma$ of type $(g,{\bf b})$, a Galois closure $\tilde\varphi:\tilde\Sigma\to\Sigma$ for $\varphi:\Sigma'\to\Sigma$ (with a factorization $\tilde\Sigma\to\Sigma'\to\Sigma$) and an isomorphism $\iota:G\xrightarrow{\simeq}\text{Gal}(\tilde\Sigma/\Sigma)$. An isomorphism between points $(\varphi,\tilde\varphi,\iota)$ and $(\phi,\tilde\phi,\hat\iota)$ is the data of an isomorphism of the Galois covers $\tilde\varphi$ and $\tilde\phi$ over $\Sigma$, which descends to an isomorphism between $\varphi$ and $\phi$ over $\Sigma$ and is compatible with the isomorphisms $\iota$ and $\hat\iota$. 
    
    To describe the orbifold structure of $\clM^d_{{\bf b},G}(\Sigma)$, we show how to get local charts near any given point $(\varphi,\tilde\varphi,\iota)$. Using \cite[Example 3.6]{Wendl-19}, and the discussion immediately preceding it, we obtain a family $\tau\in B^{2r}\mapsto[\Sigma'_\tau\xrightarrow{\varphi_\tau}\Sigma]$ parametrizing all morphisms to $\Sigma$ of type $(g,{\bf b})$ near $\varphi:\Sigma'\to\Sigma$. Further, by construction, this family comes with a (smooth) family of Galois closures $\tilde\Sigma_\tau\xrightarrow{\tilde\varphi_\tau}\Sigma$ (called \emph{minimal regular presentations} in \cite{Wendl-19}). The given isomorphism $\iota:G\xrightarrow{\simeq}\text{Gal}(\tilde\Sigma/\Sigma)$ uniquely deforms to isomorphisms $\iota_\tau:G\xrightarrow{\simeq}\text{Gal}(\tilde\Sigma_\tau/\Sigma)$. The family $\{(\varphi_\tau,\tilde\varphi_\tau,\iota_\tau)\}_{\tau\in B^{2r}}$ almost gives us the desired orbifold chart for $\clM^d_{{\bf b},G}(\Sigma)$ near $(\varphi,\tilde\varphi,\iota)$ except that it doesn't account for the isomorphisms of $\text{Gal}(\tilde\Sigma_\tau/\Sigma)$ with $G$ other than $\iota_\tau$ (as well as the action of the group of automorphisms of $\tilde\Sigma_\tau/\Sigma$ which descend to $\Sigma'_\tau/\Sigma$ on these data). Let $H\subset G$ be the subgroup corresponding to $\text{Gal}(\tilde\Sigma/\Sigma')\subset\text{Gal}(\tilde\Sigma/\Sigma)$ under $\iota$. The group of automorphisms of $\tilde\Sigma/\Sigma$ which descend to $\Sigma'/\Sigma$ gets identified under $\iota$ with the normalizer $N_G(H)$ of $H\subset G$. 
    
    Using $\{(\varphi_\tau,\tilde\varphi_\tau,\iota_\tau)\}_{\tau\in B^{2r}}$, it now follows that $[(B^{2r}\times\text{Aut}(G)) / N_G(H)]$ provides a local orbifold chart for $\clM^d_{{\bf b},G}(\Sigma)$ near $(\varphi,\tilde\varphi,\iota)$ where $N_G(H)$ acts on $B^{2r}$ trivially and acts on the finite set $\text{Aut}(G)$ via the group homomorphism $N_G(H)\to\text{Aut}(G)$ which maps any $\gamma\in N_G(H)$ to conjugation by $\gamma$. Observe that the orbifold structure on $\clM^d_{{\bf b},G}(\Sigma)$ is \emph{purely ineffective} with every point having isotropy group isomorphic to the center $Z_G$ of $G$. In particular, the underlying coarse space of the orbifold is actually a manifold, i.e., has no quotient singularities.
\end{proof}

We will next discuss Cauchy--Riemann operators on curves and their pullbacks along Galois covers. Denote by $E$ a complex vector bundle over a curve $\Sigma$. Recall that an \emph{$\mathbb{R}$-linear Cauchy--Riemann operator} is a first-order real linear partial differential operator $D: \Omega^0(\Sigma,E)\to\Omega^{0,1}(\Sigma,E)$ such that for any smooth function $f \in C^{\infty}(\Sigma, \mathbb{R})$ and section $s \in \Omega^0(\Sigma,E)$, we have
\begin{equation}
    D(fs) = (\overline{\partial}f) s + f D(s),
\end{equation}
where $\overline{\partial}f = \frac12 (df + i\cdot df \circ j)$ and $j$ denotes the complex structure on $\Sigma$.
\begin{Definition}[Serre dual]\label{serre}
    Let $D:\Omega^0(\Sigma,E)\to\Omega^{0,1}(\Sigma,E)$ be an $\bR$-linear Cauchy--Riemann operator. Define the \emph{Serre dual bundle} of $E$ to be 
    \begin{align}\label{eqn:dual-E}
        E^\dagger = \Lambda^{1,0}T^*_\Sigma\otimes_\bC E^*
    \end{align}
    and the \emph{Serre dual operator} $D^\dagger$ to be the $\bR$-linear Cauchy--Riemann operator on $E^\dagger$ determined by the following condition: for all $s\in\Omega^0(\Sigma,E)$ and $s^\dagger\in\Omega^0(\Sigma,E^\dagger)$,
    \begin{align}\label{eqn:serre-dual-operator}
        \normalfont\text{Re}\left(\int_\Sigma \langle Ds,s^\dagger\rangle + \langle s,D^\dagger s^\dagger\rangle\right) = 0.
    \end{align}
    In \eqref{eqn:serre-dual-operator}, the $\bC$-linear pairing $\langle Ds,s^\dagger\rangle\in\Omega^{1,1}(\Sigma,\bC)$ between the sections $Ds\in\Omega^{0,1}(\Sigma,E)$ and $s^\dagger\in\Omega^0(\Sigma,E^\dagger) = \Omega^{1,0}(\Sigma,E^*)$ is defined by the wedge product followed by the natural pairing between $E$ and and its $\bC$-linear dual $E^*$. The $\bC$-linear pairing $\langle s,D^\dagger s^\dagger\rangle$ between the sections $s\in\Omega^0(\Sigma,E)$ and $D^\dagger s^\dagger\in\Omega^{0,1}(\Sigma,E^\dagger) = \Omega^{1,1}(\Sigma,E^*)$ is defined analogously. 
    
    The $\bR$-linear pairing $\normalfont\text{Re}\int_\Sigma\langle\cdot,\cdot\rangle$ gives the \emph{Serre duality} isomorphisms
    \begin{align}
        \ker D^\dagger &= (\coker D)^* \\
        \ker D &= (\coker D^\dagger)^*.
    \end{align}
\end{Definition}
\begin{Definition}[Twisted operators]\label{twisted-def}
    Let $D:\Omega^0(\Sigma,E)\to\Omega^{0,1}(\Sigma,E)$ be an $\bR$-linear Cauchy--Riemann operator on the curve $\Sigma$ and let $\psi:\tilde\Sigma\to\Sigma$ be a Galois cover with Galois group $G = \normalfont\text{Gal}(\tilde\Sigma/\Sigma)$. We can then form the pullback operator
    \begin{align}\label{galois-pullback}
        \psi^*D:\Omega^0(\tilde\Sigma,\psi^*E)\to\Omega^{0,1}(\tilde\Sigma,\psi^*E)
    \end{align}
    which is $G$-equivariant with respect to the natural linear $G$-actions on the source and target. Let $\theta$ be a linear $G$-representation over $\bR$. We define the \emph{twisted operator}
    \begin{equation}
    D_{\psi,\theta}: \normalfont\text{Hom}^G(\theta^*,\Omega^0(\tilde\Sigma,\psi^*E)) \to \normalfont\text{Hom}^G(\theta^*,\Omega^{0,1}(\tilde\Sigma,\psi^*E))
    \end{equation}
    by applying the functor $\normalfont\text{Hom}^G(\theta^*,-) = (\theta\otimes-)^G$ to \eqref{galois-pullback}.
    The operator $D_{\psi,\theta}$ is naturally linear over the endomorphism algebra $\bK_\theta = \normalfont\text{End}^G(\theta)$.
\end{Definition}
\begin{remark}
    This definition is seen to be equivalent to the twisted operators from \cite[\textsection 2.2 and \textsection 3.4]{Wendl-19} as follows. We may regard the elements of $\text{Hom}^G(\theta^*,\Omega^0(\tilde\Sigma,\psi^*E)) = (\theta\otimes\Omega^0(\tilde\Sigma,\psi^*E))^G$ as sections of $(\tilde\Sigma\times V)/G$ on the orbifold $\tilde\Sigma/G$. The natural map $\tilde\Sigma/G\to\Sigma$ is an isomorphism except over the branch points. The quotient $(\tilde\Sigma\times V)/G$ defines a vector bundle on the complement of the branch points and coincides with the twisted bundle of \cite{Wendl-19}. The advantage of our formulation over the one in \cite{Wendl-19} is that it avoids the need to introduce punctures and weighted Sobolev spaces.
\end{remark}
\begin{remark}\label{invariant-sections}
    Suppose $G\to G'$ is a quotient map between groups, corresponding to a factorization $\tilde\Sigma\xrightarrow{\varphi}\Sigma'\xrightarrow{\psi'}\Sigma$ of $\psi$ as the composition of two Galois covers $\varphi$, $\psi'$ and $G' = \normalfont\text{Gal}(\Sigma'/\Sigma)$. Also, suppose that $\theta'$ is a $G'$-representation and that $\theta$ is its pullback to $G$. Then, $\varphi^*$ gives a natural map of complexes from $D_{\psi',\theta'}$ to $D_{\psi,\theta}$ and this map is seen to be a quasi-isomorphism (using the removable singularity theorem near the branch points of $\varphi$ for $D$ and $D^\dagger$). Similarly, if $H\subset G$ is any subgroup and $\tilde\Sigma' = \tilde\Sigma/H$ with the associated morphism $\psi':\tilde\Sigma'\to\Sigma$, then the natural pullback map from $(\psi')^*D$ to $(\psi^*D)^H$ is also a quasi-isomorphism (i.e. an isomorphism on kernels and cokernels).
\end{remark}

The following result, which appears in \cite[Remark 1.4.7]{DW-20} in a slightly different form, will also be useful for the discussion below.

\begin{Lemma}\label{operator-decomp}
    Let $\theta_i$ for $i=1,\ldots,p$ be an enumeration of all the irreducible representations of $G$ over $\bR$ with $\bK_i = \normalfont\text{End}^G(\theta_i)$. Then, the tautological map
    \begin{align}
        \bigoplus_i\theta_i^*\otimes_{\bK_i}D_{\psi,\theta_i} \to \psi^*D
    \end{align}
    of $G$-equivariant two term complexes is a quasi-isomorphism.
\end{Lemma}
\begin{proof}
    This is because for any finite dimensional $\bR[G]$-module $V$, the tautological map $\bigoplus_i\theta_i^*\otimes_{\bK_i}\text{Hom}^G(\theta_i^*,V)\to V$ is an isomorphism of $\bR[G]$-modules by Maschke's theorem. We apply this to the (co)kernels of the operators in question.
\end{proof}

Given a family of curves and a corresponding family of Cauchy--Riemann operators, we introduce the following definition which measures the variation of the Cauchy--Riemann operators on a particular finite-dimensional vector space. It will be used later in the analysis of Kuranishi maps.

Let $\pi:\clC\to\clB$ be a smooth family of curves, i.e., $\pi$ is a proper submersion between two manifolds $\clC$ and $\clB$ with fibers of $\bR$-dimension $2$ along with a smooth endomorphism $j$ of $\ker(d\pi) = T_{\clC/\clB}$ satisfying $j^2 = -1$. Let $E\to\clC$ be a $\bC$-vector bundle and set
\begin{align}
    F = \Lambda^{0,1}T^*_{\clC/\clB,j}\otimes_\bC E.
\end{align}
Assume that we have a smooth family $\clD = \{D_b\}_{b\in\clB}$ of Cauchy--Riemann operators
\begin{align}
    D_b:\Omega^0(\clC_b,E_b)\to\Omega^0(\clC_b,F_b)
\end{align}
where $\clC_b = \pi^{-1}(b)$, $E_b = E|_{\clC_b}$ and $F_b = F|_{\clC_b}$ for $b\in\clB$.
\begin{Definition}\label{linearize-family-op}
    For any point $a\in\clB$, define the \emph{linearized map} at $a$ associated to the family $\clD$
    \begin{align}\label{linearization-inv}
        \clD'|_a:T_a\clB\to\normalfont\text{Hom}(\ker D_a,\coker D_a)
    \end{align}
    as follows. Given $\sigma\in\ker D_a$, extend it arbitrarily to a section $\tilde\sigma\in\Omega^0(\clC,E)$. Then, the section $\tilde\tau = \clD(\tilde\sigma)\in\Omega^0(\clC,F)$ defined by $\tilde\tau|_{\clC_b} = D_b(\tilde\sigma|_{\clC_b})$ for $b\in\clB$ vanishes along $\clC_a$ and thus, has a well-defined normal derivative
    \begin{align}
        \nabla\tilde\tau|_a \in\normalfont\text{Hom}(N_{\clC_a/\clC},F_a) = \text{Hom}(T_a\clB,\Omega^0(\clC_a,F_a))
    \end{align}
    where we have used the identification $N_{\clC_a/\clC} = \pi^*(T_a\clB)$ in the last equality. Now, for any $v\in T_a\clB$ we set $\clD'|_a(v)\cdot\sigma\in\coker D_a$ to be the projection of the section $\nabla\tilde\tau|_a\cdot v\in\Omega^0(\clC_a,F_a)$.  
\end{Definition}
The operator \eqref{linearization-inv} is well-defined because if $\sigma = 0$, then any choice of $\tilde\sigma$ has a well-defined normal derivative
    \begin{align}
        \nabla\tilde\sigma|_a \in\normalfont\text{Hom}(N_{\clC_a/\clC},E_a) = \text{Hom}(T_a\clB,\Omega^0(\clC_a,E_a)) 
    \end{align}
    which satisfies $D_a(\nabla\tilde\sigma|_a\cdot v) = \nabla\tilde\tau|_a\cdot v$ for all $v\in T_a\clB$. 
Note that if a finite group $G$ acts compatibly on $\clB$, $\clC$, $E$, $F$, with the action on $\clB$ being trivial, such that the family $\clD$ is $G$-equivariant, then \eqref{linearization-inv} factors through $\normalfont\text{Hom}^G(\ker D_a,\coker D_a)$.
\begin{remark}
    It can be seen from Definition \ref{linearize-family-op} that we can define the linearized map $\clD'|_a$ by first choosing suitable trivializations to regard $\clD$ as a family of $G$-equivariant operators
    \begin{align}
        b\mapsto D_b:\Omega^0(\clC_a,E_a)\to\Omega^0(\clC_a,F_a)
    \end{align}
    and setting $\clD'|_a(v)\cdot\sigma$ to be the projection of $\left[\frac d{dt}D_{\gamma(t)}\sigma|_{t = 0}\right]$ onto $\coker D_a$, where $\gamma$ is any smooth path in $\clB$ with $\gamma(0) = a$ and $\dot\gamma(0) = v$.
\end{remark}
\subsection{Super-rigidity}\label{super-rigidity-recall}
{For the remainder of this section, we assume that the manifold $X$ is a symplectic Calabi--Yau $3$-fold with symplectic form $\omega$ and, moreover, all almost complex structures under consideration will be assumed to be $\omega$-compatible.}

Fix $g,d,\bf b$ as in Definition \ref{cover-combinatorial-type} and a finite group $G$. Now, given $A\in \overline{H}_2(X,\bZ)$, define the universal moduli space
\begin{align}
    \clM^d_{{\bf b},G}(\clM_g^{*}(X,A)) \to \clJ(X,\omega)
\end{align}
to consist of tuples $(J,\Sigma,\varphi:\Sigma'\to\Sigma,\iota)$ up to isomorphism where $J\in\clJ(X,\omega)$, $\Sigma\subset X$ is an embedded $J$-curve of genus $g$ in class $A$ and $\varphi$ is a branched cover of type $(g,\bf b)$ along with an isomorphism $\iota$ between $G$ and its generalized automorphism group (this space becomes a purely ineffective smooth Banach orbifold if we impose appropriate $C_\epsilon$ regularity conditions as explained in \textsection\ref{technical-transversality-remarks}). Given any such tuple, we can then form the associated Galois closure $\psi:\tilde\Sigma\to\Sigma$ and consider the $G$-equivariant operator $\psi^*D^N_{\Sigma,J}$. Its kernel and cokernel are finite dimensional real $G$-representations.
\begin{Definition}[Loci of failure of super-rigidity]\label{failure-strata-def}
Given two nonzero finite dimensional $G$-representations $\bf k$ and $\bf c$ over $\bR$, define the universal moduli space
    \begin{align}\label{wall-to-acs}
        \clM^d_{{\bf b},G}(\clM_g^{*}(X,A);{\bf k},{\bf c}) \to\clJ(X,\omega)
    \end{align}
    to be the subset of $\clM^d_{{\bf b},G}(\clM_g^{*}(X,A))$ consisting of those tuples $(J,\Sigma,\varphi:\Sigma'\to\Sigma,\iota)$, with associated Galois cover $\psi:\tilde\Sigma\to\Sigma$, for which we have isomorphisms $\ker \psi^*D^N_{\Sigma,J}\simeq\bf k$ and $\coker\psi^*D^N_{\Sigma,J}\simeq\bf c$ of $G$-representations.
\end{Definition}

\begin{remark}
    Observe that if the domain of \eqref{wall-to-acs} is non-empty, then $\bf k$ must in fact be a $G$-sub-representation of $\bf c$ and, therefore, $\text{Hom}^G({\bf k},{\bf c})\ne 0$. Indeed, when we decompose $\psi^*D^N_{\Sigma,J}$ as in Lemma \ref{operator-decomp}, each twisted operator appearing in the decomposition has index $\le 0$ by \cite[Lemma 2.15]{Wendl-19}.
\end{remark}

We now state the version of \cite[Theorem D]{Wendl-19} that we will need for the rest of the article. Some immediate corollaries are also explained below.
\begin{Theorem}[Stratification]\label{stratification}
    In the setting of Definition \ref{failure-strata-def}, away from an exceptional subset of infinite codimension of $\clJ(X,\omega)$, the subspace 
    \begin{align}\label{stratification-inclusion}
        \clM^d_{{\bf b},G}(\clM_g^{*}(X,A);{\bf k},{\bf c}) \subset \clM^d_{{\bf b},G}(\clM_g^{*}(X,A))    
    \end{align}
    is a sub-orbifold\footnote{Again, as explained in \textsection\ref{technical-transversality-remarks}, this should be interpreted to mean that we get a Banach sub-orbifold after passing to suitable $C_\epsilon$ completions.} of codimension given by $\dim_\bR\normalfont\text{Hom}^G({\bf k},{\bf c})$.
\end{Theorem}

\begin{remark}\label{reconciling-formulations}
We explain in detail why our version (Theorem \ref{stratification}) follows from the somewhat differently formulated \cite[Theorem D]{Wendl-19}. Arguing exactly as in the proof of Lemma \ref{orbifold-of-G-covers}, we find that the natural forgetful map from our universal moduli space $\clM^d_{{\bf b},G}(\clM_g^*(X,A))$ to Wendl's universal moduli space $\clM^d_{\bf b}(\clM_{g,m}(A;\ell_1,\ldots,\ell_m))$, with $m = \ell_1 = \cdots = \ell_m = 0$, is a finite-to-one local diffeomorphism. The finite fibers correspond to the different choices of isomorphism between the fixed group $G$ and the generalized automorphism groups of the covers in question. 

Keeping track of this explicit isomorphism allows us to globally define the loci $\clM^d_{{\bf b},G}(\clM_g^{*}(X,A);{\bf k},{\bf c})$ for a pair of non-trivial $G$-representations ${\bf k},{\bf c}$ as in Definition \ref{failure-strata-def}. Since \cite{Wendl-19} doesn't record the isomorphism of the generalized automorphism group with $G$ explicitly, the analogous spaces $P({\bf k},{\bf c})$  are defined only locally (see \cite[Section 3.5]{Wendl-19} for their definition). It is then proved in \cite[Lemma 3.24]{Wendl-19} that $P({\bf k},{\bf c})$ is locally characterized by constancy of the dimension of the kernel of the normal deformation operator pulled back to a Galois closure. By Remark 2.14 and the discussion immediately following Definition 6.1 in \cite{Wendl-19}, we see that $P({\bf k},{\bf c})$ are exactly the ``walls" to which the codimension formula of \cite[Theorem D]{Wendl-19} applies. Finally, observe that the codimension formula (introduced in \cite[Definition 2.11]{Wendl-19}) coincides with the dimension of $\text{Hom}^G({\bf k},{\bf c})$. In fact, for any point $p =(J,\Sigma,\varphi:\Sigma'\to\Sigma,\iota)$ lying in the domain of \eqref{wall-to-acs}, the vector spaces $\text{Hom}^G({\bf k},{\bf c})$ and $\text{Hom}^G(\ker\psi^*D^N_{\Sigma,J},\coker\psi^*D^N_{\Sigma,J})$ are isomorphic, where $\psi:\tilde\Sigma\to\Sigma$ is the associated Galois cover as in Definition \ref{failure-strata-def}. The latter vector space is the normal space to the inclusion \eqref{stratification-inclusion} at $p$.

The exceptional subset of infinite codimension referenced in the statement of Theorem \ref{stratification} is explained in the following remark. We conclude that Theorem \ref{stratification} is indeed a consequence of \cite[Theorem D]{Wendl-19}.
\end{remark}

\begin{remark}
    In the statement of Theorem \ref{stratification}, we are using the fact that the Petri condition (see \cite[Section 5]{Wendl-19} for more details) holds on $\clJ(X,\omega)$ away from an exceptional subset which has infinite codimension. In practice, we can ignore this exceptional set as in the next corollary.
\end{remark}
\begin{Corollary}\label{transverse-to-wall}
    Let $P$ be a second countable manifold and let $\clJ_P(X,\omega)$ denote the space of smooth $\omega$-compatible almost complex structures parametrized by $P$ with the $C^\infty$ topology. Let $\clP = \{J_s\}_{s\in P}$ be an element in $\clJ_P(X,\omega)$ for which the associated map $P\to\clJ(X,\omega)$ is transverse to the forgetful map \eqref{wall-to-acs} over a closed submanifold $Q\subset P$. Define $\clJ(X,\omega)_{\clP,Q}$ to be the space of $\clP'\in\clJ_P(X,\omega)$ which coincide with $\clP$ over $Q\subset P$. Then, the subset of $\clJ(X,\omega)_{\clP,Q}$ consisting of those $\clP'$ for which the associated map $P\to\clJ(X,\omega)$ is transverse to \eqref{wall-to-acs} is a Baire set, and therefore, dense.
\end{Corollary}
\begin{Definition}[Wall]\label{wall-strata-def}
    For $\delta = (g,d,{\bf b},G,{\bf k},{\bf c},A)$ as in Definition \ref{failure-strata-def}, in particular with ${\bf k}$ and ${\bf c}$ being nonzero $G$-representations, define
    \begin{align}
        \clW^d_{{\bf b},G}(g,A,{\bf k},{\bf c})\subset\clJ(X,\omega)
    \end{align}
    to be the image of the forgetful map \eqref{wall-to-acs}. Define the \emph{wall} $\clW\subset\clJ(X,\omega)$ to be the set of all $J$ which are not super-rigid. We then have
    \begin{align}\label{wall-decomp}
        \clW = \normalfont\clW_\text{emb}\cup\bigcup_{\delta}\clW^d_{{\bf b},G}(g,A,{\bf k},{\bf c})
    \end{align}
    by passing to Galois covers in Definition \ref{super-rigidity-def}. We refer to each term in the union in \eqref{wall-decomp} as a \emph{stratum} of the wall $\clW$.
\end{Definition}

\begin{remark}
    Even though we find it convenient to occasionally speak of the wall strata as if they were submanifolds of $\clJ(X,\omega)$, in practice we will be using the maps \eqref{wall-to-acs} rather than their images (the wall strata) in all our arguments. For example, when we discuss transversality of a path in $\clJ(X,\omega)$ to a wall stratum, we actually mean that the path is transverse to the corresponding map \eqref{wall-to-acs}. Similarly, when we speak of the normal space to the wall inside $\clJ(X,\omega)$ at a point, we will actually mean the normal space of the map \eqref{wall-to-acs} at a point where it is an immersion.
\end{remark}

We will now determine a way to separate $\clW$ into those strata which have codimension $1$ in $\clJ(X,\omega)$ and those which have codimension $\ge 2$. We do this essentially by rewording the proof of \cite[Theorem A]{Wendl-19} for $X$ and following the remarks from \cite[\textsection 2.4]{Wendl-19}.

\begin{Lemma}[Wall: codimension 1 strata]\label{top-strata}
    Away from a codimension $2$ subset of $\clJ(X,\omega)$, the codimension $1$ stratum of $\clW \subset J(X,\omega)$ is a union of those strata $\clW^d_{{\bf b},G}(g,A,{\bf k},{\bf c})$ which are non-empty such that
    \begin{enumerate}[\normalfont(i)]
        \item each $b_i^j\in\{1,2\}$, i.e., each ramification point is simple and,
        \item the $G$-representation $\bf k$ is faithful, irreducible and of real type, i.e., its equivariant endomorphism algebra $\normalfont\text{End}^G({\bf k})$ is isomorphic to $\bR$.
    \end{enumerate}
\end{Lemma}
\begin{proof}
    Begin by noting that $\clW_\text{emb}$ is of codimension $2$ in $\clJ(X,\omega)$. Now, take $J$ in $\clW\setminus\clW_\text{emb}$ and not in the exceptional set from Theorem \ref{stratification}. This means that we can find a minimum possible $d\ge 1$, an embedded $J$-curve $\Sigma\subset X$ and a $d$-fold branched cover $\varphi:\Sigma'\to\Sigma$ such that $\ker\varphi^*D^N_{\Sigma,J}\ne 0$. Let $\psi:\tilde\Sigma\to\Sigma$ be a Galois closure of $\varphi$ and let $G = \text{Gal}(\tilde\Sigma/\Sigma)$ while $H = \text{Gal}(\tilde\Sigma/\Sigma')$. Let $\theta_i$ for $i=1,\ldots,p$ denote the complete list of irreducible representations of $G$ over $\bR$. For each $i$, recall that we have the twisted operators $D^N_{\psi,\theta_i}$, as in \cite[\textsection 2.2 and \textsection 3.4]{Wendl-19} and Definition \ref{twisted-def}, whose kernels (and similarly cokernels) are related to $\psi^*D^N_{\Sigma,J}$ by the following formulas
    \begin{align}
        \label{twisted-kernel}
        \ker D^N_{\psi,\theta_i} &= \text{Hom}^G(\theta_i^*,\ker\psi^*D^N_{\Sigma,J})\\
        \label{kernel-decomp}
        \ker\psi^*D^N_{\Sigma,J} &= \bigoplus_i\theta_i^*\otimes_{\bK_i}\ker D^N_{\psi,\theta_i}
    \end{align}
    where $\bK_i = \text{End}^G(\theta_i)$, and \eqref{twisted-kernel}, \eqref{kernel-decomp} are isomorphisms of $\bK_i$-modules and $G$-representations respectively. Now, the assumption implies that we can find an $1\le j\le p$ such that $(\theta_j^*)^H\ne 0$ and $\ker D^N_{\psi,\theta_j}\ne 0$. We contend that $\theta_j$ must be a faithful $G$-representation. If not, let $(1)\ne K\lhd G$ be such that $\theta_j$ descends to $G/K$. Note that $K\not\subset H$ (since $H$ has trivial normal core) and thus, $D^N_{\Sigma,J}$ pulled back to $\tilde\Sigma/HK$ has a non-trivial kernel (contradicting the minimality of $d$). Now that we know that $\theta_i$ is faithful, we can use \cite[Lemma 2.15]{Wendl-19} to get
    \begin{align}
        \text{ind}_{\bK_j}D^N_{\psi,\theta_j} \le -2r
    \end{align}
    where $r$ is the number of distinct branch points of $\varphi$. Moreover, this estimate is strict in case $\bK_j = \bR$ unless all the ramification points of $\varphi$ are simple (i.e. they're all locally $2$-to-$1$). Now, set $t_i = \dim_\bR\bK_i$, $k_i = \dim_{\bK_i}\ker D^N_{\psi,\theta_i}$ and $c_i = \dim_{\bK_i}\coker D^N_{\psi,\theta_i}$ and note that the codimension of $\clW^d_{{\bf b},G}(g,A,{\bf k},{\bf c})$ near $J$ is
    \begin{align}
        \ge -2r + \sum_i t_ik_ic_i \ge -2r + t_jk_j(k_j + 2r)\ge 1
    \end{align}
    using Theorem \ref{stratification}. Note that the shifting $-2r$ comes from the dimension of the moduli space of branched covers. (Here, $\bf b$ is constructed using $\varphi$, $A = [\Sigma]$, $g = \text{genus}(\Sigma)$, ${\bf k} = \ker\psi^*D^N_{\Sigma,J}$ and ${\bf c} = \coker\psi^*D^N_{\Sigma,J}$.) Moreover, to obtain the equality in the last estimate, we must have $t_j = k_j = 1$, $k_i = 0$ for all $j\ne i$ (this last part follows since $\text{ind}_{\bK_i}D^N_{\psi,\theta_i}\le 0$ by \cite[Lemma 2.15]{Wendl-19}) and $\text{ind}_{\bK_j}D^N_{\psi,\theta_j} = -2r$. \cite[Lemma 2.15]{Wendl-19} now implies that all the ramification points of $\varphi$ must be locally $2$-to-$1$.
\end{proof}

\begin{remark}
    The reader may observe that in the proof of Lemma \ref{top-strata} we concluded that the branched cover in question must have \emph{simple ramification} using \cite[Lemma 2.15]{Wendl-19} which seems to make the stronger claim of \emph{simple branching}. This is a result of a difference in terminology. Indeed, the proof of \cite[Lemma 2.15]{Wendl-19} (given in \cite[Cororally 4.2]{Wendl-19}) shows that near each ramification point the cover is locally $2$-to-$1$ and but \emph{does not show that there is a unique ramification point lying over each branch point}. Example \ref{non-elem-branched} exhibits some branched covers with simple ramification (but not simple branching) allowed by Lemma \ref{top-strata}.
\end{remark}
\section{A necessary condition for bifurcations}\label{compactness-section}

In this section, we will describe a necessary condition for a sequence of simple pseudo-holomorphic curves (defined with respect to a sequence of almost complex structures) to converge to a multiple cover of an embedded curve. 

\emph{We emphasize that the main result of this section (Theorem \ref{compactness-result}) is independent of \cite{Wendl-19} and holds for almost complex manifolds in any dimension.}

Fix an even dimensional manifold $X^{2n}$ for the rest of this section. Let 
\begin{align}
    \clF:V\to\clJ(X)    
\end{align}
be a smooth family of almost complex structures on $X$ parametrized by a smooth finite dimensional manifold $V$ (i.e. $\clF$ defines a smooth section $\pi_X^*\text{End}(T_X)$ with $\clF^2 = -\text{id}$ with $\pi_X:V\times X\to X$ being the projection). Given a smooth projective curve $C$, denote by $\overline\clM_h(C,k)$ the moduli stack of genus $h$ stable maps of degree $k$ with target $C$. 
\begin{Theorem}[Necessary condition for bifurcation]\label{compactness-result}
    Let $x_\nu\to x$ be a convergent sequence in $V$, with $J_\nu := \clF(x_\nu)$ for $\nu\ge 0$ and $J := \clF(x)$. Let $h\ge 0$ be an integer and suppose that we have a sequence $(J_\nu,\varphi_\nu:\Sigma'_\nu\to X)$ of simple $J_\nu$-curves of genus $h$ converging, in the Gromov topology, to a stable map 
    \begin{align}
        (J,C'\xrightarrow{\varphi} C\subset X)    
    \end{align}
    with $C$ being a smooth embedded $J$-curve of genus $g\le h$ and $\varphi:C'\to C$ being an element of $\overline\clM_h(C,k)$ for some integer $k\ge 1$. Then, exactly one of the following must be true.
    \begin{enumerate}[\normalfont(i)]
        \item We have $g = h$, $k = 1$ and $\varphi:C'\to C$ is an isomorphism.
        \item\label{nec_cond_bif} The natural pullback map
            \begin{align}\label{kernel-pullback}
            \varphi^*:\ker D^N_{C,J}\to\ker\varphi^*D^N_{C,J}
            \end{align}
        is injective but not surjective.
    \end{enumerate}
\end{Theorem}
The alternative (\ref{nec_cond_bif}) provides a necessary condition for bifurcations of the moduli spaces of embedded pseudo-holomorphic curves to occur. Before proceeding to the proof, we make a few remarks comparing this to similar statements in the literature and point out a natural extension of this result.
\begin{remark}\label{compactness-remark}
    \begin{enumerate}[\normalfont (a)]
        \item If we further suppose that there exists a sequence $\{ C_\nu \}$ of embedded $J_\nu$-curves of genus $g$ converging to $C$ (this is ensured, for example, if $D^N_{C,J}$ is surjective), then the rescaling argument used to prove \cite[Proposition B.1]{Wendl-19} yields the weaker conclusion $\ker\varphi^*D^N_{C,J}\ne 0$.
        \item Theorem \ref{compactness-result} rules out the possibility of a sequence of simple curves of genus $h$ converging in the Gromov topology to an embedded curve of genus $g<h$ with some ghost components attached. In particular, if $A\in \overline{H}_2(X,\mathbb Z)$ is a primitive homology class in a $6$-dimensional symplectic manifold $(X,\omega)$ with $c_1(A) = 0$, then for any path $\gamma$ in $\clJ_{\normalfont\text{emb}}(X,\omega)$ (in the sense of \cite[Definition 2.36]{doan2019counting}), the moduli space $\clM_h^{\normalfont\text{emb}}(X,\gamma,A)$ of embedded genus $h$ curves of class $A$ is compact. This can be used to recover \cite[Theorem 1.5(1)]{doan2019counting}.
        \item The restriction in Theorem \ref{compactness-result} to finite dimensional families of almost complex structures suffices for all applications we are aware of but it should be possible to extend our proof to the infinite dimensional case as well. We would have to replace the appeal to Brouwer's Invariance of Domain (as in \textsection\ref{compactness-conclusion}) by a more careful analysis of the differential of the inclusion from Lemma \ref{compatible-thick}. To deal with the lack of differentiability of the moduli space near curves with nodal domains, one would have to use suitable relative tangent bundles associated to canonical ``rel--$C^\infty$ structures" on the moduli spaces as in \cite{rel-smooth}.
    \end{enumerate}
\end{remark}

The remainder of this section is devoted to the proof of Theorem \ref{compactness-result}. The main idea of the argument is given in \textsection\ref{outline} and involves counting dimensions of moduli spaces. However, there are some issues related to transversality which we need to overcome before this outline can be turned into a rigorous proof. In \textsection\ref{compactness-prep}--\textsection\ref{compactness-big-sec}, we show how to tackle these issues and complete the proof in \textsection\ref{compactness-conclusion}. 

\subsection{Outline of the argument}\label{outline}
Set $A = [C]\in H_2(X,\bZ)$ and let
\begin{align}\label{small}
    \clM^\text{emb}_g(X,\clF,A)
\end{align}
be the moduli space of pairs $(y,\Sigma)$, where $y\in V$ and $\Sigma\subset X$ is an embedded $\clF(y)$-curve of genus $g$ in class $A$. Also consider the moduli space 
\begin{align}\label{big}
    \overline\clM_h(X,\clF,kA)
\end{align}
consisting of pairs $(y,\varphi':\Sigma'\to X)$, where $y\in V$ and $\varphi'$ is a stable $\clF(y)$-holomorphic map in $X$ of genus $h$ and class $k[C]$ and the moduli space
\begin{align}\label{cover}
    \overline\clM_h(\clM^\text{emb}_g(X,\clF,A),k)
\end{align}
consisting of triples $(y,\Sigma,\psi:\Sigma'\to\Sigma)$, where $(y,\Sigma)$ lies in \eqref{small} and $(\Sigma',\psi)$ lies in $\overline\clM_h(\Sigma,k)$. There is a natural inclusion of \eqref{cover} inside \eqref{big}, given by
\begin{align}\label{cover-big-inclusion}
    (y,\Sigma,\psi:\Sigma'\to\Sigma)\mapsto(y,\Sigma'\xrightarrow{\psi}\Sigma\subset X),    
\end{align}
and the proof of Theorem \ref{compactness-result} will follow by showing that codimension of \eqref{cover} inside \eqref{big} at the point $(x,C'\xrightarrow{\varphi} C\subset X)$ is given by the dimension of the cokernel of \eqref{kernel-pullback}. However, both \eqref{big} and \eqref{cover} are not smooth in general therefore the notion of codimension is not well-defined. To remedy the situation, we will need to thicken the moduli spaces (in the sense of \cite[Definition 9.2.3]{Pardon-VFC}) compatibly so that \eqref{cover-big-inclusion} becomes an inclusion of manifolds, at least near the point $(x,\varphi)$, whose codimension can then be computed.

\subsection{Preparation}\label{compactness-prep}
\begin{Lemma}[Transversality for embedded curves]\label{trans-for-emb}
    We can find a smooth manifold $V'$ containing $V$ as a submanifold and a smooth family $\clF':V'\to\clJ(X)$ extending $\clF$ such that the moduli space
    \begin{align}
        \clM^{\normalfont\text{emb}}_g(X,\clF',A)
    \end{align}
    is a smooth manifold of the expected dimension $ = \dim V' + \normalfont\text{ind } D^N_{C,J}$ near $(x,C)$.
\end{Lemma}
\begin{proof}
    This is a variant of the argument in \cite[Proposition 3.2.1]{Mc-Sa}. Indeed, by the implicit function theorem, it suffices to arrange that the linearized map
    \begin{align}\label{emb-transverse-map}
        T_xV'&\to\coker D^N_{C,J} \\
        \label{emb-transverse-def}
        v'&\mapsto\textstyle\frac12(d\clF'|_x\cdot v')\circ d\iota\circ j
    \end{align}
    is surjective, where $\iota$ denotes the inclusion $C\subset X$ and $j$ denotes the restriction of $J$ to $T_C$. The right side of \eqref{emb-transverse-def} is an element of $\Omega^{0,1}(C,T_X|_C)$ which we are projecting to $\Omega^{0,1}(C,N_C)$ and then to $\coker D^N_{C,J}$. Note that the restriction map
    \begin{align}
        T_J\clJ = \overline{\text{End}}_J(T_X)\to\overline{\text{Hom}}_\bC(T_C,N_C) = \Omega^{0,1}(C,N_C)
    \end{align}
    is surjective and thus, we may lift a (finite) basis of $\coker D^N_{C,J}$ to $T_J\clJ$. It is now easy to define $V\subset V'$ and $\clF'$ extending $\clF$ such that \eqref{emb-transverse-map} is surjective.
\end{proof}

Before proceeding further, we pick some auxiliary data.
\begin{enumerate}
    \item Choose a Riemannian metric on $X$ such that we can identify the unit disc bundle in $N_C$ with a tubular neighborhood of $C\subset X$, with fiber over $p\in C$ denoted by $\Delta_p$ (a codimension $2$ submanifold with boundary of $X$).
    \item Choose an integer $m\ge 0$ and finitely many points $p_1,\ldots,p_m\in C$ (which are not the images of critical points of $\varphi$ or nodal points of $C'$) such that
    \begin{enumerate}
        \item $(C,p_1,\ldots,p_m)$ is a stable curve and,
        \item the evaluation map
        \begin{align}\label{ker-eval}
            \ker D^N_{C,J}\to\bigoplus_{i=1}^m N_{C,p_i}
        \end{align}
        is injective.
    \end{enumerate}
    Note that for all $(y,\Sigma)$ in \eqref{small} close to $(x,C)$, the curve $\Sigma$ intersects each $\Delta_i:=\Delta_{p_i}$ transversally at a unique point (and is disjoint from $\partial \Delta_i$). The submanifolds $\Delta_1,\ldots,\Delta_m$ will serve as ``stabilizing divisors" in the sense of \cite[Lemma 9.2.8]{Pardon-VFC}.
    \item Choose a complex manifold $\sM$ and a complex analytic family $\pi:\sC\to\sM$ of genus $g$ curves with sections $\sigma_1,\ldots,\sigma_m$ (corresponding to $m$ marked points which stabilize the fibers), equipped with a basepoint $*\in\sM$ and an isomorphism \begin{align}\label{central}
        i:(C,p_1,\ldots,p_m)\xrightarrow{\simeq}(\sC_*,\sigma_1(*),\ldots,\sigma_m(*))    
    \end{align}
    such that the resulting map $\sM\to\clM_{g,m}$ is \'etale at $*$ (i.e. provides a local orbifold chart near $*$). Here, $\sC_s$ denotes the curve $\pi^{-1}(s)$ for $s\in\sM$. The choice of the family $(\pi:\sC\to\sM,\sigma_1,\ldots,\sigma_m)$ is analogous to choosing a local Teichm\"uller slice at $(C,p_1,\ldots,p_m)$.
    \item For $1\le i\le m$, choose a numbering $\varphi^{-1}(p_i) = \{q_i^1,\ldots,q_i^k\}$. Note that each $q_i^j$ must be a smooth point of $C'$ with $d\varphi(q_i^j)\ne 0$. Moreover, $(C',\{q_i^j\})$ is a stable curve. Choose a complex manifold $\sM'$ and a (flat) complex analytic family $\pi':\sC'\to\sM'$ of (possibly nodal) arithmetic genus $h$ curves with sections $\tau_i^j$ (corresponding to $km$ marked points which stabilize the fibers), equipped with a basepoint $\bullet\in\sM'$ and an isomorphism
    \begin{align}
        i':(C',\{q_i^j\})\xrightarrow{\simeq}(\sC'_{\bullet},\{\sigma_i^j(\bullet)\})
    \end{align}
    such that the resulting map $\sM'\to\overline\clM_{h,km}$ is \'etale at $\bullet$ (i.e. provides a local orbifold chart near $\bullet$).
    \item Let us denote by
    \begin{align}\label{small-univ}
        \clC_g^\text{emb}(X,\clF',A) \hookrightarrow \clM_g^\text{emb}(X,\clF',A)\times X
    \end{align}
    the universal embedded curve. As a subset, it is given by the union of the subsets $\{(y,\Sigma)\}\times\Sigma$ over all $(y,\Sigma)\in\clM_g^\text{emb}(X,\clF',A)$. The first projection of \eqref{small-univ} defines a smooth family of curves (i.e., Riemann surfaces) over $\clM_g^\text{emb}(X,\clF',A)$ and thus, there is a smooth classifying map
    \begin{center}
    \begin{equation}\label{classify}
    \begin{tikzcd}
        \clC_g^\text{emb}(X,\clF',A) \arrow[d]\arrow[r,"\Psi"] & \sC \arrow[d] \\
        \clM_g^\text{emb}(X,\clF',A) \arrow[r,"\Phi"] & \sM
    \end{tikzcd} 
    \end{equation}
    \end{center}
    defined near $(x,C)$ which maps $(y,\Sigma)\in\clM_g^\text{emb}(X,\clF',A)$ to the point of $\sM$ corresponding to $(\Sigma,\Sigma\cap \Delta_1,\ldots,\Sigma\cap \Delta_m)$ and specializes to the isomorphism \eqref{central} over $(x,C)$. 
\end{enumerate}
\begin{remark}
    All the assertions for the remainder of the argument are to be understood to hold only when we are sufficiently close to the basepoints $(x,C)$, $(x,C,\varphi:C'\to C)$, $(x,C'\xrightarrow{\varphi} C\subset X)$, $*$ and $\bullet$. The size of these neighborhoods may shrink from one line to the next.
\end{remark}
\begin{Lemma}[Auxiliary embedding]
    The map
    \begin{align}\label{small-emb}
        \clM^{\normalfont\text{emb}}_g(X,\clF',A)\to V'\times \Delta_1\times\cdots\times \Delta_m
    \end{align}
    defined in a neighborhood of $(x,C)$ by
    \begin{align}
        (y,\Sigma)\mapsto(y,\Sigma\cap \Delta_1,\cdots,\Sigma\cap \Delta_m)    
    \end{align}
    is a smooth embedding near $(x,C)$.
\end{Lemma}
\begin{proof}
    By the inverse function theorem, it will suffice to check that the differential
    \begin{align}
        T_{(x,C)}\clM^{\normalfont\text{emb}}_g(X,\clF',A)\to T_xV'\oplus\bigoplus_{i=1}^m N_{C,p_i}
    \end{align}
    of \eqref{small-emb} at $(x,C)$ is injective. This follows from the exact sequence
    \begin{align}
        0\to\ker D^N_{C,J}\to T_{(x,C)}\clM^{\normalfont\text{emb}}_g(X,\clF',A)\to T_xV' 
    \end{align}
    and the fact that \eqref{ker-eval} is injective.
\end{proof}
\subsection{Thickened moduli space of covers}\label{compactness-cover-sec}
As an intermediate step to tackling \eqref{cover}, we first study the moduli space
\begin{align}\label{cover-aux}
    \overline\clM_h(\sM,k)
\end{align}
which consists of pairs $(s,\psi:\Sigma'\to\sC_s)$, where $s\in\sM$ and $(\Sigma',\psi)$ is in $\overline\clM_h(\sC_s,k)$.\footnote{Recall that $\overline\clM_h(\sC_s,k)$ consists of holomorphic genus $h$ stable maps $\psi:\Sigma'\to\sC_s$ in the class $k[\sC_s]$. Observe that $\sC_s$ is smooth but $\Sigma'$ could be nodal. If $\Sigma'$ has a ghost component with positive genus, then $\overline\clM_h(\sM,k)$ will not be a smooth orbifold near $(s,\psi:\Sigma'\to\sC_s)$. In view of this, we need to thicken the moduli space by considering \eqref{cover-thick-CR}.} We note that
\begin{align}\label{eqn:local-fibre-prod-description}
    \overline\clM_h(\clM_g^\text{emb}(X,\clF',A),k) = \clM^\text{emb}_g(X,\clF',A) \times_\sM\overline\clM_h(\sM,k)
\end{align}
where $\overline\clM_h(\sM,k)\to\sM$ is the natural projection while $\clM_g^\text{emb}(X,\clF',A)\to\sM$ is the classifying map from \eqref{classify}. Since $\pi:\sC\to\sM$ is a local universal family for $\clM_{g,m}$ at $*$, the identification \eqref{eqn:local-fibre-prod-description} follows by observing that the pair $(\Phi_{(y,\Sigma)},\Psi_{(y,\Sigma)})$ of classifying maps from \eqref{classify} is completely determined for all $(y,\Sigma)$ near $(x,C)$ once we specify that it agrees with the isomorphism \eqref{central} for $(y,\Sigma) = (x,C)$ and that it is induced by the stable pointed curve obtained by intersecting $\Sigma$ with $\Delta_1,\ldots,\Delta_m$. Now, let 
\begin{align}\label{branched-linearization}
    D\delbar|_\varphi:C^\infty(C',\varphi^*T_C)\to C^\infty(\tilde C',\Omega^{0,1}_{\tilde C'}\otimes_\bC \varphi^*T_C)
\end{align}
be the linearized Cauchy--Riemann operator\footnote{See the discussion in Remark \ref{rem:nodal-CR} for more details on how to define Cauchy--Riemann operators on nodal curves.} associated to the holomorphic map $\varphi:C'\to C$, where $\tilde C'$ is the normalization of the curve $C'$. Choose $E$ to be a finite dimensional vector space and
\begin{align}\label{perturb-cover-aux}
    \lambda: E\to C^\infty(\sC'^\circ\times\sC,\Omega^{0,1}_{\sC'^\circ/\sM'}\boxtimes_\bC T_{\sC/\sM})
\end{align}
to be a linear map (where $\sC'^\circ\subset\sC'$ is the complement of the nodal points of the fibers of $\pi'$), with $\mathrm{supp}(\lambda)\to\sM'\times\sM$ being a proper map, such that its restriction along the map $C'\hookrightarrow\sC'\times\sC$ given by $i'\times (i\circ\varphi)$ yields an isomorphism
\begin{align}\label{cover-cokernel-iso}
    E = \coker D\delbar|_\varphi.
\end{align}
Now introduce the moduli space
\begin{align}\label{cover-aux-thick}
    \overline\clM_h(\sM,k)_\lambda
\end{align}
consisting of tuples $(s\in\sM,s'\in\sM',\psi:\sC'_{s'}\to\sC_s,e\in E)$ satisfying the $\lambda$-thickened Cauchy--Riemann equation (in the sense of \cite[Definition 9.2.3]{Pardon-VFC}) 
\begin{equation}
    \label{cover-thick-CR}\delbar\psi +  \lambda(e)(\cdot,\psi(\cdot)) = 0
\end{equation}
and incidence conditions
\begin{equation}
    \label{cover-incidence}\psi(\tau_i^j(s')) = \sigma_i(s)  \quad\text{transverse}  
\end{equation}
for each $1\le i\le m$ and $1\le j\le k$. The word \emph{transverse} in \eqref{cover-incidence} indicates that $d\psi|_{\tau_i^j(s')}\ne 0$. 

\begin{remark}\label{remark-cover-thick-CR-linearization}
    The linearization of \eqref{cover-thick-CR} at $(*,\bullet,i\circ\varphi\circ i'^{-1},0)$ is identified, under the isomorphisms $i$ and $i'$, with the operator
    \begin{align}\label{cover-thick-CR-linearization}
        C^\infty(C',\varphi^*T_C)\oplus E\to C^\infty(\tilde C',\Omega^{0,1}_{\tilde C'}\otimes_\bC \varphi^*T_C)
    \end{align}
    which, on the first summand, is the operator $D\delbar|_\varphi$ from \eqref{branched-linearization} and, on the second summand, is the restriction of $\lambda$ along the map $C'\hookrightarrow\sC'\times\sC$ given by $i'\times (i\circ\varphi)$. We emphasize that since the linearization is computed at a point with vanishing $E$-coordinate, no covariant derivatives of the $\lambda$-term (arising from variations in the $\psi$-coordinate) appear in the linearization. In particular, the linearization \emph{completely ignores} the values of $\lambda$ outside the image of $C'\hookrightarrow\sC'\times\sC$. The isomorphism \eqref{cover-cokernel-iso} now shows that \eqref{cover-thick-CR-linearization} is surjective.
\end{remark}

It now follows from \cite[Proposition 9.2.6 \& Appendix B]{Pardon-VFC} that \eqref{cover-aux-thick}, near the point $(*,\bullet,i\circ\varphi\circ i'^{-1},0)$, is a topological manifold of the expected dimension with the projection to $\sM$ being a topological submersion. More explicitly, the dimension of \eqref{cover-aux-thick} relative to $\sM$ is $ = \text{vdim }\overline\clM_h(C,k) + \dim E$ where we have set
\begin{align}\label{cover-vdim}
    \text{vdim }\overline\clM_h(C,k) = 2r
\end{align}
with $r = (2h-2) - k(2g-2)$ being the expected complex dimension of $\overline\clM_h(C,k)$. 

We can now define the $\lambda$-thickened moduli space
\begin{align}\label{cover-thick}
    \overline\clM_h(\clM_g^\text{emb}(X,\clF',A),k)_\lambda = \clM^\text{emb}_g(X,\clF',A)\times_\sM\overline\clM_h(\sM,k)_\lambda
\end{align}
which is also a topological manifold, near the point $(x,C,\varphi:C'\to C,0)$. Moreover, its natural projection to $\clM_g^\text{emb}(X,\clF',A)$ is a topological submersion of relative dimension $2r + \dim E$ near this point.
The loci in \eqref{cover-aux-thick} and \eqref{cover-thick} where the $E$-coordinate vanishes respectively form local \'etale charts for the orbispaces \eqref{cover-aux} and the analogue of \eqref{cover} with $\clF$ replaced by $\clF'$.

Pulling back \eqref{perturb-cover-aux} along \eqref{classify}, we obtain a thickening datum given by a linear map
\begin{align}\label{perturb-cover}
    E\to C^\infty(\sC'^\circ\times\clC_g^\text{emb}(X,\clF',A),\Omega^{0,1}_{\sC'^\circ/\sM'}\boxtimes_\bC T_{\clC_g^\text{emb}(X,\clF',A)/\clM^\text{emb}_g(X,\clF',A)})
\end{align}
which we continue to denote by the same symbol $\lambda$. We record here, for later use, that we have an embedding 
\begin{align}\label{central-branched-cover-domain-emb}
    C'\hookrightarrow\sC'\times\clC_g^\text{emb}(X,\clF',A) 
\end{align}
given by $i'\times ((x,C)\times \varphi)$ where we are identifying $\clC_g^\text{emb}(X,\clF',A)$ with a subset of $\clM_g^\text{emb}(X,\clF',A)\times X$ as in \eqref{small-univ}.  

\subsection{Thickened moduli space of stable maps}\label{compactness-big-sec}
We will now define a new thickening datum given by a linear map
\begin{align}\label{perturb-big}
    \mu: E\to C^\infty(\sC'^\circ\times X\times V'\times \Delta_1\times\cdots\times \Delta_m,\Omega^{0,1}_{\sC'^\circ/\sM'}\boxtimes_{\bC} T_{\clX'/V'})
\end{align}
which will be obtained from the thickening datum $\lambda$ of \eqref{perturb-cover} via the procedure described in the next paragraph. Before that, we first explain the notations in \eqref{perturb-big}. Firstly, $T_{\clX'/V'}$ denotes the relative tangent bundle of the family $\clX' = X\times V'\to V'$ of almost complex manifolds defined by $\clF'$ (i.e. it is the vertical tangent bundle of $X\times V'\to V'$ with fiberwise almost complex structures specified by the family $\clF'$). Secondly, the bundle $\Omega^{0,1}_{\sC'^\circ/\sM'}\boxtimes_{\bC} T_{\clX'/V'}$ is defined on $\sC'^\circ\times (X\times V')$ and the right side of \eqref{perturb-big} is to be interpreted as the space of sections of this bundle pulled back along the coordinate projection $\sC'^\circ\times X\times V'\times \Delta_1\times\cdots\times \Delta_m\to\sC'^\circ\times X\times V'$.

We now explain how $\mu$ is obtained from $\lambda$. Combining the embeddings \eqref{small-univ} and \eqref{small-emb}, we get compatible embeddings of submanifolds
\begin{center}
\begin{equation}\label{extension-emb}
\begin{tikzcd}
    \clC_g^\text{emb}(X,\clF',A) \arrow[r] \arrow[d] & X\times V'\times \Delta_1\times\cdots\times \Delta_m \arrow[d] \\
    \clM_g^\text{emb}(X,\clF',A) \arrow[r] & V'\times \Delta_1\times\cdots\times \Delta_m
\end{tikzcd}
\end{equation}
\end{center}
which are fiberwise pseudo-holomorphic, i.e., their vertical differential induces a fiberwise injective $\mathbb C$-linear map from $T_{\clC_g^\text{emb}(X,\clF',A)/\clM^\text{emb}_g(X,\clF',A)}$ to the pullback of $T_{\clX'/V'}$. Let $\clZ\subset X\times V'\times\Delta_1\times\cdots\times\Delta_m$ denote the (locally closed) submanifold which is the image of $\clC_g^\text{emb}(X,\clF',A)$ under the embedding \eqref{extension-emb}. Via \eqref{extension-emb}, we regard $\lambda$ from \eqref{perturb-cover} as a section of $\Omega^{0,1}_{\sC'^\circ/\sM'}\boxtimes_\bC T_{\clX'/V'}$ defined over the submanifold $\sC'^\circ\times\clZ$. Using a partition of unity argument (concerning extension of a vector bundle section from a submanifold to the ambient manifold), we can now find a linear map $\mu$ as in \eqref{perturb-big} such that the restriction $\mu|_{\sC'^\circ\times\clZ}$ agrees with the linear map $\lambda$ coming from \eqref{perturb-cover} in a neighborhood (in $\sC'^\circ\times\clZ$) of the image of the map 
\begin{align}\label{central-branched-cover-domain-emb-2}
    C'\hookrightarrow\sC'^\circ\times\clZ
\end{align}
given by \eqref{central-branched-cover-domain-emb} followed by $\text{id}_{\sC'}\times\eqref{extension-emb}$.

Next, choose $E'$ to be a finite dimensional vector space and
\begin{align}
    \mu': E'\to C^\infty(\sC'^\circ\times X\times V',\Omega^{0,1}_{\sC'^\circ/\sM'}\boxtimes_{\bC}T_{\clX'/V'})
\end{align}
to be a linear map inducing an isomorphism
\begin{align}\label{normal-obstruction}
    E' = \frac{\coker \varphi^*D^N_{C,J}}{\coker D^N_{C,J}}
\end{align}
when restricted along the map $C'\hookrightarrow\sC'\times X\times V'$ given by $i'\times\varphi\times x$. Here, we view $\coker D^N_{C,J}$ as a linear subspace of $\coker \varphi^*D^N_{C,J}$ via the pullback along $\varphi$.

Now, we may define the $(\mu\oplus\mu')$-thickened moduli space
\begin{align}\label{big-thick}
    \overline\clM_h(X,\clF',kA)_{\mu,\mu'}
\end{align}
to consist of all tuples $(y\in V',s'\in\sM',f:\sC'_{s'}\to X,e\in E,e'\in E')$ such that
\begin{align}
    \label{big-homology}
    f_*[\sC'_{s'}] = kA\\
    \label{big-incidence} r_i^j := f(\tau_i^j(s'))\in \Delta_i \quad\text{transverse intersection} \\
    \label{big-CR}
    \delbar_{\clF'(y)} f + \mu(e)(\cdot,f(\cdot),y,r_1^1,\ldots,r_m^1) + \mu'(e')(\cdot,f(\cdot),y) = 0.
\end{align}
The words \emph{transverse intersection} in \eqref{big-incidence} indicate that the image of $df|_{\tau_i^j(s')}$ and $T_{r_i^j}\Delta_i$ together span $T_{r_i^j}X$. Note that the expected dimension of \eqref{big-thick} is
\begin{align}
    (n-3)(2-2h) + 2c_1(kA) &+\dim(E\oplus E') + \dim V' \\
    = (n-1)(2-2h) + 2k(c_1(A) + 2g - 2) + 2r &+ \dim(E\oplus E') + \dim V'\\
    = \text{ind }\varphi^*D^N_C + \text{vdim }\overline\clM_h(C,k) &+ \dim(E\oplus E') + \dim V'
\end{align}
where $r$ is as in \eqref{cover-vdim}.
Moreover, the locus in \eqref{big-thick} where the $(E\oplus E')$-coordinates vanish gives a local \'etale chart for the orbispace $\overline\clM_h(X,\clF',kA)$ near $(x,\bullet,\varphi\circ i'^{-1},0,0)$.

\begin{Lemma}
    The moduli space $\overline\clM_h(X,\clF',kA)_{\mu,\mu'}$ is a topological manifold of the expected dimension near $(x,\bullet,\varphi\circ i'^{-1},0,0)$.
\end{Lemma}
\begin{proof}
    By \cite[Proposition 9.2.6 \& Appendix B]{Pardon-VFC}, it is enough to check that the relevant linearized operator at $(x,\bullet,\varphi\circ i'^{-1},0,0)$ is surjective. Via $i'$, this is identified with the operator
    \begin{align}\label{big-linearization-explicit}
        C^\infty(C',\varphi^*T_X)\oplus T_xV'\oplus E\oplus E'\to C^\infty(\tilde C',\Omega^{0,1}_{\tilde C'}\otimes_\bC \varphi^*T_X)
    \end{align}
    which can be described explicitly as follows. On the first summand, \eqref{big-linearization-explicit} is given by the linearized Cauchy--Riemann operator $D\delbar_J|_{\varphi}$ on $\varphi^*T_X$. On the second summand, corresponding to variation of the almost complex structure on $X$, the operator \eqref{big-linearization-explicit} is given by the expression
    \begin{align}\label{emb-transverse-def-pullback}
        v'\mapsto\textstyle\frac12(d\clF'|_x\cdot v')\circ d\varphi\circ j_{C'},
    \end{align}
    i.e., the pullback under $\varphi$ of the expression \eqref{emb-transverse-def}. On the third summand, since $\mu$ was constructed to agree with $\lambda$ on the image of \eqref{central-branched-cover-domain-emb-2}, the operator \eqref{big-linearization-explicit} coincides with the second component of \eqref{cover-thick-CR-linearization} followed by the inclusion $\varphi^*T_C\subset\varphi^*T_X$. In more detail, exactly as in Remark \ref{remark-cover-thick-CR-linearization}, it is important to notice that the values of $\mu$ at points not on the image of \eqref{central-branched-cover-domain-emb-2}, and also the covariant derivatives of $\mu$, are irrelevant for computing the linearization, since we are at a point where the $(E\oplus E')$-coordinates vanish. Finally, on the fourth summand, the operator \eqref{big-linearization-explicit} is given by restricting $\mu'$ along the embedding $C'\hookrightarrow\sC'\times X\times V'$ given by $i'^{-1}\times\varphi\times x$ (again, as in Remark \ref{remark-cover-thick-CR-linearization}, only the restriction of $\mu'$ to the image of this embedding is relevant for computing the linearization).
    
    As the linearized Cauchy--Riemann operator $D\delbar_J|_\varphi$ on $\varphi^*T_X$ restricts to $D\delbar|_{\varphi}$ from \eqref{branched-linearization} on $\varphi^*T_C$ and induces $\varphi^*D^N_{C,J}$ on the quotient $\varphi^*N_C$, we can show surjectivity of \eqref{big-linearization-explicit} as follows. By Remark \ref{remark-cover-thick-CR-linearization} and the explicit description of \eqref{big-linearization-explicit} in the previous paragraph, we see that the subspace $C^\infty(C',\varphi^*T_C)\oplus E$ maps onto the subspace $C^\infty(\tilde C',\Omega^{0,1}_{\tilde C'}\otimes_\bC \varphi^*T_C)$ under \eqref{big-linearization-explicit}. Thus, we are left to show that the induced operator
    \begin{align}\label{big-linearization-reduced}
        C^\infty(C',\varphi^*N_C)\oplus T_xV'\oplus E'\to C^\infty(\tilde C',\Omega^{0,1}_{\tilde C'}\otimes_\bC \varphi^*N_C)
    \end{align}
    on quotients is surjective. The operator \eqref{big-linearization-reduced} can be described explicitly as follows. On the first summand, it is given by $\varphi^*D^N_{C,J}$ while on the second and third summands, it is the projection to $\varphi^*N_C$ of the corresponding components of \eqref{big-linearization-explicit}. Thus, we need to show that the map
    \begin{align}\label{big-linearization-reduced-2}
        T_xV'\oplus E'\to\coker \varphi^*D^N_{C,J}
    \end{align}
    induced from \eqref{big-linearization-reduced} is surjective. Restricted to the first summand, \eqref{big-linearization-reduced-2} maps onto the image of $\varphi^*:\coker D^N_{C,J}\to\coker\varphi^*D^N_{C,J}$ by the argument of Lemma \ref{trans-for-emb}. The isomorphism \eqref{normal-obstruction} now concludes the proof.
\end{proof}
\begin{Lemma}[Compatibility of thickenings]\label{compatible-thick}
    The assignment
    \begin{align}
    \begin{aligned}
        \overline\clM_h(\clM_g^{\normalfont\text{emb}}(X,\clF',A),k)_\lambda &\to \overline\clM_h(X,\clF',kA)_{\mu,\mu'}  \\
        \left((y,\Sigma),(s,s',\psi:\sC'_{s'}\to\sC_s,e)\right)&\mapsto(y,s',\Psi_{(y,\Sigma)}^{-1}\circ\psi:\sC'_{s'}\to\Sigma\subset X,e,0)
    \end{aligned}
    \end{align}
    defines a continuous injective map from a neighborhood of $((x,C),(*,\bullet,i\circ\varphi\circ i'^{-1},0))$ to a neighborhood of $(x,\bullet,\varphi\circ i'^{-1},0,0)$. Here $\Psi$ is the classifying map from \eqref{classify}.
\end{Lemma}
\begin{proof}
    We first show that the map is well-defined. Let us assume that $\left((y,\Sigma),(s,s',\psi:\sC'_{s'}\to\sC_s,e)\right)$ is a point in \eqref{cover-thick}. Setting $\{\hat p_i\} = \Delta_i\cap\Sigma$, this means that we have $\Phi(y,\Sigma) = s$, a biholomorphism
    \begin{align}
        \Psi_{(y,\Sigma)}:(\Sigma,\hat p_1,\ldots,\hat p_m)\xrightarrow{\simeq}(\sC_s,\sigma_1(s),\ldots,\sigma_m(s))
    \end{align}
    of stable curves and that the equations \eqref{cover-thick-CR} and \eqref{cover-incidence} are satisfied. We then need to verify that the tuple $(y,s',\Psi_{(y,\Sigma)}^{-1}\circ\psi:\sC'_{s'}\to\Sigma\subset X,e,0)$ satisfies the equations \eqref{big-homology}--\eqref{big-CR}. If $(y,\Sigma)$ is sufficiently close to $(x,C)$, then we have $[\Sigma] = [C] = A$ and thus, \eqref{big-homology} is satisfied. \eqref{cover-incidence} has \eqref{big-incidence} as an immediate consequence. To deduce \eqref{big-CR} from \eqref{cover-thick-CR}, it will suffice to check that
    \begin{align}\label{equality-of-thickenings}
        \Psi_{(y,\Sigma)}^*\left[\lambda(e)(\cdot,\psi(\cdot))\right] = \mu(e)(\cdot,\Psi_{(y,\Sigma)}^{-1}\circ\psi(\cdot),y,\hat p_1,\ldots,\hat p_m)
    \end{align}
    since $\hat p_i = \Psi^{-1}_{(y,\Sigma)}(\sigma_i(s)) =  \Psi_{(y,\Sigma)}^{-1}\circ\psi(\tau_i^1(s'))$ for $1\le i\le m$. By construction, $\mu|_{\sC'^\circ\times\clZ}$ agrees with $\lambda$ via the embedding \eqref{extension-emb} in a neighborhood of the image of \eqref{central-branched-cover-domain-emb-2} and now \eqref{equality-of-thickenings} follows from this immediately. Note that continuity and injectivity are evident and this concludes the proof.
\end{proof}
\begin{Corollary}
    The map from Lemma \ref{compatible-thick} is an inclusion of codimension given by
    \begin{align}\label{codim}
        \dim \ker\varphi^*D^N_{C,J} - \dim\ker D^N_{C,J}\ge 0.
    \end{align}
\end{Corollary}
\begin{proof}
    The dimension of \eqref{cover-thick} is given by
    \begin{align}
        \dim V' + \text{ind }D^N_{C,J} + \text{vdim }\overline\clM_h(C,k) + \dim E
    \end{align}
    while the dimension of \eqref{big-thick} is given by
    \begin{align}
        \dim V' + \text{ind }\varphi^*D^N_{C,J} + \text{vdim }\overline\clM_h(C,k) + \dim E + \dim E'
    \end{align}
    and thus, from \eqref{normal-obstruction}, we find that the codimension of \eqref{cover-thick} in \eqref{big-thick} is precisely given by \eqref{codim}.
\end{proof}
\subsection{Concluding the proof of Theorem \ref{compactness-result}}\label{compactness-conclusion}
We will now prove Theorem \ref{compactness-result} by contradiction. To this end, assume that \eqref{codim} is zero. Then, the map from Lemma \ref{compatible-thick} is an injective continuous map between topological manifolds of the same dimension and thus, by Brouwer's Invariance of Domain (\cite[Theorem 2B.3]{Hatcher}), its image must cover an open neighborhood of $(x,\bullet,\varphi\circ i'^{-1},0,0)$. Restricting this over $V\times 0\times 0\subset V'\times E\times E'$, we conclude that
\begin{align}
    \overline\clM_h(\clM_g^\text{emb}(X,\clF,A),k) = \overline\clM_h(X,\clF,kA)
\end{align}
in a neighborhood of $(x,\varphi:C'\to C\subset X)$. This is a contradiction to the existence of the sequence $(J_\nu,\varphi_\nu:\Sigma'_\nu\to X)$ of simple $J_\nu$-curves of genus $h$ converging to $(J,C'\xrightarrow{\varphi} C\subset X)$ unless $h = g$ and $k = 1$ (in which case $\varphi$ must also be an isomorphism). \qed

\section{Bifurcation analysis}
\label{bifurcation}

We begin by describing the class of bifurcations that we will treat in this article. We restrict throughout this section to (branched) covers with smooth and connected domains. Given any $\delta = (g,d,{\bf b},G,{\bf k},{\bf c},A)$ as in Definition \ref{failure-strata-def}, let $r,h$ and $H\subset G$ be the number of branch points, domain genus and the (conjugacy class of) subgroup corresponding (in the sense of Definition \ref{cover-combinatorial-type}) to the $d$-fold covers in question. We remind the reader that, as in \textsection\ref{super-rigidity-recall}, we will be assuming throughout this section that $(X,\omega)$ is a symplectic Calabi--Yau $3$-fold.
\begin{Definition}[Elementary wall type]\label{elem-wall}
    We say that $\delta$ is an \emph{elementary wall type} if it satisfies conditions {\normalfont(1)--(2)} below.
    \begin{enumerate}[\normalfont(1)]
        \item $d\ge 2$ and either $r=0$, i.e., unbranched covers, or for each $1\le i\le r$, we have $b_i^1 = 2$ and $b_i^j = 1$ for $1<j\le q_i$, i.e., branched covers with only simple \emph{branch points}.
        \item The $G$-representation ${\bf k}$ is faithful, irreducible and of real type. Moreover, if $r = 0$, then we have $\dim_{\bR}{\bf k}^H\le 1$.
    \end{enumerate}
    In this case, the subset $\clW^d_{{\bf b},G}(g,A,{\bf k},{\bf c})\subset\clJ(X,\omega)$ from Definition \ref{wall-strata-def} is called an \emph{elementary wall}.
    If we also have ${\bf k}^H = 0$, then we say that the elementary wall type $\delta$ is \emph{trivial}. 
\end{Definition}

    We now briefly explain the conditions appearing in Definition \ref{elem-wall}. Note that condition {\normalfont(1)} in Definition \ref{elem-wall} is a strengthening of condition {\normalfont(i)} in Lemma \ref{top-strata}: we ask for a unique critical point of order $2$ over each branch point. This has the consequence that any small deformation of a morphism $\varphi:\Sigma'\to\Sigma$ of curves of type $(g,{\bf b})$ is again of the same type.
    On the other hand, condition {\normalfont (2)} in Definition \ref{elem-wall} is a strengthening of condition {\normalfont (ii)} in Lemma \ref{top-strata}. The relevance of the $\dim_\bR{\bf k}^H\le 1$ requirement in condition {\normalfont (2)} is that it prevents the pulled back normal Cauchy--Riemann operator from having a kernel of dimension greater than one. This is important in both the unbranched and branched cases but, in the latter case, it turns out to be a formal consequence of the rest of Definition \ref{elem-wall} (see Lemma \ref{ker-coker-final} below for more details).
    Finally, using \cite[Appendix B]{Wendl-19} and Lemma \ref{operator-decomp}, the non-triviality of $\delta$ is a necessary condition for a bifurcation to occur (see Lemma \ref{trivial-case} below for more details).

Definition \ref{elem-wall} covers many cases, including the following examples. The examples given below are abstract and do not necessarily appear as elementary wall types associated to a given Calabi--Yau $3$-fold.
\begin{Example}\label{exa_bif_1}
    Any degree $2$ morphism $\Sigma'\to\Sigma$ of curves gives an elementary wall type $\delta$, with $G = \bZ/2\bZ$ and $\bf k$ being the unique one-dimensional non-trivial representation of $\bZ/2\bZ$. A result analogous to Theorem \ref{main-bifurcation} below in the case of unbranched double covers of tori in dimension $4$ appears in \cite{Taubes-Gr}.
\end{Example}

\begin{Example}\label{exa_bif_2}
    For any $d\ge 2$, consider a degree $d$ morphism $\Sigma'\to\Sigma$ with distinct simple branch points with generalized automorphism group being the whole symmetric group $S_d$. This gives rise to an elementary wall type with $G = S_d$, $H = S_{d-1}$ and $\bf k$ being the subrepresentation $\{(x_1,\ldots,x_d):\sum_i x_i = 0\}$ of the tautological representation of $S_d$ on $\bR^d$. The representation $\bf k$ is faithful and the following computation of the character of its complexification ${\bf k}_\bC = {\bf k}\otimes_\bR\bC$
    \begin{align}\label{character-comp-1}
        |\chi_{{\bf k}_\bC}|^2 + 1 = |\chi_{\bC^d}|^2 = \frac{1}{|S_d|}\sum_{\sigma\in S_d}\left|\{i\in\{1,\ldots,d\}:\sigma(i) = i\}\right|^2 = 2
    \end{align}
    shows that ${\bf k}_\bC$ is irreducible implying that ${\bf k}$ itself is irreducible and of real type. In the first equality of \eqref{character-comp-1}, we have split the tautological $S_d$-representation $\bC^d$ into the trivial representation and ${\bf k}_\bC$ while for the last equality we have used Burnside's orbit-counting lemma for the diagonal action of $S_d$ on $\{1,\ldots,d\}\times\{1,\ldots,d\}$ which has precisely two orbits.
    
    Note that covers with generalized automorphism group $S_d$ always exist if
    \begin{align}
        2h-2 - d(2g - 2) = r \ge 2d-2
    \end{align}
    where $h = \normalfont\text{genus}(\Sigma')$ and $g = \normalfont\text{genus}(\Sigma)$. Explicitly, if $x_1,\ldots,x_r$ are distinct points in $\Sigma$, then we may define $\Sigma'$ to be the $d$-fold branched cover of $\Sigma$, associated to the $d$-fold unbranched cover of $\dot\Sigma := \Sigma\setminus\{x_1,\ldots,x_r\}$ given by the permutation representation
    \begin{align}
        \rho&:\pi_1(\dot\Sigma) \\ &= \langle a_1,b_1,\ldots,a_g,b_g,c_1,\ldots,c_r : \textstyle\prod_{i=1}^g[a_i,b_i]\prod_{j=1}^r c_j = 1\rangle\to S_d \\
        &\rho(a_i) = \rho(b_i)  = 1 \quad\normalfont\text{for $1\le i\le g$} \\
        &\rho(c_{2j-1}) = \rho(c_{2j}) = (j\, j+1) \quad\normalfont\text{for $1\le j\le d-1$} \\
        &\rho(c_k) = (d-1\,d) \quad\normalfont\text{for $2d-2< k\le r$}.
    \end{align}
\end{Example}
    
\begin{Example}\label{exa_bif_3}
    For any $n\ge 1$, take $d = 2n+1$ and consider a degree $d$ unbranched cover $\Sigma'\to\Sigma$ of a genus $g\ge 2$ curve given by the permutation representation
    \begin{align}
        \rho:\pi_1(\Sigma_g) &= \langle a_1,b_1,\ldots,a_g,b_g : \textstyle\prod_{i=1}^g[a_i,b_i] = 1\rangle\to S_{2n+1}\\
        \rho(a_1) &= \rho(b_1) =  (1\,2)(3\,4)\cdots(2n-1\,2n)\\
        \rho(a_2) &= \rho(b_2) = (2\,3)(4\,5)\cdots(2n\,2n+1)
    \end{align}
    and $\rho(a_i) = \rho(b_i) = 1$ for $i>2$. Note that the generalized automorphism group $\normalfont\text{im }\rho = G\subset S_{2n+1}$ of this cover is isomorphic to the dihedral group
    \begin{align}
        D_{2n+1} = \langle\tau,\theta:\tau^2 = \theta^{2n+1} = 1, \tau\theta\tau^{-1} = \theta^{-1}\rangle
    \end{align}
    with $\tau = (1\,2)(3\,4)\cdots(2n-1\,2n)$ and $\tau\theta = (2\,3)(4\,5)\cdots(2n\,2n+1)$. The subgroup $H$ is given by $\langle\tau\rangle\subset D_{2n+1}$. The standard real $2$-dimensional faithful representation $\bf k$ which realizes $D_{2n+1}$ as the isometry group of a regular $(2n+1)$-gon in $\bR^2$ gives rise to a non-trivial elementary wall type. Indeed, if the complexification ${\bf k}_\bC$ of $\bf k$ splits as the direct sum of $D_{2n+1}$-invariant complex lines, then the relation $\tau\theta\tau^{-1} = \theta^{-1}$ forces that $\langle\theta\rangle$ must act trivially on both, which is absurd. We conclude, as in Example \ref{exa_bif_2}, that $\bf k$ is irreducible and of real type. Also, ${\bf k}^H$ is $1$-dimensional over $\bR$ as it can be identified with an axis of symmetry of the regular $(2n+1)$-gon.
\end{Example}

\begin{Example}\label{exa_bif_4}
    Let $n\ge 2$ be an integer and let $d = 2n$. For integers $g,s\ge 1$, define a degree $d$ branched cover $\Sigma'\to\Sigma$ of a genus $g$ curve with $2s$ distinct simple branch points $\{q_1,\ldots,q_{2s}\}\subset\Sigma$ by the permutation representation
    \begin{align}
        \rho&:\pi_1(\Sigma\setminus\{q_1,\ldots,q_{2s}\})
        \\ &= \langle a_1,b_1,\ldots,a_g,b_g,c_1,\ldots,c_{2s}:\textstyle\prod_{i=1}^g[a_i,b_i]\prod_{j=1}^{2s}c_j = 1\rangle\to S_d \\
        &\rho(a_i) = \rho(b_i) = (1\,2\,\cdots \,n)(n+1\,\cdots\, 2n),\quad\rho(c_j) = (1\,\,n+1)
    \end{align}
    for all $1\le i\le g$ and $1\le j\le 2s$. Note that the generalized automorphism group $\normalfont\text{im }\rho = G\subset S_d$ is isomorphic to $(\bZ/2\bZ)^n\rtimes\bZ/n\bZ$ via
    \begin{align}
        ((0,0,\ldots,0),1)&\mapsto (1\,2\,\cdots \,n)(n+1\,\cdots\, 2n) \\
        ((1,0,\ldots,0),0)&\mapsto (1\,\,n+1)
    \end{align}
    and the subgroup $H\subset G$ is given by $(\bZ/2\bZ)^{n-1}\times\{0\}$. The group $G$ has a faithful representation ${\bf k}=\bR^n$, where $\bZ/n\bZ$ acts by cyclically permuting the coordinates while $(\alpha_1,\ldots,\alpha_n)\in(\bZ/2\bZ)^n$ acts by the matrix $\normalfont\text{diag}[(-1)^{\alpha_1},\ldots,(-1)^{\alpha_n}]$, and this gives rise to a non-trivial elementary wall type. As before, to establish that $\bf k$ is irreducible and of real type, we will show that the complexification ${\bf k}_\bC$ is irreducible. Indeed, we have
    \begin{align}
        |\chi_{{\bf k}_\bC}|^2 = \frac1{n2^n}\sum_{k=0}^n\binom{n}{k}(n-2k)^2 = 1
    \end{align}
    where the first equality holds because all elements in $G\setminus(\bZ/2\bZ)^n$ have zero trace on $\bf k$ and the second equality comes by computing $(z\partial_z)^2(z+z^{-1})^n$ at $z=1$ in two ways: by direct differentiation or by first using the binomial expansion.
\end{Example}

The next two examples are included for completeness and their purpose is to show that the conditions in Definition \ref{elem-wall} are strictly stronger than the conditions of Lemma \ref{top-strata}.

\begin{Example}[Non-elementary, unbranched]\label{non-elem-unbranched}
    Consider the unbranched cover $\Sigma''\to\Sigma$ given by the Galois closure of the unbranched cover $\Sigma'\to\Sigma$ from Example \ref{exa_bif_3}. Let $\bf k$ be the $2$-dimensional representation of the group $G = D_{2n+1}$ described in Example \ref{exa_bif_3}. It is immediate that the conditions of Lemma \ref{top-strata} are satisfied but we do not have an elementary wall type since the subgroup $H\subset G$ is trivial in this case implying that $\dim_\bR{\bf k}^H = 2$.
\end{Example}

\begin{Example}[Non-elementary, branched]\label{non-elem-branched}
    For any $n\ge 2$, take $d = 2n+1$ and four distinct points $x_1,\ldots,x_4$ on $\Sigma = \bP^1$. Consider the $d$-fold branched cover $\Sigma'\to\Sigma$ associated to the unbranched cover of $\dot\Sigma = \Sigma\setminus\{x_1,\ldots,x_4\}$ given by the permutation representation
    \begin{align}
        \rho:\pi_1(\dot\Sigma) &= \langle a_1,\ldots,a_4:\textstyle\prod_{i=1}^4 a_i = 1\rangle\to S_{2n+1} \\
        \rho(a_1) &= \rho(a_2) = (1\,2)(3\,4)\cdots(2n-1\,2n)\\
        \rho(a_3) &= \rho(a_4) = (2\,3)(4\,5)\cdots(2n\,2n+1).
    \end{align}
    Taking the same group $G = D_{2n+1}$ and representation $\bf k$ as in Example \ref{exa_bif_3}, we see that the conditions of Lemma \ref{top-strata} are satisfied but we do not have an elementary wall type since $\Sigma'\to\Sigma$ does not have simple branching.
\end{Example}

Let $J_+,J_-\in\clJ(X,\omega)$ be super-rigid almost complex structures (which will remain fixed for the rest of this section). Define $\clP(J_-,J_+)$ to be the space of smooth paths $\gamma:[-1,1]\to\clJ(X,\omega)$ with $\gamma(\pm1) = J_\pm$. Before giving the statement of our main result on bifurcations, we need some preliminaries which we now turn to.
\begin{Lemma}[Generic transversality to wall strata]\label{baire-transverse}
    There is a Baire subset
    \begin{align}\label{baire-set-transverse}
        \clB'\subset\clP(J_-,J_+)
    \end{align}
    such that any $\gamma\in\clB'$ has the following property. The image $\gamma([-1,1])\subset\clJ(X,\omega)$ avoids $\clW_{\normalfont\text{emb}}$ and $\gamma$ is transverse to each of the maps $\eqref{wall-to-acs}$. In particular, whenever the Fredholm index of \eqref{wall-to-acs} is $\le -2$, the image $\gamma([-1,1])$ avoids the image $\clW^d_{{\bf b},G}(g,A,{\bf k},{\bf c})$ of \eqref{wall-to-acs}.
\end{Lemma}
\begin{proof}
    Note that $J_\pm$ are not contained in the wall $\clW$ (see Definition \ref{wall-strata-def}). Applying the Sard--Smale theorem (countably many times) to suitable universal moduli spaces over $\clP(J_-,J_+)$, see e.g., \cite[Section 2.10]{DW-20}, we obtain a Baire subset $\clB'\subset\clP(J_-,J_+)$ consisting of paths $\gamma$ which are transverse to $\clW_\text{emb}$ as well as to each of the maps \eqref{wall-to-acs}. Since $\clW_\text{emb}\subset\clJ(X,\omega)$ is a subset of codimension $\ge 2$ (as are those $\clW^d_{{\bf b},G}(g,A,{\bf k},{\bf c})$ for which \eqref{wall-to-acs} has Fredholm index $\le -2$), transversality implies that $\gamma([-1,1])$ avoids all of these.
\end{proof}
    Let $\gamma\in\clB'$ be a path from $J_-$ to $J_+$. Suppose that we have $t_0\in(-1,1)$ and $p = (\gamma(t_0),\Sigma,\varphi:\Sigma'\to\Sigma,\iota)\in\clM^d_{{\bf b},G}(\clM^*_g(X,A);{\bf k},{\bf c})$ with $\delta = (g,d,{\bf b},G,{\bf k},{\bf c},A)$ being an elementary wall type. Let $\psi:\tilde\Sigma\to\Sigma$ be a Galois closure of $\varphi$ with $H\subset G$ being such that $\Sigma'=\tilde\Sigma/H$. Introduce the abbreviation
    \begin{align}
        D = D^N_{\Sigma,\gamma(t_0)}.
    \end{align}
    Lemma \ref{baire-transverse} guarantees that $\gamma$ and \eqref{wall-to-acs} are transverse at $(t_0,p)$. When $\delta$ is also assumed to be non-trivial (in the sense of Definition \ref{elem-wall}), this yields a canonical isomorphism
    \begin{align}\label{abstract-wc-iso}
        T_{t_0}\gamma\xrightarrow{\simeq}\normalfont\text{Hom}(\ker\varphi^*D,\coker\varphi^*D)/T_\varphi\clM^d_{{\bf b},G}(\Sigma)
    \end{align}
    between the tangent space of $\gamma$ at $t_0$ and the normal space of the map \eqref{wall-to-acs} at $p$, where we have used Remark \ref{reconciling-formulations} to compute the normal space\footnote{In more detail, Remark \ref{reconciling-formulations} shows that the normal space of the inclusion \eqref{stratification-inclusion} is $\text{Hom}^G(\ker\psi^*D,\coker\psi^*D)$. To obtain the normal space of \eqref{wall-to-acs} from this, we must quotient $\text{Hom}^G(\ker\psi^*D,\coker\psi^*D)$ by the vertical tangent space at $p$ of the projection obtained by the composition $\clM^d_{{\bf b},G}(\clM_g^{*}(X,A))\to\clM_g^{*}(X,A)\to\clJ(X,\omega)$. Since we have $\ker D^N_{\Sigma,\gamma(t_0)} = (\ker\psi^*D)^G = {\bf k}^G = 0$, the second arrow is seen to be a local diffeomorphism at $(\gamma(t_0),\Sigma)$. The vertical tangent space of the first arrow is seen to be $T_\varphi\clM^d_{{\bf b},G}(\Sigma)$. Another version of this argument also appears in Lemma \ref{lem:transverse-wall-lift} below.}
    and Lemma \ref{ker-coker-final} (proved later in this section) to make the natural identification 
    \begin{align}
        \text{Hom}^G(\ker\psi^*D,\coker\psi^*D) = \normalfont\text{Hom}(\ker\varphi^*D,\coker\varphi^*D)
    \end{align} 
    given by restriction to the $H$-fixed subspace.
    Orient the two sides of \eqref{abstract-wc-iso} by identifying $T_{t_0}\gamma$ with $\mathbb R$ in the obvious way, using the complex structure on the orbifold $\clM^d_{{\bf b},G}(\Sigma)$ and using the standard orientation on the determinant line $\det(\varphi^*D)$, which is obtained via the spectral flow when we deform $\varphi^*D$ to a $\mathbb C$-linear Cauchy--Riemann operator.
\begin{Definition}[Sign of wall crossing]\label{wc-sign-def}
    With notation as above, we call the pair $(t_0,p)$ an \emph{elementary wall crossing} along $\gamma$. When $\delta$ is non-trivial, the \emph{sign} of this wall crossing is defined to be the sign of the isomorphism \eqref{abstract-wc-iso} and is denoted by
    \begin{align}
        \normalfont\text{sgn}(\gamma,t_0;p) \in \{-1,+1\}.
    \end{align}
    We call the wall crossing \emph{positive} or \emph{negative} in accordance with its sign.
\end{Definition}
If $\delta = (g,d,{\bf b},G,{\bf k},{\bf c},A)$ is an elementary wall type for which \eqref{wall-to-acs} has Fredholm index $\le -1$, then we have a natural codimension $\ge 1$ subspace
\begin{align}\label{bad-stratum-statement}
    \clM^d_{{\bf b},G}(\clM^*_g(X,A);{\bf k},{\bf c})_\spadesuit \subset \clM^d_{{\bf b},G}(\clM^*_g(X,A);{\bf k},{\bf c})
\end{align}
consisting of points which are ``degenerate". (See Definition \ref{bad-stratum-defined} for details; for the proof of the codimension $\geq 1$ property, see Lemma \ref{main-technical-result}.) Denote its image under \eqref{wall-to-acs} by $\clW^d_{{\bf b},G}(g,A,{\bf k},{\bf c})_\spadesuit$ and note that it is a subset of codimension $\ge 2$ in $\clJ(X,\omega)$.
\begin{Lemma}[Generic non-degeneracy]\label{baire-taylor}
    There is a Baire subset
    \begin{align}\label{baire-set-taylor}
        \clB''\subset\clP(J_-,J_+)
    \end{align}
    with the property that every $\gamma\in\clB''$ has image disjoint from $\clW^d_{{\bf b},G}(g,A,{\bf k},{\bf c})_\spadesuit$ for each elementary wall type $\delta = (g,d,{\bf b},G,{\bf k},{\bf c},A)$ for which the map \eqref{wall-to-acs} has Fredholm index $\le -1$.
\end{Lemma}
\begin{proof}
    This follows by an application of the Sard--Smale theorem (countably many times) to suitable universal moduli spaces over $\clP(J_-,J_+)$. The key observation is that the subsets $\clW^d_{{\bf b},G}(g,A,{\bf k},{\bf c})_\spadesuit$ have codimension $\ge 2$ in $\clJ(X,\omega)$, which follows from Lemma \ref{main-technical-result} as noted above.
\end{proof}
\begin{remark}
    By definition, the subsets $\clB'$ and $\clB''$ are both invariant under smooth reparametrizations of paths. Therefore, to simplify the notation below, we will assume that in Definition \ref{wc-sign-def} we have $t_0 = 0$.
\end{remark}
\begin{Theorem}[Elementary bifurcations] \label{main-bifurcation}
        Let $\gamma\in\clP(J_-,J_+)$ be a path lying in the Baire subset $\clB'\cap\clB''$. Consider a point
        \begin{align}
            p = (J,\Sigma,\varphi:\Sigma'\to\Sigma,\iota)\in\clM^d_{{\bf b},G}(\clM^*_g(X,A);{\bf k},{\bf c})
        \end{align}
        such that $(0,p)$ is an elementary wall crossing.

        We can then find an open neighborhood $\clO$ of $\varphi:\Sigma'\to X$ in the Gromov topology and an $\epsilon>0$ such that the parametric moduli space
        \begin{align}\label{local-emb-moduli}
            \left((-\epsilon,\epsilon)\times\clO\right)\cap\clM^*_h(X,\gamma,dA)\to [-1,1]
        \end{align}
        of smooth \emph{embedded} $\gamma(t)$-curves, with $|t|<\epsilon$, of genus $h$ and class $dA$, lying in $\clO$ has the following properties. The fiber of \eqref{local-emb-moduli} over $t = 0$ is empty and \eqref{local-emb-moduli} is a smooth oriented $1$-dimensional manifold, with the projection to $(-\epsilon,0)\cup(0,\epsilon)$ being a proper submersion. Let $m_\pm$ and $n_\pm$ denote the unsigned and signed number, respectively, of points in the fiber of \eqref{local-emb-moduli} over $t$ for $\pm t>0$.
        \begin{enumerate}[\normalfont(i)]
            \item We have $m_\pm = 0$ if $\delta$ is trivial, or if $r = 0$ and $H\subset G$ is not a maximal subgroup such that ${\bf k}^H\ne 0$.
            \item If we are not in either of the cases from \emph{(i)}, then we have $\Gamma = \normalfont\text{Aut}(\varphi)\subset\bZ/2\bZ$, the estimate $|\Gamma|\cdot|m_\pm|\le 2$ and the identity
            \begin{align}
                n_+ - n_- = -\frac{2\cdot\normalfont\text{sgn}(\gamma,0;p)}{|\Gamma|}.
            \end{align}
            Moreover, if $(t_i,\Sigma'_i\subset X)$ is any sequence of points in \eqref{local-emb-moduli} with $t_i\to 0$, then the Gromov limit of this sequence exists and is given by $\varphi:\Sigma'\to X$.
        \end{enumerate}
\end{Theorem}

\begin{remark}\label{rmk_slag}
Note that for $J = \gamma(0)$, an element $\sigma \in \ker(\varphi^*D^N_{\Sigma,J})$ can be viewed as a $d$-valued holomorphic section of the bundle $N_{\Sigma}$ over the curve $\Sigma$ with singular points given by the branched points of the cover $\varphi: \Sigma' \rightarrow \Sigma$. Theorem \ref{main-bifurcation} asserts that by varying $J$ along a generic path $\gamma$, we can obtain a $1$-parameter family of solutions to the non-linear Cauchy--Riemann equations in $X$ which give rise to embedded $\gamma(t)$-holomorphic curves of class $d [\Sigma]$. We expect the proof of Theorem \ref{main-bifurcation} (which is carried out in the next section via Kuranishi models) to be useful in other similar situations. For instance, it is speculated in \cite[Section 1]{donaldson2019deformations} that a $2$-valued harmonic $1$-form defined over a special Lagrangian submanifold $M$ in a Calabi--Yau manifold $(X, \omega, \Omega)$ should give rise to a $1$-parameter family of special Lagrangian immersions $\iota_{t}: \tilde{M} \rightarrow X$, where $\tilde{M}$ is the branched double cover of $M$ branching along the singular set of the given harmonic $1$-form. The similarity is quite transparent, though the analysis needed to carry out the program in \cite{donaldson2019deformations} is still under development.
\end{remark}

The rest of this section is devoted to the proof of Theorem \ref{main-bifurcation}. To orient the reader, we start with a brief plan of the argument. In \textsection\ref{simplifications}, we deal with the trivial case of the theorem and focus our attention on the non-trivial case for the remainder. In \textsection\ref{subsec_prep} and \textsection\ref{sec_ift}, we use the implicit function theorem to construct a Kuranishi model for the moduli space of interest. The detailed analysis of the Kuranishi model, which is the heart of the proof, is carried out in \textsection\ref{kura-analysis}. The proof is completed in \textsection\ref{mult-kura} and \textsection\ref{simple-kura} (modulo a crucial technical result whose proof is carried out in \textsection\ref{generic-nonzero-taylor}).

Throughout the argument, we will be freely using the notation from the statement of Theorem \ref{main-bifurcation}. In particular, the super-rigid almost complex structures $J_\pm$, the elementary wall type $\delta = (g,d,{\bf b},G,{\bf k},{\bf c},A)$, the path $\gamma\in\clB'\cap\clB''\subset\clP(J_-,J_+)$ and the point $p = (J,\Sigma,\varphi:\Sigma'\to\Sigma,\iota)$ with $J = \gamma(0)$ will remain fixed for the remainder of this section. We will be abbreviating $D^N_{\Sigma,J}$ to $D$ and will sometimes denote the almost complex structure $\gamma(t)$ by the symbol $J_t$.
\subsection{Simplifications}\label{simplifications}

Let $r,h$ and $H\subset G$ be the number of branch points, domain genus and the (conjugacy class of) subgroup corresponding to the $d$-fold cover $\varphi$ (as in Definition \ref{elem-wall}). We will study what happens to the moduli space 
\begin{align}\label{proof-emb-curve-moduli}
    \clM^*_h(X,\gamma,dA)\to[-1,1]    
\end{align}
of embedded genus $h$ curves of class $dA$ (lying in a neighborhood of $\varphi:\Sigma'\to X$ in the Gromov topology) as we go from $t<0$ to $t>0$.

Write $\bf k = \theta^*$ for some irreducible faithful representation $\theta$ of $G$ of real type. At the point $p = (J,\Sigma,\varphi:\Sigma'\to\Sigma,\iota)\in\clM^d_{{\bf b},G}(\clM_g^{*}(X,A);{\bf k},{\bf c})$, with associated Galois closure $\psi:\tilde\Sigma\to\Sigma$, we have $\dim_\bR\ker D^N_{\psi,\theta} = 1$ and
\begin{align}\label{galois-ker}
    \ker \psi^*D = \theta^*\otimes\ker D^N_{\psi,\theta}.
\end{align}
Since $d = [G:H]\ge 2$, it follows that $G$ is non-trivial (note that we study bifurcation analysis only in the multiply covered case since \cite[Section 6]{IP-GV} already treats the corresponding analysis in the somewhere injective case). Taking $G$-invariants in \eqref{galois-ker} implies that $\ker D = 0$. There is therefore a unique embedded $\gamma(t)$-curve $\Sigma_t$ close to $\Sigma$ for each $|t|$ small enough (obtained using the inverse function theorem).

\begin{Lemma}[No bifurcations for trivial $\delta$]\label{trivial-case}
    If $\delta$ is a trivial elementary wall type, then we can find $\epsilon>0$ and a neighborhood $\clO$ of $\varphi:\Sigma'\to X$ in the Gromov topology such that \eqref{proof-emb-curve-moduli} has no points contained in $(-\epsilon,\epsilon)\times\clO$.
\end{Lemma}
\begin{proof}
    Assume to the contrary that we can find a sequence $(t_i,\Sigma'_i\subset X)$ of embedded $\gamma(t_i)$-curves of genus $h$ and class $dA$ converging in the Gromov topology to $(0,\varphi:\Sigma'\to X)$. Now, exactly as in the proof of \cite[Proposition B.1]{Wendl-19}, we may suitably rescale $\Sigma'_i$ in the normal direction to the embedded $\gamma(t_i)$-curve $\Sigma_{t_i}$ of genus $g$ and class $A$ to obtain a nonzero element in $\ker\varphi^*D = (\ker\psi^*D)^H$ in the Gromov limit as $i\to\infty$. Now, \eqref{galois-ker} shows that ${\bf k}^H \ne 0$, which implies that $\delta$ is non-trivial.
\end{proof}

Lemma \ref{trivial-case} completes the proof of Theorem \ref{main-bifurcation} in the case when $\delta$ is trivial. Thus, for the rest of this section, we shall focus exclusively on the case when $\delta$ is non-trivial, i.e., $(\theta^*)^H\ne 0$.
We record below an important consequence of $\delta$'s non-triviality.

\begin{Lemma}\label{ker-coker-final}
    The non-triviality of $\delta$ implies that $p = (J,\Sigma,\varphi:\Sigma'\to\Sigma,\iota)\in\clM^d_{{\bf b},G}(\clM_g^{*}(X,A);{\bf k},{\bf c})$ satisfies the following.
    \begin{enumerate}[\normalfont(i)]
        \item We have $\dim_\bR(\theta^*)^H = 1$.
        \item For each irreducible representation $\rho$ of $G$ over $\bR$ with $\rho\not\simeq\theta$, we have either $\normalfont\text{ind}_\bR D^N_{\psi,\rho} = 0$ or $(\rho^*)^H = 0$.
    \end{enumerate}
    As a consequence, we have the identifications
    \begin{align}
        \label{ker-final}
        \ker\varphi^*D &= (\theta^*)^H\otimes\ker D^N_{\psi,\theta} \\
        \label{coker-final}
        \coker\varphi^*D &= (\theta^*)^H\otimes \coker D^N_{\psi,\theta}.
    \end{align}
    Moreover, the natural map
    \begin{align}\label{eq-to-non-eq-iso}
        \normalfont\text{Hom}^G(\ker\psi^*D,\coker\psi^*D)\xrightarrow{\simeq}\text{Hom}(\ker\varphi^*D,\coker\varphi^*D)
    \end{align}
    given by restriction to $H$-fixed points is an isomorphism.
\end{Lemma}
\begin{proof}
    We use Lemma \ref{operator-decomp} and Remark \ref{invariant-sections} to deduce the quasi-isomorphism
    \begin{align}\label{fixed-point-quasi-iso}
        \bigoplus_i\,(\theta_i^*)^H\otimes_{\bK_i}D^N_{\psi,\theta_i}\to\varphi^*D,
    \end{align}
    which implies that we have the identity
    \begin{align}\label{index-decomp}
        \text{ind}_\bR(\varphi^*D) = \sum_i\dim_\bR(\theta_i^*)^H\cdot\text{ind}_{\bK_i}D^N_{\psi,\theta_i}.
    \end{align}
    Now, a simple computation using the Riemann--Roch formula (combined with the Riemann--Hurwitz formula for $\varphi$ and condition (1) of Definition \ref{elem-wall}) shows that the left side of \eqref{index-decomp} is $-2r$. 
    
    On the other hand, by the first inequality in \cite[Lemma 2.15]{Wendl-19}, we know that each term on the right side of \eqref{index-decomp} is non-positive since $\text{ind}_{\mathbb{R}}D = 0$. Moreover, using the second inequality in \cite[Lemma 2.15]{Wendl-19}, we see that
    $\text{ind}_{\bR}(D^N_{\psi,\theta}) \le -2r$ since $\theta$ is faithful, irreducible, and of real type. Since $\dim_\bR(\theta^*)^H \geq 1$, the above argument implies that for all $\theta_i \neq \theta$ in \eqref{index-decomp}, the corresponding term on the right side vanishes, and the term corresponding to $\theta_i = \theta$ is equal to $-2r$. 
    
    If $r>0$, this immediately gives us (i) and (ii). If $r = 0$, then we get (ii) as above and (i) follows from Definition \ref{elem-wall} and the non-triviality of $\delta$. Now, \eqref{ker-final} and \eqref{coker-final} are formal consequences of statements (i), (ii), and \eqref{fixed-point-quasi-iso}. Finally, for the isomorphism \eqref{eq-to-non-eq-iso}, observe that
    \begin{align}
        \text{Hom}^G(\ker\psi^*D,\coker\psi^*D) 
        &= \text{Hom}(\ker D^N_{\psi,\theta},\coker D^N_{\psi,\theta}) \\
        &= \text{Hom}((\theta^*)^H\otimes\ker D^N_{\psi,\theta},(\theta^*)^H\otimes\coker D^N_{\psi,\theta})\\
        &=\text{Hom}(\ker\varphi^*D,\coker\varphi^*D)
    \end{align}
    where we use the fact that $\text{End}^G(\theta^*) = \bR$, i.e., $\theta^*$ is of real type for the first equality and $\dim_\bR(\theta^*)^H = 1$ for the second equality.
\end{proof}

To summarize the discussions above, we note for future reference that
\begin{align}
    \dim_{\mathbb{R}} \ker D &= 0, \\
    \dim_{\mathbb{R}} \ker \varphi^* D &= 1.
\end{align}
These will be used below without further comment.

\subsection{Preparation}\label{subsec_prep}

Since $\gamma\in\clB'$, we know that it is transverse to the map \eqref{wall-to-acs} at $t=0$ and $p = (J,\Sigma,\varphi:\Sigma'\to\Sigma,\iota)$. Let us clarify the meaning of this transversality statement a little further. Choose $\epsilon>0$ small enough so that for $|t|<\epsilon$ the inverse function theorem may be applied to produce an embedded $\gamma(t)$-curve $\Sigma_t$ (of genus $g$ and class $A$) close to $\Sigma$. Define
\begin{align}\label{1-param-covers}
    \clM^d_{{\bf b},G}(\Sigma,\gamma)\to(-\epsilon,\epsilon)
\end{align}
to be a $1$-parametric form of Definition \ref{cover-combinatorial-type} over $\{\Sigma_t\}_{|t|<\epsilon}$. Explicitly, it consists of pairs $(t,\varphi'':\Sigma''\to\Sigma_t)$ such that $t\in(-\epsilon,\epsilon)$ and $\varphi''$ belongs to $\clM^d_{{\bf b},G}(\Sigma_t)$ as in Definition \ref{cover-combinatorial-type} (where we are suppressing the isomorphism of $G$ with the generalized automorphism group from the notation) and is equipped with the projection $(t,\varphi'')\mapsto t$. Being a smooth orbifold (which is proved exactly as in Lemma \ref{orbifold-of-G-covers}), it has a well-defined tangent space at each point (which may be regarded as representations of the isotropy group at that point). In fact, since $\{\Sigma_t\}_{|t|<\epsilon}$ may be viewed as a $1$-parameter family of complex structures on a single smooth surface, it follows that \eqref{1-param-covers} is a smooth locally trivial bundle of orbifolds over $(-\epsilon,\epsilon)$. Thus, we have a natural short exact sequence
\begin{align}\label{1-param-branched-cover}
    0\to T_\varphi\clM^d_{{\bf b},G}(\Sigma)\to T_{(0,\varphi)}\clM^d_{{\bf b},G}(\Sigma,\gamma)\to\bR_t\to 0
\end{align}
relating the tangent spaces of the fiber, the total space and the base. 

Let $\psi:\tilde\Sigma\to\Sigma$ be a Galois closure of $\varphi$. Define
\begin{align}\label{transverse-wall-crossing}
    T_{(0,\varphi)}\clM^d_{{\bf b},G}(\Sigma,\gamma)\to\text{Hom}^G(\ker \psi^*D,\coker \psi^*D)
\end{align} 
to be the linearized map, in the sense of Definition \ref{linearize-family-op}, associated to the family of $G$-equivariant Cauchy--Riemann operators on $\clM^d_{{\bf b},G}(\Sigma,\gamma)$ given by
\begin{align}\label{G-family-op}
    p' = (t,\varphi'':\Sigma''\to\Sigma_t,\iota_t)\mapsto \psi_{p'}^*(D^N_{\Sigma_t,\gamma(t)})
\end{align}
where $p'\mapsto(\psi_{p'}:\tilde\Sigma''\to\Sigma''\to\Sigma_t)$ is a smooth family of Galois closures extending $\psi$ at $p$, obtained as in the discussion preceding \cite[Example 3.6]{Wendl-19}. Transversality of $\gamma$ and \eqref{wall-to-acs} at $(0,p)$ then has the following consequence.

\begin{Lemma}\label{lem:transverse-wall-lift}
    The map \eqref{transverse-wall-crossing} is an isomorphism.
\end{Lemma}
\begin{proof}
    Using $\ker D = \coker D = 0$ and the inverse function theorem, we may identify a neighborhood of $(0,\varphi)$ in $\clM^d_{{\bf b},G}(\Sigma,\gamma)$ with a neighborhood of $p = (J,\Sigma,\varphi:\Sigma'\to\Sigma,\iota)$ inside the pullback of $\clM^d_{{\bf b},G}(\clM_g^*(X,A))$ under the path $\gamma:(-\epsilon,\epsilon)\to\clJ(X,\omega)$. 
    
     Transversality at $(0,p)$ of the path $\gamma$ and the map \eqref{wall-to-acs} of Fredholm index $-1$ shows, by pullback along $\clM^d_{{\bf b},G}(\clM_g^*(X,A))\to\clJ(X,\omega)$, that the map $\clM^d_{{\bf b},G}(\Sigma,\gamma)\to\clM^d_{{\bf b},G}(\clM_g^*(X,A))$ and the inclusion \eqref{stratification-inclusion} have an isolated transverse intersection at $((0,\varphi),p)$. This gives an isomorphism between the tangent space $T_{(0,\varphi)}\clM^d_{{\bf b},G}(\Sigma,\gamma)$ and the normal space at $p$ of the inclusion \eqref{stratification-inclusion}. The normal space of \eqref{stratification-inclusion} at $p$ is indeed given by the right side of \eqref{transverse-wall-crossing} using Remark \ref{reconciling-formulations}. The description of this isomorphism as the linearized map \eqref{transverse-wall-crossing} from above relies on the proof of Theorem \ref{stratification} or equivalently \cite[Theorem D]{Wendl-19}. In more detail, we use the explicit description (from \cite[Section 3.5]{Wendl-19}) of the (surjective) linearization of the equations defining the inclusion \eqref{stratification-inclusion} and this completes the proof.
\end{proof}

Next we will study the stable map moduli space $\clM_h(X,\gamma,dA)$ locally near $(0,\varphi)$ as an intermediate step to understanding \eqref{local-emb-moduli}. The strategy is to use the implicit function theorem to first produce a Kuranishi model and then explicitly analyze the resulting Kuranishi map. Before proceeding to this, let us establish the notations for the domain and target of the Kuranishi map (which are respectively the kernel and cokernel of the corresponding linearized operator).
\subsubsection{Domain of the Kuranishi model}\label{domain}
To describe the domain explicitly, let us denote the inclusion of $\Sigma$ in $X$ by $i$ and the linearized Cauchy--Riemann operator on $i^*T_X$ (resp. $\varphi^*T_X$) by $\tilde D_i$ (resp. $\tilde D_\varphi = \varphi^*\tilde D_i$). Moreover, write $K = \dot\gamma(0)$. The domain of the Kuranishi map is then an open neighborhood of $0$ in the vector space 
\begin{align}
    T^\text{Zar}_{(0,\varphi)}\clM_h(X,\gamma,dA)   
\end{align}
which consists of ``cocycle" triples $(\xi,y,t)$, where $\xi\in\Omega^0(\Sigma',\varphi^*T_X)$, $y\in\Omega^{0,1}(\Sigma',T_{\Sigma'})$ and $t\in\bR$ are such that 
\begin{align}\label{1-param-linearization}
    \tilde D_\varphi\xi + \textstyle\frac12(J\circ d\varphi\circ y + tK\circ d\varphi\circ j_{\Sigma'}) = 0
\end{align}
modulo the ``coboundary" triples $(\frac12 d\varphi(V),j_{\Sigma'}\cdot\bar\partial V,0)$ where $V\in\Omega^0(\Sigma',T_{\Sigma'})$ and $\bar\partial$ stands for the Cauchy-Riemann operator on $\Omega^0(\Sigma',T_{\Sigma'})$. This vector space contains
\begin{align}T^\text{Zar}_\varphi\clM_h(X,J,dA)\end{align} 
which is the subspace of $T^\text{Zar}_{(0,\varphi)}\clM_h(X,\gamma,dA)$ defined by setting $t=0$. 

Similarly, we can consider the parametrized moduli space of $d$--fold genus $h$ branched covers of $\{ \Sigma_t \}_{|t|<\epsilon}$ denoted by $\clM_h(\Sigma,\gamma,d)$. Its (Zariski) tangent space at $(0,\varphi)$ is denoted by $T_{(0,\varphi)}\clM_h(\Sigma,\gamma,d)$. 
The fiber of $\clM_h(\Sigma,\gamma,d)$ over $t = 0$ is
$\clM_h(\Sigma,d)$ and its tangent space at $\varphi$ is written as $T_\varphi \clM_h(\Sigma,d)$.

\begin{Lemma}\label{lemma:big-diagram}
The vector spaces introduced above fit into the following commutative diagram with exact rows and columns (recall that $D = D^N_{\Sigma,J}$).
\begin{center}
\begin{equation}\label{eqn:big-diagram}
\begin{tikzcd}
    & 0 \arrow[d] & 0 \arrow[d] \\
    0 \arrow[r] & T_\varphi \clM_h(\Sigma,d) \arrow[r] \arrow[d] & T_{(0,\varphi)}\clM_h(\Sigma,\gamma,d) \arrow[r] \arrow[d] & \bR_t \arrow[d, equal] \arrow[r] & 0\\
    0 \arrow[r] & T^{\normalfont\text{Zar}}_\varphi\clM_h(X,J,dA) \arrow[r] \arrow[d] & T^{\normalfont\text{Zar}}_{(0,\varphi)}\clM_h(X,\gamma,dA) \arrow[r] \arrow[d] & \bR_t \arrow[r] & 0\\
    & \ker(\varphi^*D) \arrow[r, equal] \arrow[d] & \ker(\varphi^*D) \arrow[d] \\
    & 0 & 0 
\end{tikzcd}
\end{equation}
\end{center}
\end{Lemma}
\begin{proof}

We begin by explaining the maps appearing in \eqref{eqn:big-diagram}. All the maps between the (Zariski) tangent spaces are given by inclusions. The horizontal maps to $\bR_t$ are induced by the projection to the $t$-coordinate. This implies the first two rows form a commutative diagram. The vertical maps to $\ker(\varphi^*D)$ are defined more indirectly and we explain them later.

As in \eqref{1-param-covers}, the natural projection $\clM_h(\Sigma,\gamma,d)\to(-\epsilon,\epsilon)$ is a smooth locally trivial fiber bundle of orbifolds and its fiber over $t = 0$ is $\clM_h(\Sigma,d)$. Linearizing this projection at $(0,\varphi)$ yields the exactness of the first row. For the second row, observe that the map to $\bR_t$ is surjective since the corresponding map on the first row is surjective and the diagram commutes. The definition of $T^\text{Zar}_\varphi\clM_h(X,J,dA)$ now shows that the second row is also exact. Moreover, the compatible short exact sequences of the first two rows induce a canonical isomorphism
\begin{align}\label{tangent-normal-quotient}
    \frac{T^\text{Zar}_\varphi\clM_h(X,J,dA)}{T_\varphi\clM_h(\Sigma,d)} \xrightarrow{\simeq} \frac{T^\text{Zar}_{(0,\varphi)}\clM_h(X,\gamma,dA)}{T_{(0,\varphi)}\clM_h(\Sigma,\gamma,d)}
\end{align}
of quotients. It will suffice to identify the quotient on the left side of \eqref{tangent-normal-quotient} with $\ker(\varphi^*D)$ to simultaneously define the vertical maps to $\ker(\varphi^*D)$ in \eqref{eqn:big-diagram} and establish commutativity and exactness of the two columns. 

To determine the left side of \eqref{tangent-normal-quotient}, we use the long exact sequence associated to the short exact sequence (of Cauchy--Riemann operators on) $0\to \varphi^*T_{\Sigma}\to\varphi^*T_X\to \varphi^*N_\Sigma\to 0$.
\end{proof}

In view of condition (1) of Definition \ref{elem-wall}, we have an isomorphism
\begin{align}\label{eqn:distinct-branch-pts-iso}
    T_\varphi\clM^d_{{\bf b},G}(\Sigma) \xrightarrow{\simeq} T_\varphi\clM_h(\Sigma,d)
\end{align}
induced by the natural map $\clM^d_{{\bf b},G}(\Sigma)\to\clM_h(\Sigma,d)$. Indeed, recall that any small deformation of the morphism $\varphi:\Sigma'\to\Sigma$ is again of type $(g,{\bf b})$. It follows that the natural map from the short exact sequence \eqref{1-param-branched-cover} to the upper row short exact sequence in \eqref{eqn:big-diagram} is an isomorphism and, in particular, this gives
\begin{align}\label{eqn:distinct-branch-pts-iso-parametric}
    T_{(0,\varphi)}\clM^d_{{\bf b},G}(\Sigma,\gamma) \xrightarrow{\simeq} T_{(0,\varphi)}\clM_h(\Sigma,\gamma,d).
\end{align}
In fact, the following diagram
\begin{center}
\begin{equation}\label{tangent-diagram}
\begin{tikzcd}
    T_{(0,\varphi)}\clM^d_{{\bf b},G}(\Sigma,\gamma) \arrow[r, "\simeq"] \arrow[d, "\simeq"] & \text{Hom}^G(\ker\psi^*D,\coker\psi^*D) \arrow[d, "\simeq"] \\
    T_{(0,\varphi)}\clM_h(\Sigma,\gamma,d) \arrow[r,"{\bf w}_\varphi"] & \text{Hom}(\ker\varphi^*D,\coker\varphi^*D)
\end{tikzcd}
\end{equation}
\end{center}
is commutative and, as explained below, all its arrows are isomorphisms. The left vertical arrow is the isomorphism \eqref{eqn:distinct-branch-pts-iso-parametric}, the right vertical arrow comes by restricting to the $H$-fixed subspace (which is an isomorphism by Lemma \ref{ker-coker-final}). The top and bottom horizontal arrows come by considering the families of operators \eqref{G-family-op} and $(t,\varphi'':\Sigma''\to\Sigma_t)\mapsto\varphi''^*(D^N_{\Sigma_t,\gamma(t)})$ on $\clM^d_{{\bf b},G}(\Sigma,\gamma,d)$ and $\clM_h(\Sigma,\gamma,d)$ respectively. The top horizontal arrow was shown to be an isomorphism in Lemma \ref{lem:transverse-wall-lift} and commutativity of \eqref{tangent-diagram} implies that the bottom horizontal arrow must be an isomorphism as well. For future reference, we will label the bottom horizontal arrow of this commutative diagram as ${\bf w}_\varphi$.

Observe that both sides of the isomorphism ${\bf w}_\varphi$ have natural orientations. Indeed, we have the following.
\begin{enumerate}
    \item $\clM_h(\Sigma,\gamma,d)$ is fibered over the oriented interval $(-\epsilon,\epsilon)$ with complex fibers of complex dimension $r$. Recall that $r = (2h-2) - d(2g-2)$ is the number of branch points of $\varphi$.
    \item An orientation of the vector space
    \begin{align}\label{hom-space-orientation}
        \text{Hom}(\ker\varphi^*D,\coker\varphi^*D) = (\ker\varphi^*D)^*\otimes\coker\varphi^*D
    \end{align}
    is the same as an orientation of its determinant line
    \begin{align}
       (\ker\varphi^*D)^{\otimes 1+2r} \otimes\det(\coker\varphi^*D)
    \end{align}
    where we recall that $\dim\ker\varphi^*D = 1$ and $\dim\coker\varphi^*D = 1 + 2r$.
    Since the tensor square of a $1$-dimensional space is canonically oriented, it follows that we simply need an orientation of the determinant line 
    \begin{align}
        \det(\varphi^*D) = \det(\ker\varphi^* D)\otimes\det(\coker\varphi^* D)^*
    \end{align}
    of the Cauchy--Riemann operator $\varphi^*D$. This can be defined in the usual manner by homotoping $D$ to a $\bC$-linear Cauchy--Riemann operator and using the spectral flow.
\end{enumerate}

Thus, the isomorphism ${\bf w}_\varphi$ has a well-defined sign. Notice that the isomorphism \eqref{abstract-wc-iso} for $t_0=0$, whose sign is referenced in Theorem \ref{main-bifurcation}, can be derived as follows from ${\bf w}_\varphi$ by passing to the quotient $\bR_t$ of $T_{(0,\varphi)}\clM_h(\Sigma,\gamma,d)$, using the short exact sequence 
\begin{align}
    0\to T_\varphi\clM_h(\Sigma,d)\to T_{(0,\varphi)}\clM_h(\Sigma,\gamma,d)\to\bR_t\to 0
\end{align}
and making the identification $T_{(0,\varphi)}\clM^d_{{\bf b},G}(\Sigma,\gamma)\xrightarrow{\simeq}T_{(0,\varphi)}\clM_h(\Sigma,\gamma,d)$ given by \eqref{eqn:distinct-branch-pts-iso-parametric}. In particular, we have
\begin{align}
    \text{sgn}({\bf w}_\varphi) = \text{sgn}(\gamma,0;p).
\end{align}
Without loss of generality, we will assume for the rest of this section that $(0,p)$ is a positive wall-crossing, i.e., $\text{sgn}({\bf w}_\varphi) = +1$.
 This assumption can be achieved by reversing the direction of $\gamma$ if it is not already satisfied.
\subsubsection{Target of the Kuranishi model}\label{target}
The target of the Kuranishi map is the space $E_{(0,\varphi)}\clM_h(X,\gamma,dA)$ of all $\eta\in\Omega^{0,1}(\Sigma',\varphi^*T_X)$ modulo all elements of the kind that appear on the left side of \eqref{1-param-linearization}. In fact, in view of the surjectivity of $T^\text{Zar}_{(0,\varphi)}\clM_h(X,\gamma,dA)\to\bR_t$, which follows from \eqref{eqn:big-diagram}, we can restrict attention to just elements as in the left side of \eqref{1-param-linearization} for which $t = 0$. Using the long exact sequence associated to the short exact sequence (of Cauchy--Riemann operators on) $0\to \varphi^*T_\Sigma\to\varphi^*T_X\to\varphi^*N_\Sigma\to 0$, it now follows that we have a natural identification
\begin{align}\label{eqn:coker-varphi-D}
    E_{(0,\varphi)}\clM_h(X,\gamma,dA) = \coker\varphi^*D.
\end{align}
\subsection{Implicit function theorem}\label{sec_ift}
We will now explain how to construct the Kuranishi model for $\clM_h(X,\gamma,dA)$ near $(0,\varphi)$ using the implicit function theorem on an appropriate map $\Phi:\clX\to\clY$ of Banach spaces (defined on a neighborhood of $0$). Let $m\ge 2$ be an integer. Define $\clY$ to be the space of sections of $\Lambda^{0,1}T^*_{\Sigma'}\otimes_\bC\varphi^*T_X$ of class $L^2_m$ (i.e. $m$ derivatives in $L^2$). Note that $\clY$ has a natural action of
\begin{align}
    \Gamma=\text{Aut}(\varphi) = N_G(H)/H    
\end{align}
We define $\clX = \clX_\text{vf}\oplus\clX_\text{acs}\oplus\bR_t$, where the summands are defined as follows. Choose a finite $\Gamma$-invariant set of points $M\subset\Sigma'$ where $\varphi$ is unramified such that $\Sigma'$ has no automorphisms fixing $M$. For each $[z]\in M/\Gamma$, choose a complement $\Delta_{[z]}\subset T_{\varphi(z)}X$ for $d\varphi_z(T_z\Sigma')$. Define $\clX_\text{vf}$ to be the kernel of the $\Gamma$-equivariant evaluation map
\begin{align}
    L^2_{m+1}(\Sigma',\varphi^*T_X)\to\bigoplus_{z\in M} (T_{\varphi(z)}X)/\Delta_{[z]}.
\end{align}
Define $\clX_\text{acs}$ to be a (finite dimensional) $\Gamma$-invariant complement of the map
\begin{align}
    \Omega^0_M(\Sigma',T_{\Sigma'})\xrightarrow{\bar\partial}\Omega^{0,1}(\Sigma',T_{\Sigma'})
\end{align}
consisting of sections compactly supported in $\Sigma'\setminus M$. Here, $\Omega^0_M(\Sigma',T_{\Sigma'})$ is the space of vector fields on $\Sigma'$ which vanish on $M$.

To define the map $\Phi:\clX\to\clY$, we need to fix the following choices.
\begin{enumerate}
    \item A metric on $X$ with its associated exponential map $\exp$.
    \item A $J$-linear connection $\nabla$ on $T_X$.
    \item A smooth $1$-parameter family of complex vector bundle isomorphisms
    \begin{align}
        I_{X,t}:(T_X,J)\xrightarrow{\simeq}(T_X,J_t)
    \end{align}
    with $I_{X,t} = \text{id} + \frac12JA_{X,t}$, where $A_{X,0} = 0$ and $A_{X,t}$ is $J$-antilinear for all $t\in\bR$ with $|t|<\epsilon$ (after shrinking $\epsilon$ if necessary). Explicitly, we can define the endomorphism $A_{X,t}$ by the formula
    \begin{align}\label{explicit-antilinear-perturb}
        A_{X,t} = -2J(\text{id} + JJ_t)(\text{id}-JJ_t)^{-1} = -2J(J_t - J)(J_t + J)^{-1}.    
    \end{align}
    Note that we have $\dot\gamma(0) = \frac{d}{dt}|_{t=0} J_t= \frac{d}{dt}|_{t=0} A_{X,t} = K$.
    \item A smooth family $y\mapsto j_y$ of almost complex structures  on $\Sigma'$ parametrized by $y\in\clX_\text{acs}$ with $\frac d{d\tau}|_{\tau = 0} j_{\tau y} = y$ and $j_0 = j_{\Sigma'}$ being the given almost complex structure on $\Sigma'$. The concrete choice we shall work with is given by declaring the following to be an isomorphism of complex vector bundles
    \begin{align}
        I_{\Sigma',y} = \text{id} + \textstyle\frac12 j_{\Sigma'}y:(T_{\Sigma'},j_{\Sigma'})\to (T_{\Sigma'},j_y)
    \end{align}
    for $y\in\clX_\text{acs}$ with small $|y|$.
\end{enumerate}

We can now define the map $\Phi:\clX\to\clY$ for small $(\xi,y,t)\in\clX$ as
\begin{align}\label{fredholm-map-wc}
    \Phi(\xi,y,t) = PT_{\nabla}\left[I_{X,t}^{-1}\circ(\bar\partial_{J_t,j_y}\exp_\varphi\xi)\circ I_{\Sigma',y}\right].
\end{align}
Note that $\Phi$ is a smooth map of Banach manifolds (see e.g. the discussion in \cite[page 43]{Mc-Sa} or \cite[Section B.5]{Pardon-VFC}).
Here, the symbol $\exp_\varphi\xi$ denotes the map $z\mapsto\exp_{\varphi(z)}\xi(z)$ of regularity $L^2_{m+1}$ (i.e. $m+1$ derivatives in $L^2$), the symbol $\bar\partial_{J_t,j_y}$ denotes the non-linear Cauchy--Riemann operator 
\begin{align}
u \mapsto \textstyle\frac12(du + J_t \circ du \circ j_y)
\end{align}
and $PT_{\nabla}$ denotes the backward parallel transport along the paths $\tau\mapsto\exp_{\varphi(z)}{\tau\xi(z)}$ induced by the connection $\nabla$. Clearly, $\Phi$ is $\Gamma$-equivariant, $\Phi(0) = 0$ and it is immediate that a neighborhood of $(0,\varphi)\in\clM_h(X,\gamma,dA)$ is identified homeomorphically with a neighborhood of $[0]\in\Phi^{-1}(0)/\Gamma$.

To apply the implicit function theorem to $\Phi$, we first note that the linearization $\tilde D = D\Phi(0):\clX\to\clY$ is given by the following formula (compare \eqref{1-param-linearization} from \textsection\ref{domain})
\begin{align}\label{eqn:tilde-D}
    \tilde{D}(\xi,y,t) = \tilde D_\varphi\xi + \textstyle\frac12(J\circ d\varphi\circ y + tK\circ d\varphi\circ j_{\Sigma'})
\end{align}
where recall (from \textsection\ref{domain}) that $\tilde D_\varphi$ denotes the  linearized Cauchy--Riemann operator on $\varphi^*T_X$ (defined with $j_{\Sigma'}$ and $J$ held fixed).
We can now readily see that we have natural identifications
\begin{align}
    \ker\tilde D &= T^\text{Zar}_{(0,\varphi)}\clM_h(X,\gamma,dA) \\
    \label{eqn:coker-tilde-D}\coker\tilde D &= E_{(0,\varphi)}\clM_h(X,\gamma,dA).
\end{align}
Split the maps $\ker\tilde D\to\clX$ and $\clY\to\coker\tilde D$ by choosing a $\Gamma$-invariant complement $\clX'\subset\clX$ for $\ker\tilde D$ and a $\Gamma$-invariant lift $E\subset\clY$ of $\coker\tilde D$.
\begin{Lemma}\label{IFT}
    There are $\Gamma$-invariant neighborhoods $0\in U\subset\ker\tilde D$ and $0\in V\subset\clX'$ such that $\Phi^{-1}(E)\cap(U\times V)$ is given by the graph of a $\Gamma$-equivariant smooth function $\Psi:U\to \clX'$ satisfying $\Psi(0) = 0$ and $D\Psi(0) = 0$.
\end{Lemma}
\begin{proof}
    Apply the implicit function theorem to the $\Gamma$-equivariant map $\Phi$.
\end{proof}
\begin{Definition}[Kuranishi map]\label{kura-def}
    Define the $\Gamma$-equivariant map $F:U\to E$ by
    \begin{align}
        F(\kappa) = \Phi(\kappa + \Psi(\kappa)).
    \end{align}
    This map $F$ yields the germ at $0$ of a $\Gamma$-equivariant smooth map from $\ker\tilde D$ to $\coker\tilde D$ or equivalently
    \begin{align}
        \normalfont T^\text{Zar}_{(0,\varphi)}\clM_h(X,\gamma,dA)\to E_{(0,\varphi)}\clM_h(X,\gamma,dA)
    \end{align} 
    which we call a \emph{Kuranishi map} for $(0,\varphi)\in\clM_h(X,\gamma,dA)$.
\end{Definition}
\begin{remark}
    For a more general abstract discussion of Kuranishi reduction and some of its properties, see Appendix \ref{kuranishi-appendix}.
\end{remark}

    It is clear from Lemma \ref{IFT} that there is a neighborhood $\clU'$ of $[0]\in F^{-1}(0)/\Gamma$ and a neighborhood $\clU''$ of $(0,\varphi)\in\clM_h(X,\gamma,dA)$ along with a homeomorphism 
    \begin{align}\label{kuranishi-homeo}
        \Xi:\clU'\xrightarrow{\simeq}\clU''    
    \end{align} 
    mapping $[0]$ to $(0,\varphi)$. We may assume, after shrinking $\clU'$, $\clU''$ and $\epsilon$, that we have
    \begin{align}
        \clU' = ((-\epsilon,\epsilon)\times\clO)\cap\clM_h(X,\gamma,dA)
    \end{align}
    for a neighborhood $\clO$ of $\varphi:\Sigma'\to X$ in the Gromov topology (see the discussion following \cite[Definition 9.1.3]{Pardon-VFC} for an explicit description of a basis of neighborhoods).

\subsection{Analysis of the Kuranishi map}\label{kura-analysis}
All functions in this subsection are to be understood to be germs near $0$ of smooth functions which vanish at $0$. 
\begin{Definition}
    Given two finite dimensional vector spaces $W_1,W_2$, two functions $F',F'':W_1\to W_2$ and an integer $l\ge 1$, we write
    \begin{align}
        F' \equiv_l F''
    \end{align}
    to mean that the $i^\text{th}$ derivatives of $F'-F''$ vanish at $0$ for all $0\le i\le l$. Clearly, $\equiv_l$ is an equivalence relation.
\end{Definition}
We will now analyze the Kuranishi map $F$ from Definition \ref{kura-def}. From Lemma \ref{IFT}, it follows that $F(0) = 0$ and $DF(0) = 0$. We fix some notations first. Recalling $\tilde{D}$ from \eqref{eqn:tilde-D}, let 
\begin{align}\label{intro-varsigma}
    \varsigma&:\ker\tilde D\to\ker(\varphi^*D)\\
    \label{intro-t} t&:\ker\tilde D\to\bR_t
\end{align}
be the natural projections (see the exact commutative diagram \eqref{eqn:big-diagram} for more details, recalling that $T^\text{Zar}_{(0,\varphi)}\clM_h(X,\gamma,dA) = \ker\tilde D$). Using Lemma \ref{lemma:big-diagram}, we fix a splitting
\begin{align}\label{splitting}
    \ker\tilde D = T_{(0,\varphi)}\clM_h(\Sigma,\gamma,d)\oplus \ker\varphi^*D = \mathbb R_t\oplus T_\varphi\clM_h(\Sigma,d)\oplus\ker\varphi^*D
\end{align}
compatible with the inclusions $T_{\varphi}\clM_h(\Sigma,d)\subset T_{(0,\varphi)}\clM_h(\Sigma,\gamma,d)\subset\ker\tilde D$ and the projections \eqref{intro-varsigma}, \eqref{intro-t}.

Observe that $F^{-1}(0)\subset\ker\tilde D = T^\text{Zar}_{(0,\varphi)}\clM_h(X,\gamma,dA)$ contains a $\Gamma$-invariant submanifold $Z$ (passing through $0$) corresponding to $\clU'\cap\clM_h(\Sigma,\gamma,d)$ under $\Xi$ with
\begin{align}\label{trivial-solution-tgt}
    T_0Z = T_\varphi\clM_h(\Sigma,\gamma,d)\subset\ker\tilde D.
\end{align}
Choosing $\varsigma':\ker\tilde D\to\ker(\varphi^*D)\simeq\bR$ to be a smooth function such that, near $0$, $Z = (\varsigma')^{-1}(0)$ is a regular level set, we can then find a smooth function (which we call the ``reduced" version of the Kuranishi map $F$) \begin{align}\label{eqn:tilde-D-varphi}
    F_1:\ker\tilde D\to(\ker\varphi^*D)^*\otimes\coker\tilde D = \text{Hom}(\ker\varphi^*D,\coker\varphi^*D)    
\end{align}
such that $F = \langle\varsigma',F_1\rangle$, where 
\begin{align}
    \langle,\rangle:\ker\varphi^*D\otimes \text{Hom}(\ker\varphi^*D,\coker\varphi^*D)\to\coker\varphi^*D
\end{align}
is the natural pairing. In \eqref{eqn:tilde-D-varphi}, we identify $\coker\tilde D$ with $\coker\varphi^*D$ using \eqref{eqn:coker-varphi-D} and \eqref{eqn:coker-tilde-D}. In view of \eqref{trivial-solution-tgt}, $Z$ is given by the graph of a function
\begin{align}\label{eqn:function-f}
    f:T_0Z\to\ker(\varphi^*D)
\end{align}
with $f(0) = 0$ and $Df(0) = 0$. As a result, we may take $\varsigma' = \varsigma - f$ and, in particular, we have the relation
\begin{align}
    \varsigma' \equiv_1 \varsigma.
\end{align}
Since $DF(0) = 0$, it follows that $F_1(0) = 0$. Let $F_1'$ denote the function obtained by restricting $F_1$ to the subspace $T_0Z\subset\ker\tilde D$. We then have the following.

\begin{Lemma}\label{leading-quadratic}
    The relation
    \begin{align}
        F_1' \equiv_1 {\bf w}_\varphi
    \end{align}
    holds, where ${\bf w}_\varphi$ is the isomorphism introduced in diagram \eqref{tangent-diagram}.
\end{Lemma}
\begin{proof}
    Given any $b\in T_0Z$ and $\sigma\in\ker\varphi^*D$, we need to show that
    \begin{align}\label{wall-crossing-term-identity}
        {\bf w}_\varphi(b)\cdot\sigma = \left(\textstyle\frac d{d\lambda}F_1(b_\lambda)\cdot\tilde\sigma\right)|_{\lambda = 0}
    \end{align}
    where $\tilde\sigma\in\ker\tilde D$ is a lift of $\sigma$ with $t(\tilde\sigma) = 0$ (which exists by Lemma \ref{lemma:big-diagram}) and $\lambda\mapsto b_\lambda$ is a path in $Z$ with $b_\lambda(0) = 0$ and $\dot b_\lambda(0) = b$.
    Recall that ${\bf w}_\varphi$ comes from linearizing the family of Cauchy--Riemann operators on $\clM_h(\Sigma,\gamma,d)$ specified immediately after \eqref{tangent-diagram}. Accordingly, we have
    \begin{align}
        \label{eqn:linearization-Dphi}{\bf w}_\varphi(b)\cdot\sigma&\equiv \left(\textstyle\frac d{d\lambda}D\Phi(b_\lambda + \Psi(b_\lambda))\cdot\tilde\sigma\right)|_{\lambda = 0} \pmod{\text{im }\tilde D}\\
        \label{eqn:linearization-Dphi-2}
        &\equiv (D^2\Phi)|_0 ( b, \tilde\sigma) \pmod{\text{im }\tilde D}
    \end{align}
    where $D^2\Phi = D(D\Phi)$ denotes the second derivative of $\Phi$, and we evaluate it at $0$ along the two directions $b$ and $\tilde\sigma$. To justify \eqref{eqn:linearization-Dphi}, we argue as follows. Lemma \ref{IFT} shows that $b_\lambda + \Psi(b_\lambda)\in\clX$ is the representative of the element $(t_\lambda,\varphi_\lambda:\Sigma'_\lambda\to\Sigma_{t_\lambda})\in\clM_h(\Sigma,\gamma,d)$ corresponding to $b_\lambda\in Z$ under $\Xi$. Thus, we see that $D\Phi(b_\lambda + \Psi(b_\lambda))\cdot\tilde\sigma$ represents the application of the linearized Cauchy--Riemann operator at the $\gamma(t_\lambda)$-holomorphic map $\varphi_\lambda$ to (the parallel transport of) $\tilde\sigma$. Since the result of applying the linearized Cauchy--Riemann operator at $\varphi_\lambda$ to $\tilde\sigma$ lifts the result of applying $\varphi_\lambda^*D^N_{\Sigma_{t_\lambda},\gamma(t_\lambda)}$ to $\sigma$, \eqref{eqn:linearization-Dphi} now follows. In \eqref{eqn:linearization-Dphi-2}, we are able to disregard the contribution of $\Psi$ since, by Lemma \ref{IFT}, $\Psi(0) = 0$ and $D\Psi(0) = 0$. On the other hand,
    \begin{align}
        F_1(b_\lambda)\cdot\tilde\sigma = \textstyle\frac \partial{\partial\mu}F(b_\lambda + \mu\tilde\sigma)|_{\mu = 0}
    \end{align}
    for each $\lambda$, by the definition of $F_1$. We now compute
    \begin{align}
        \left(\textstyle\frac d{d\lambda}F_1(b_\lambda)\cdot\tilde\sigma\right)|_{\lambda = 0} &= \textstyle\frac {\partial^2}{\partial\lambda\,\partial\mu}F(b_\lambda + \mu\tilde\sigma)|_{\lambda = \mu = 0}\\
        &\equiv \textstyle\frac {\partial^2}{\partial\lambda\,\partial\mu}\Phi(b_\lambda + \mu\tilde\sigma + \Psi(b_\lambda + \mu\tilde\sigma))|_{\lambda = \mu = 0} \pmod{\text{im }\tilde D} \\
        &\equiv (D^2\Phi)|_0(b,\tilde\sigma) \pmod{\text{im }\tilde D}
    \end{align}
    which proves \eqref{wall-crossing-term-identity}, where we have again used $\Psi(0) = 0$ and $D\Psi(0) = 0$.
\end{proof}
\begin{Corollary}\label{Z'-structure}
    $Z' = F_1^{-1}(0)\subset \ker\tilde D$ is a $\Gamma$-invariant submanifold passing through $0$ and its natural projection to $\ker\varphi^*D$ is a local diffeomorphism at $0$. In other words, $Z'$ is given by the graph of a smooth function
    \begin{align}\label{eqn:function-g}
        g:\ker\varphi^*D\to T_{(0,\varphi)}\clM_h(\Sigma,\gamma,d)
    \end{align}
     such that $g(0) = 0$ and $g$ is $\Gamma$-equivariant. In particular, the submanifolds $Z$ and $Z'$ are transverse at $0\in\ker\tilde D$.
\end{Corollary}
\begin{proof}
    Apply the implicit function theorem to the $\Gamma$-equivariant map $F_1$, noting that $DF_1(0)|_{T_0Z} = {\bf w}_\varphi$ is an isomorphism.
\end{proof}

To continue with the bifurcation analysis, we need to understand the natural projection $Z'\subset\ker\tilde D\to\bR_t$, which is equivalent to understanding the smooth function
\begin{align}\label{tog-intro}
    \ker\varphi^*D\xrightarrow{g} T_{(0,\varphi)}\clM_h(\Sigma,\gamma,d)\xrightarrow{t}\bR_t.
\end{align}
This will be more convenient to do in explicit coordinates which we now introduce. Recall that $\dim_{\mathbb R}\ker\varphi^*D = 1$ and fix an element $0\ne\sigma\in\ker\varphi^*D$ for the rest of this section. Using $\sigma$, we will sometimes regard the map $\varsigma$ from \eqref{intro-varsigma} as an $\bR$-valued linear functional on $\ker\tilde D$ (with $\sigma$ corresponding to $1$). Finally, let ${\bf z}:T_\varphi\clM_h(\Sigma,d)\to\bC^r$ be a $\bC$-linear isomorphism. Thus, the splitting \eqref{splitting} now gives us explicit coordinates
\begin{align}\label{coordinates}
    \ker\tilde D = T^{\text{Zar}}_{(0,\varphi)}\clM_h(X,\gamma,dA)\xrightarrow{(t,{\bf z},\varsigma)}\bR_t\oplus\bC^r_{\bf z}\oplus\bR_\varsigma.
\end{align}
We also obtain (non-canonical) splittings $\ker\tilde D = T_{(0,\varphi)}\clM_h(\Sigma,\gamma,d)\oplus\ker \varphi^*D$ and $T_{(0,\varphi)}\clM_h(\Sigma,\gamma,d) = \bR_t\oplus T_\varphi\clM_h(\Sigma,d)$ which we will use without further comment below.

Orient the $1$-dimensional space $\ker\varphi^*D$ using $\sigma$. Then the natural orientation on $\text{Hom}(\ker\varphi^*D,\coker\varphi^*D)$ gives rise to an orientation on $\coker\varphi^*D$. Using the isomorphism ${\bf w}_\varphi$, the element $\sigma$ and the coordinates \eqref{coordinates}, we can identify
\begin{align}
    \coker\varphi^*D = \bR_t\oplus\bC^r_{\bf z}.
\end{align}
This isomorphism preserves orientation since $\text{sgn}({\bf w}_\varphi) = +1$.
With these choices of coordinates, we can write the Kuranishi map as
\begin{align}\label{kura-std-form}
    F(t,{\bf z},\varsigma) = (\text{id} + A(t,{\bf z},\varsigma))\cdot(\varsigma - f(t,{\bf z})) ((t,{\bf z}) - g(\varsigma))
\end{align}
where $f \equiv_1 0$ is as in \eqref{eqn:function-f}, $g$ is as in \eqref{eqn:function-g} and $A$ is a $(1+2r)\times(1+2r)$ real matrix valued function with $A(0) = 0$. The factorization in \eqref{kura-std-form} comes from the fact that the zeros of $F$ are exactly the points of $Z$ (zeros of the factor $\varsigma - f(t,{\bf z})$) and the points of $Z'$ (zeros of the factor $(t,{\bf z}) - g(\varsigma)$). Moreover, as noted in Corollary \ref{Z'-structure}, the submanifolds $Z,Z'\subset\ker\tilde D$ have a transverse intersection at $0$. Now, recall that the zeros of $F$ represent stable maps in $\clM_h(X,\gamma,dA)$ lying close to $(0,\varphi)$ in the Gromov topology and we are interested in determining which of these zeros correspond to \emph{embedded} curves (all the other zeros must correspond to multiple covers of embedded curves since the path $\gamma$ misses $\clW_\text{emb}$). The points in $Z$ represent multiple covers of the embedded curves $\Sigma_t$ and we are left to determine whether or not the points in $Z'\setminus\{0\}$ represent \emph{embedded} pseudo-holomorphic curves of genus $h$ and class $dA$. This depends on the section $\sigma\in\ker\varphi^*D$ as we explain below.

We begin by recalling that the normal bundle $N_{\Sigma}$ admits an almost complex structure $\hat J^0_\infty$ such that the projection $N_{\Sigma} \rightarrow \Sigma$ is pseudo-holomorphic and the element $\sigma\in\ker\varphi^*D$ can be regarded as a $\hat J^0_\infty$-holomorphic map $\hat\sigma:\Sigma'\to N_\Sigma$ (cf. \cite[Equation (B.1) and Lemma B.4]{Wendl-19}, see also Lemma \ref{CR-op-v-acs}).
\begin{Lemma}\label{simple-v-multiple-cover}
    $\hat\sigma:\Sigma'\to N_\Sigma$ is a multiply covered map if and only if there is a subgroup $H\subsetneq H'\subset G$ with $(\theta^*)^{H'} = (\theta^*)^H\ne 0$.
    Moreover, in this case, 
    \begin{enumerate}[\normalfont(1)]
        \item $\varphi$ is unramified, i.e., $r = 0$ and,
        \item if $\Sigma'\to C\to\Sigma$ is an intermediate cover to which $\sigma$ descends, then any element of $\coker\varphi^*D$ also descends to $C$. In this case, $C\to\Sigma$ also determines a non-trivial elementary wall type.
    \end{enumerate}
\end{Lemma}
\begin{proof}
    Let $\tilde\Sigma\to\Sigma'=\tilde\Sigma/H\xrightarrow{\varphi}\Sigma$ be a Galois closure. 
    If there exists a subgroup $H\subsetneq H'\subset G$ with $(\theta^*)^{H'} = (\theta^*)^H\ne 0$, then $\zeta:\tilde\Sigma/H'\to\Sigma$ corresponds to a non-trivial elementary wall type. Indeed, condition (1) of Definition \ref{elem-wall} holds for $\zeta$ since it does for $\varphi$ and condition (2) follows from $\dim_\bR(\theta^*)^{H'} = \dim_{\bR}(\theta^*)^H = 1$.  Equation \eqref{ker-final} applied to $\varphi$ and $\zeta$ now yields the identification $\ker\zeta^*D = \ker\varphi^*D$ given by pullback along $\Sigma' = \tilde\Sigma/H\to\tilde\Sigma/H'$. Thus, $\hat\sigma$ (or equivalently $\sigma$) descends to $\tilde\Sigma/H'$ and thus, $\hat\sigma$ is a multiply covered map.
    
    Conversely, suppose $\hat\sigma$ is multiply covered. Then, it must descend to $\zeta:\tilde\Sigma/H'\to\Sigma$ for some subgroup $H\subsetneq H'\subset G$. This implies that we have $(\theta^*)^{H'}\ne 0$ and thus, $(\theta^*)^{H'} = (\theta^*)^H$ since the latter is $1$-dimensional by Lemma \ref{ker-coker-final}. This, in particular, shows that $\zeta$ is elementary by the argument of the previous paragraph. It follows that the (co)kernels of $\varphi^*D$ and $\zeta^*D$ are identified, using formulas \eqref{ker-final} and \eqref{coker-final}. This proves the statement about the descent of elements in $\coker\varphi^*D$. It also follows that the indices of $\varphi^*D$ and $\zeta^*D$ are the same, i.e., the number $n_\varphi$ of zeros of $d\varphi$ is the same as the number $n_\zeta$ of zeros of $d\zeta$ (when counted with multiplicity). This cannot happen with $[H':H]>1$, unless $\varphi$ is unramified.
\end{proof}
We now split our discussion into two cases depending on whether $\hat\sigma$ is a simple or multiply covered $\hat J^0_\infty$-holomorphic map.

\subsubsection{\bf Case 1 ($\hat\sigma$ is multiply covered)}\label{mult-kura}

By Lemma \ref{simple-v-multiple-cover} we can now assume that $\varphi$ is unramified and that both $\ker\varphi^*D$ and $\coker\varphi^*D$ are pulled back from some intermediate cover $\Sigma'=\tilde\Sigma/H\to C = \tilde\Sigma/H'\xrightarrow{\zeta}\Sigma$ with $\zeta$ being of degree $d'$ and $C$ being of genus $h'$. Since both $\varphi$ and $\zeta$ determine non-trivial elementary wall types, we can compare the Kuranishi models for $(0,\varphi)$ and $(0,\zeta)$ and we find that they are compatible under the natural pullback maps
\begin{align}
    T^\text{Zar}_{(0,\zeta)}\clM_{h'}(X,\gamma,d'A)&\to T^\text{Zar}_{(0,\varphi)}\clM_h(X,\gamma,dA) \\
    E_{(0,\zeta)}\clM_{h'}(X,\gamma,d'A)&\to    E_{(0,\varphi)}\clM_h(X,\gamma,dA)
\end{align}
both of which are isomorphisms. Thus, every curve close to $(0,\varphi)\in\clM_h(X,\gamma,dA)$ is a $(d/d')$-fold cover of a curve close to $(0,\zeta)\in\clM_{h'}(X,\gamma,d'A)$. Thus, the space $\clM_h^*(X,\gamma,dA)$ is empty near $(0,\varphi)\in\clM_h(X,\gamma,dA)$. This completes the proof of Theorem \ref{main-bifurcation}(i) in view of Lemma \ref{simple-v-multiple-cover}.
\subsubsection{\bf Case 2 ($\hat\sigma$ is simple)}\label{simple-kura}

\begin{Lemma}
    $\Gamma = \normalfont\text{Aut}(\varphi) \subset\bZ/2\bZ$.
\end{Lemma}
\begin{proof}
    It suffices to show that $\Gamma$ acts faithfully on the $1$-dimensional space $(\theta^*)^H$. Suppose to the contrary that there is an element $h'\in N_G(H)\setminus H$ which acts trivially on $(\theta^*)^H$. Letting $H'$ be the subgroup of $N_G(H)$ generated by $H$ and $h'$, we find that $(\theta^*)^{H'} = (\theta^*)^H \ne 0$. By Lemma \ref{simple-v-multiple-cover}, we see that this is a contradiction to the simplicity of the $\hat J^0_\infty$-holomorphic map $\hat\sigma$.
\end{proof}
\begin{Lemma}[Nearby embedded curves]\label{nearby-emb-curves}
    All points in $Z'\setminus\{0\}$ which are sufficiently close to $0\in\ker\tilde D$ are mapped into $\clM_h^*(X,\gamma,dA)$ by the map $\Xi$ from \eqref{kuranishi-homeo}.
\end{Lemma}
\begin{proof}
    Since $\gamma\in\clB'$, it avoids $\clW_\text{emb}$ and it is therefore enough to show that the curves in question are simple. We argue by contradiction. To this end, suppose $(t_n,\varphi_n':\Sigma'_n\to X)$ is a sequence in $\clM_h(X,\gamma,dA)$ converging to $(0,\varphi)$ such that for each $n\ge 1$, we have the following properties.
    \begin{enumerate}
        \item $(t_n,\varphi_n':\Sigma'_n\to X)$ is the image under $\Xi$ of a point $(t_n,\varphi_n,\sigma_n)\in \ker\tilde D$ lying in $Z'\setminus\{0\}$. Here, we are using the splitting \eqref{splitting}.
        \item $\varphi'_n$ is a multiply covered $J_{t_n}$-holomorphic map.
    \end{enumerate}
    The submanifolds $Z$ and $Z'$ of $\ker\tilde D$ have a transverse intersection at $0$ (by Corollary \ref{Z'-structure}) and the points of $Z$ correspond to multiple covers of the embedded $\gamma(t)$-curves $\Sigma_t$ defined in the discussion following \eqref{galois-ker}. Thus, (1) implies that, for $n\gg 1$, the $\varphi'_n$ are not covers of the embedded $\gamma(t_n)$-curves $\Sigma_{t_n}\subset X$. Arguing as in the proof of \cite[Theorem B.1]{Wendl-19}, we conclude that by rescaling the $\varphi'_n$ (and the almost complex structures $J_{t_n}$) suitably in the normal directions to $\Sigma_{t_n}$ (and passing to a subsequence), we obtain a Gromov limit $\hat\sigma':\Sigma'\to N_\Sigma$ which is $\hat J^0_\infty$-holomorphic and corresponds (under \cite[Lemma B.4]{Wendl-19}) to an element $0\ne\sigma'\in\ker(\varphi^*D)$. Since $\dim_{\mathbb R}\ker\varphi^*D = 1$ and $\hat\sigma$ is simple, it follows that $\hat\sigma'$ is also simple. Since being a simple map is an open property, it follows that the rescalings of the $\varphi'_n$ in the convergent subsequence are also simple maps for $n\gg 1$. As a result, for $n\gg 1$, the maps $\varphi'_n:\Sigma'_n\to X$ themselves are also simple $J_{t_n}$-holomorphic maps and this contradicts (2).
\end{proof}

The next statement crucially depends on the technical results of \textsection\ref{generic-nonzero-taylor} and the fact that the path $\gamma$ lies in the Baire set $\clB''$ (by Lemma \ref{baire-taylor}).

\begin{Lemma}\label{taylor-nonva}
The map \eqref{tog-intro} has a non-trivial Taylor expansion at the origin. More precisely, there exists an integer $1\le n\le d$ and a real number $c \ne 0$ with
\begin{align}\label{t-taylor-coeff}
    (t\circ g)(\varsigma) \equiv_n c\varsigma^n.
\end{align}
\end{Lemma}
\begin{proof}
    In the course of the proof, we will be using the definitions and results of \textsection\ref{generic-nonzero-taylor}, to which we refer the reader for details. Examining Definition \ref{bad-stratum-defined}, the expression \eqref{kura-std-form} for the Kuranishi map $F$ and using Remark \ref{changing-infinite-dim-trivializations} to ignore the factor $\text{id} + A(t,{\bf z},\varsigma)$ in \eqref{kura-std-form}, we see that the function $t\circ g$ has a non-trivial $d$-jet at $0$ if and only if the point $p = (J,\Sigma,\varphi:\Sigma'\to\Sigma,\iota)$ does not lie on the degeneracy locus $\clM^d_{{\bf b},G}(\clM^*_g(X,A);{\bf k},{\bf c})_\spadesuit$ from Definition \ref{bad-stratum-defined}. But this is guaranteed by the fact that $\gamma$ belongs to the Baire set $\clB''$ (see Lemma \ref{baire-taylor}). We note that $\clB''$ being a Baire set ultimately rests on the fact (proved in Lemma \ref{main-technical-result} of \textsection\ref{generic-nonzero-taylor} below) that the degeneracy locus is of codimension $\ge 1$ in $\clM^d_{{\bf b},G}(\clM^*_g(X,A);{\bf k},{\bf c})$.
\end{proof}
 
Using Lemma \ref{taylor-nonva}, we will now determine the signed count of the zeros of $F$ in $Z'$ for small $t\ne 0$. We divide into cases based on the sign of $c$ and the parity of $n$. (If $\Gamma = \bZ/2\bZ$, then $n$ must necessarily be even.)
\begin{enumerate}
    \item \emph{$c\ne 0$, $n$ odd}. The Taylor expansion \eqref{t-taylor-coeff} shows that for each small $t_0\ne 0$ we have a unique $\varsigma_0$ such that $t_0 = (t\circ g)(\varsigma_0)$. Moreover, since $\bf z$ is a complex coordinate, the sign of this zero can be determined by looking at the unique nonzero solution $\varsigma = \varsigma_0 = (c^{-1}t_0)^{1/n}$ of the simplified equation
    \begin{align}\label{model}
        \varsigma(t_0 - c\varsigma^n) = 0.
    \end{align}
    We compute of the derivative at this point
    \begin{align}
        \textstyle\frac d{d\varsigma}\,\varsigma(t_0 - c\varsigma^n)|_{\varsigma = \varsigma_0} = t_0 - (n+1)c\varsigma_0^n = -nt_0
    \end{align}
    and find that the sign of the zero is $+1$ for $t<0$ and $-1$ for $t>0$.
    \item \emph{$c>0$, $n$ even}. The Taylor expansion \eqref{t-taylor-coeff} shows that for small $t_0 > 0$, there are exactly two distinct solutions $\varsigma_0\sim\pm(c^{-1}t_0)^{1/n}$ to $t_0 = (t\circ g)(\varsigma_0)$ while there are no solutions for small $t_0 < 0$. Again, we can compute the signs of these zeros, by considering the model equation \eqref{model}, and we find that both of these zeros have the sign $-1$.
    \item \emph{$c<0$, $n$ even}. Arguing exactly as in the previous case, we find that there are exactly two zeros $\varsigma_0\sim\pm(c^{-1}t_0)^{1/n}$ for small $t_0<0$ (both with sign $+1$) and no zeros for small $t_0 > 0$.
\end{enumerate}
Taking the $\Gamma$-symmetry into account and shrinking $\clO$ and $\epsilon$ if necessary, we find that in all cases
the signed number of embedded curves in a sufficiently small Gromov neighborhood $\clO$ of $\varphi:\Sigma'\to\Sigma$ reduces by the amount $2/|\Gamma|\in\bZ$ when $\gamma$ goes from $-\epsilon<t<0$ to $0<t<\epsilon$. Except for the fact that Lemma \ref{baire-taylor} (used in the proof of Lemma \ref{taylor-nonva}) relied on Lemma \ref{main-technical-result} (proved in the following subsection), we have now completed the proof of Theorem \ref{main-bifurcation}. \qed

\begin{figure}
\centering
\includegraphics[width=0.7\textwidth]{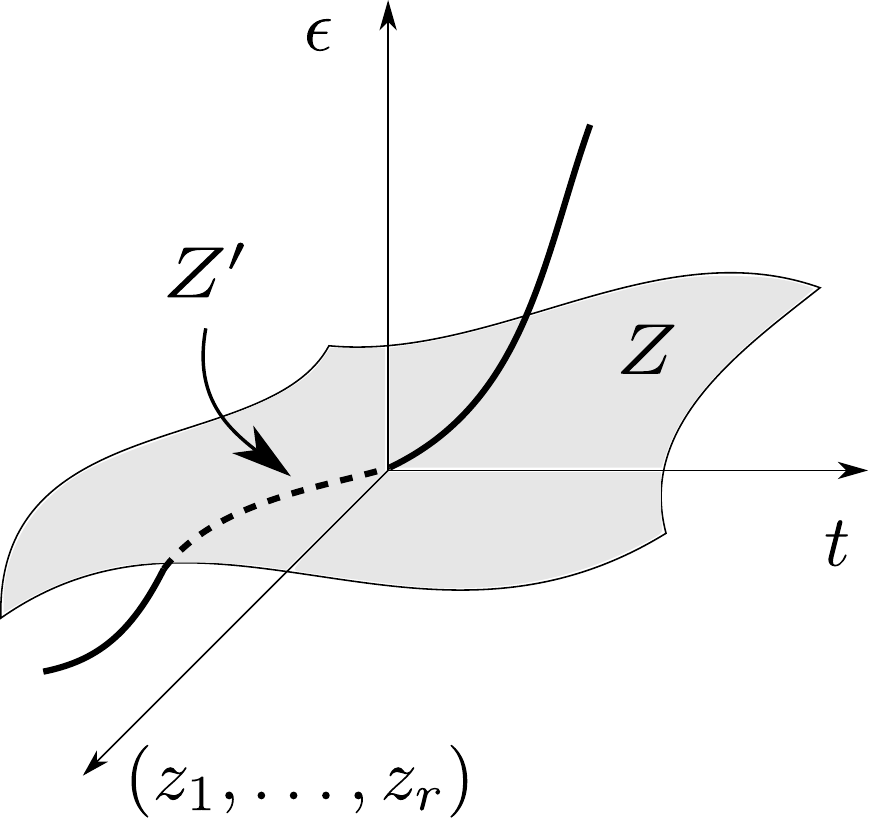}
\caption{A schematic picture of the moduli space (when $n$ is odd); see the discussion below Lemma \ref{taylor-nonva}. }
\end{figure}

\subsubsection{Non-degeneracy of Taylor expansion}
\label{generic-nonzero-taylor}
For any point 
\begin{align}
    \hat p = (\hat J,C,\hat\varphi:C'\to C,\hat\iota)\in \clM^d_{{\bf b},G}(\clM^*_g(X,A);{\bf k},{\bf c}),
\end{align} 
such that $(g,d,{\bf b},G,{\bf k}, {\bf c}, A)$ is a non-trivial elementary wall type (see Definition \ref{elem-wall}), the moduli space $\clM_h(X,\hat J,dA)$ of $\hat J$-holomorphic maps of genus $h$ and class $dA$ (up to isomorphism) can be described locally near $\hat\varphi$, in the sense of the discussion surrounding \eqref{kuranishi-homeo}, using a Kuranishi map
\begin{align}\label{cover-kuranishi}
    F_{\hat p}:T_{\hat\varphi}^\text{Zar}\clM_h(X,\hat J,dA)\to E_{\hat\varphi}\clM_h(X,\hat J,dA)
\end{align}
with the domain and target of \eqref{cover-kuranishi} being defined following \textsection\ref{domain} and \textsection\ref{target}. In particular, we have a short exact sequence and an isomorphism
\begin{align}
    \label{zar-ses} 0\to T_{\hat\varphi}\clM_h(C,d)\to &T^\text{Zar}_{\hat\varphi}\clM_h(X,\hat J, dA)\xrightarrow{\hat\varsigma}\ker\hat\varphi^*\hat D\to 0\\
    \label{ob-iso}&E_{\hat\varphi}\clM_h(X,\hat J,dA) = \coker\hat\varphi^*\hat D
\end{align}
where $\hat D = D^N_{C,\hat J}$ is the normal deformation operator of embedded $\hat J$-curve $C\subset X$. Arguing as in the discussion surrounding \eqref{eqn:tilde-D-varphi}, we express the submanifold $\hat Z\subset T^\text{Zar}_{\hat\varphi}\clM_h(X,\hat J,dA)$ corresponding to $\clM_h(C,d)$ as a regular level set $\hat Z = (\hat\varsigma')^{-1}(0)$ for a smooth map $\hat\varsigma':T^\text{Zar}_{\hat\varphi}\clM_h(X,\hat J,dA)\to\ker\hat\varphi^*\hat D$ satisfying $\hat\varsigma'\equiv_1\hat\varsigma$. Using $\hat\varsigma'$, we can write $F_{\hat p} = \langle\hat\varsigma',F_{\hat p}^\text{red}\rangle$ with
\begin{align}\label{reduced-kuranishi-map}
    F_{\hat p}^\text{red}:T^\text{Zar}_{\hat\varphi}\clM_h(X,\hat J,dA)\to\text{Hom}(\ker\hat\varphi^*\hat D,\coker\hat\varphi^*\hat D)
\end{align}
being the smooth map defined in analogy with the map $F_1$ from \eqref{eqn:tilde-D-varphi}. It follows from Corollary \ref{abstract-kuranishi-leading-order} that the linear map 
\begin{align}\label{reduced-kuranishi-linearization}
    DF^\text{red}_{\hat p}(0):T^\text{Zar}_{\hat\varphi}\clM_h(X,\hat J,dA)\to\text{Hom}(\ker\hat\varphi^*\hat D,\coker\hat\varphi^*\hat D),
\end{align}
which can be determined from $D^2F_{\hat p}(0)$, is well-defined (independent of the choices involved in the construction of the Kuranishi map $F_{\hat p}$ as well as the choice of the defining equation $\hat\varsigma'$ for $\hat Z$ satisfying $\hat\varsigma'\equiv_1\hat\varsigma$). The same argument as in Lemma \ref{leading-quadratic} shows that the restriction of $DF^{\text{red}}_{\hat p}(0)$ to the linear subspace $T_{\hat\varphi}\clM_h(C,d)$ recovers the natural linear map
\begin{align}\label{branched-cover-linearization}
    T_{\hat\varphi}\clM_h(C,d)\to\text{Hom}(\ker\hat\varphi^*\hat D,\coker\hat\varphi^*\hat D).
\end{align}
analogous to the bottom horizontal map in \eqref{tangent-diagram}. Denote the image of \eqref{branched-cover-linearization} by $V_{\hat\varphi}$. Similarly, denote the image of the linear map
\begin{align}\label{branched-cover-leading}
    T_{\hat\varphi}\clM_h(C,d)\otimes\ker\hat\varphi^*\hat D\to\coker\hat\varphi^*\hat D
\end{align}
associated to \eqref{branched-cover-linearization} by $W_{\hat\varphi}$. Note that the injectivity of \eqref{branched-cover-linearization} is equivalent to that of \eqref{branched-cover-leading} which is, in turn, equivalent to $W_{\hat\varphi}\subset \coker\hat\varphi^*\hat D$ being of codimension $1$ or, equivalently, to $V_{\hat\varphi}\subset\text{Hom}(\ker\hat\varphi^*\hat D,\coker\hat\varphi^*\hat D)$ being of codimension $1$. This is because $\dim_{\bR}\ker\hat\varphi^*\hat D = 1$ (since $\hat p$ is non-trivial elementary), the moduli space $\clM_h(C,d)$ has complex dimension $r$, and $\dim_{\bR} \coker\hat\varphi^*\hat D = 2r+1$.

Now, suppose that \eqref{branched-cover-linearization} is injective but that \eqref{reduced-kuranishi-linearization} is not an isomorphism. We further apply Kuranishi reduction to $F_{\hat p}^\text{red}$ from \eqref{reduced-kuranishi-map}. Under these assumptions, the map $DF_{\hat p}^\text{red}(0)$ from \eqref{reduced-kuranishi-linearization} must have $1$-dimensional kernel and cokernel with identifications
\begin{align}
    \label{reduced-kuranishi-ker}\ker DF_{\hat p}^\text{red}(0) &\xrightarrow{\simeq} \ker\hat\varphi^*\hat D \\
    \label{reduced-kuranishi-coker}\coker DF_{\hat p}^\text{red}(0) &= \text{Hom}(\ker\hat\varphi^*\hat D,\coker\hat\varphi^*\hat D)/V_{\hat\varphi}
\end{align}
where the isomorphism \eqref{reduced-kuranishi-ker} is induced by the projection $\hat\varsigma$ from \eqref{zar-ses}. Then \eqref{reduced-kuranishi-ker} and \eqref{reduced-kuranishi-coker} respectively prescribe the domain and target of this further Kuranishi reduction. In fact, \eqref{reduced-kuranishi-ker} provides a splitting of the short exact sequence \eqref{zar-ses} and this splitting gives a preferred choice $\hat\varsigma' = \hat\varsigma - \hat f$ of defining function for $\hat Z$, by writing $\hat Z$ as the graph of a function $\hat f:T_{\hat\varphi}\clM_h(C,d)\to\ker\hat\varphi^*\hat D$ with $\hat f(0) = 0$ and $D\hat f(0) = 0$. This preferred choice of $\hat\varsigma'$ also determines $F_{\hat p}^\text{red}$. Now, applying Kuranishi reduction (see Appendix \ref{kuranishi-appendix} for a review) to the smooth map $F_{\hat p}^\text{red}$ and using the identifications \eqref{reduced-kuranishi-ker}--\eqref{reduced-kuranishi-coker}, we obtain a corresponding Kuranishi map\footnote{To see the relevance of $G_{\hat p}$, recall from Lemma \ref{taylor-nonva} that one is interested in the Taylor expansion of $t\circ g$ at $0$. The main point is then that $G_{\hat p}$ is obtained from $F_{\hat p}^\text{red}$ in analogy with how $t\circ g$ was obtained from $F_1$ (the reduced version of $F$) in \textsection\ref{kura-analysis}.}
\begin{align}\label{fully-reduced-kuranishi-map}
    G_{\hat p}:\ker\hat\varphi^*\hat D\to\text{Hom}(\ker\hat\varphi^*\hat D,\coker\hat\varphi^*\hat D)/V_{\hat\varphi}
\end{align}
with $G_{\hat p}(0) = 0$ and $DG_{\hat p}(0) = 0$. By Corollary \ref{abstract-kuranishi-leading-order}, the leading order Taylor coefficient of $G_{\hat p}$ is well-defined independent of the choices made in constructing $G_{\hat p}$. Explicitly, if all terms of order $<n$ in the Taylor expansion of $G_{\hat p}$ at $0$ vanish for some integer $n>1$, then we may view the order $n$ term in its Taylor expansion invariantly as a linear map
\begin{align}\label{reduced-taylor-coeffs}
    D^nG_{\hat p}(0):(\ker\hat\varphi^*\hat D)^{\otimes n}\to\text{Hom}(\ker\hat\varphi^*\hat D,\coker\hat\varphi^*\hat D)/V_{\hat\varphi}
\end{align}
between $1$-dimensional (real) vector spaces.

Note that the Kuranishi map $G_{\hat p}$ is defined under the assumption that \eqref{reduced-kuranishi-linearization} fails to be surjective. To avoid confusion, we introduce linear maps 
\begin{align}\label{unreduced-taylor-coeffs}
    T^{(n)}_{\hat p}:(\ker\hat\varphi^*\hat D)^{\otimes n}\to(\coker\hat\varphi^*\hat D)/W_{\hat\varphi}
\end{align}
for integers $n>1$ as follows for subsequent discussions. For $n=2$, we define $T^{(2)}_{\hat p}$ as in \eqref{unreduced-taylor-coeffs} via the map $\ker\hat\varphi^*\hat D\to\text{Hom}(\ker\hat\varphi^*\hat D,\coker\hat\varphi^*\hat D)/V_{\hat\varphi}$ induced on quotients by $DF_{\hat p}^\text{red}(0)$ from  \eqref{reduced-kuranishi-linearization}. If $T^{(2)}_{\hat p}$ vanishes, then we are in the setting of the previous paragraph and the smooth map $G_{\hat p}$ as in \eqref{fully-reduced-kuranishi-map} can be constructed. With this in mind, whenever $n>1$ is an integer with $T^{(k)}_{\hat p} = 0$ for $2\le k\le n$, we define $T^{(n+1)}_{\hat p}$ as in \eqref{unreduced-taylor-coeffs} via $D^nG_{\hat p}(0)$.

\begin{remark}
    It is conceivable that we have $T^{(n)}_{\hat p} = 0$ for all $n>1$.
\end{remark}
\begin{Definition}[Degeneracy locus]\label{bad-stratum-defined} 
    Fix a non-trivial elementary wall type $(g,d,{\bf b},G,{\bf k}, {\bf c}, A)$.
    We define the \emph{degeneracy locus}
    \begin{align}\label{bad-stratum-inclusion}
        \clM^d_{{\bf b},G}(\clM^*_g(X,A);{\bf k},{\bf c})_\spadesuit \subset \clM^d_{{\bf b},G}(\clM^*_g(X,A);{\bf k},{\bf c})
    \end{align}
    to be the subspace consisting of all points $\hat p = (\hat J,C,\hat\varphi:C'\to C,\hat\iota)$ such that
    \begin{enumerate}[\normalfont(1)]
        \item $W_{\hat\varphi}$ from \eqref{branched-cover-leading} is a codimension $1$ subspace of $\coker\hat\varphi^*\hat D$ and,
        \item the linear maps $T^{(n)}_{\hat p}$ from \eqref{unreduced-taylor-coeffs} vanish for $2\le n\le d+1$.
    \end{enumerate}
\end{Definition}
Note that condition (1) in Definition \ref{bad-stratum-defined} is open (as it is an injectivity condition) while condition (2) is closed (as it is a vanishing condition). As a result, \eqref{bad-stratum-inclusion} is the inclusion of a locally closed subspace.

\begin{remark}
    The reader is justified in wondering why the quantities $T^{(n)}_{\hat p}$ in Definition \ref{bad-stratum-defined} need to be defined indirectly (using two Kuranishi reductions instead of one). A suggestion for a simpler definition (with the view of proving Lemma \ref{taylor-nonva}) could be to project the Taylor coefficients of $F_{\hat p}$ (or $F_{\hat p}^\text{red}$) suitably to obtain $T^{(n)}_{\hat p}$. However, we can exhibit an example showing that the non-vanishing of (one of) these is not sufficient to obtain Lemma \ref{taylor-nonva} (which is the main intended application of Definition \ref{bad-stratum-defined}) as follows. Using the notation of \eqref{kura-std-form}, suppose that we have $r=2$, ${\bf z} = (z_1,z_2)$ and
    \begin{align}
        F(t,z_1,z_2,\varsigma) =\varsigma(t-\varsigma|z_1|^2+\varsigma^5,z_1-\varsigma^2,z_2+\varsigma^3).
    \end{align}
    Solving $F = {\bf 0}$, we either get $\varsigma = 0$ or we get $z_1 = \varsigma^2$, $z_2 = -\varsigma^3$ and $t = \varsigma|z_1|^2-\varsigma^5 = 0$, violating the conclusion of Lemma \ref{taylor-nonva}. Indeed, in this case, the function $\varsigma^{-1}F(0,z_1,z_2,\varsigma)=(-\varsigma|z_1|^2+\varsigma^5,z_1-\varsigma^2,z_2+\varsigma^3)$ is the counterpart of $F^\text{red}_{\hat p}$ from \eqref{reduced-kuranishi-map} and its Kuranishi reduction $\varsigma\mapsto G_{\hat p}(\varsigma)$ is identically zero (and, in particular, has vanishing Taylor expansion).
\end{remark}

The remainder of this subsection is devoted to showing that \eqref{bad-stratum-inclusion} is an inclusion of codimension $\ge 1$ (Lemma \ref{main-technical-result}).

\begin{Lemma}\label{algebraic-lemma}
    Let $\hat p = (\hat J,C,\hat\varphi:C'\to C,\hat\iota)$ be a point of $\clM^d_{{\bf b},G}(\clM_g^*(X,A);{\bf k},{\bf c})$ such that $(g,d,{\bf b},G,{\bf k}, {\bf c}, A)$ is a non-trivial elementary wall type. Denote by $\hat D = D^N_{C,\hat J}$ the normal deformation operator of $C$. Then, for any 
    \begin{align}
        0&\ne\xi\in\ker\hat\varphi^*\hat D  \\
        0&\ne\xi'\in\ker(\hat\varphi^*\hat D)^\dagger = (\coker\hat\varphi^*\hat D)^*   
    \end{align}
    we can find an integer $1\le n\le d$ and a point $q\in C$ over which $\hat\varphi$ is unramified and the element
    \begin{align}\label{non-linear-petri-analog}
        \sum_{\tilde q\in\varphi^{-1}(q)}(\xi^{\otimes n+1}\otimes\xi')(\tilde q) \in \normalfont\text{Sym}_\bR^{n+1}N_{C,q}\otimes_\bR N^\dagger_{C,q}
    \end{align}
    is nonzero. Here, recall that the bundle $N^\dagger_C$ is defined as in \eqref{eqn:dual-E}.
\end{Lemma}
\begin{proof}
    Choose a Galois closure $\tilde C\to \tilde C/H = C'\to C$. Enlarge the subgroup $H\subset G$ to a maximal subgroup $H'$ such that we have ${\bf k}^{H'} = {\bf k}^H\ne 0$. Applying Lemma \ref{simple-v-multiple-cover}, we see that $\ker\hat\varphi^*\hat D$ and $\coker\hat\varphi^*\hat D$ both descend to the intermediate cover $\hat\varphi':\tilde C/H'\to C$. The element defined by the formula \eqref{non-linear-petri-analog} with $\hat\varphi'$ in place of $\hat\varphi$ differs from \eqref{non-linear-petri-analog} just by a factor of $[H':H]\ge 1$. Thus, there is no loss of generality in assuming that $H' = H$.
    
    Using this assumption, it follows that for all points $q\in C$, outside of a closed and at most countable set, $\xi(\hat\varphi^{-1}(q))\subset N_{C,q}$ has $d$ distinct nonzero elements. Choose such a point $q$. By perturbing it slightly if necessary, we may also assume that $\xi'(\tilde q)\ne 0$ for some $\tilde q\in\hat\varphi^{-1}(q)$. Enumerate the pre-images of $q$ under $\hat\varphi$ as $\tilde q_1,\ldots,\tilde q_d\in C'$. Choose linear functionals
    \begin{align}
        \Lambda&:N_{C,q}\to\bR \\
        \Lambda'&:N^\dagger_{C,q}\to\bR
    \end{align}
    such that the $d$ numbers $x_i = \Lambda(\xi(\tilde q_i))$ are \emph{all distinct and nonzero} and the $d$ numbers $y_j = \Lambda'(\xi'(\tilde q_j))$ are \emph{not all zero}. This implies that the product
    \begin{align}
            \begin{bmatrix}
                x_1^2 & x_2^2 & \cdots & x_d^2 \\
                x_1^3 & x_2^3 & \cdots & x_d^3 \\
                \vdots & \vdots & \ddots & \vdots\\
                x_1^{d+1} & x_2^{d+1} & \cdots & x_d^{d+1}
            \end{bmatrix}
            \begin{bmatrix}
                    y_1 \\ y_2 \\ \vdots \\ y_d
            \end{bmatrix}
    \end{align}
    cannot be zero, since the matrix is invertible (its determinant is $\prod_{i=1}^d x_i^2$ times the Vandermonde determinant associated to the $d$ distinct numbers $x_1,\ldots,x_d$) and the column vector is nonzero. Take $n$ to be the smallest number such that the $n^\text{th}$ entry of this matrix product is nonzero. It then follows that 
    \begin{align}
        \left\langle\text{Sym}^{n+1}_\bR\Lambda\otimes_\bR\Lambda',\sum_{\tilde q\in\varphi^{-1}(q)}(\xi^{\otimes n+1}\otimes\xi')(\tilde q)\right\rangle = \sum_{i=1}^d x_i^{n+1}y_i\ne 0.
    \end{align}
    which shows that this $n$ suffices.
\end{proof}

The next remark will be useful to clarify of the statement of Lemma \ref{analytic-lemma}.

\begin{remark}\label{sections-vanishing-order}
    Let $E\to M$ be a vector bundle on a smooth manifold and let $s$ be a smooth section over $M$ of $E$. Suppose $n\ge 0$ is an integer and $x\in M$ is a point such that, in some (and therefore every) choice of local trivialization of $E$ and local coordinates centered at $x$, the section $s$ is represented by a function with vanishing $n$-jet. Then, $s$ has a well-defined $(n+1)^\text{st}$ derivative $D^{n+1}s|_x\in\text{Hom}(\text{Sym}^{n+1}T_xM,E_x)$. Moreover, a straightforward computation in local coordinates shows that for $v_0,\ldots,v_n\in T_xM$,
    \begin{align}\label{high-derivative-using-connection}
        D^{n+1}s|_x(v_0,\ldots,v_n) = (\nabla_{V_0}\cdots\nabla_{V_n}s)(x)
    \end{align}
    for any choice of connection $\nabla$ on the vector bundle $E$ and choices of vector fields $V_0,\ldots,V_n$ (extending the tangent vectors $v_0,\ldots,v_n$ at $x$). In particular, the right side of \eqref{high-derivative-using-connection} is invariant under permutation of the covariant derivatives $\nabla_{V_i}$. For $n=1$, an explicit way to see this is to use the identity
    \begin{align}\label{switching-two-derivatives}
        \nabla_{V_0}\nabla_{V_1}s = \nabla_{V_1}\nabla_{V_0}s + \nabla_{[V_0,V_1]}s + F_{\nabla}(V_0,V_1)s
    \end{align}
    where $F_\nabla$ is the curvature of $\nabla$. At $x\in M$, the last two terms on the right side of \eqref{switching-two-derivatives} vanish since the $1$-jet of $s$ at $x$ is assumed to vanish.
    
    If we have a submanifold $x\in S\subset M$ at whose points $s$ has vanishing $n$-jet, a computation in local coordinates shows that $D^{n+1}s$ descends to a linear bundle homomorphism $\text{Sym}^{n+1}N_{S/M}\to E|_S$.
\end{remark}

\begin{Lemma}\label{analytic-lemma}
    Let $\hat p$, $\xi$, $\xi'$ and $1\le n\le d$ be as in Lemma \ref{algebraic-lemma}. We can then find an element $\hat K\in T_{\hat J}\clJ(X,\omega) \subset \normalfont\overline{\text{End}}_{\hat J}(T_X)$ such that it and its first $n$ derivatives vanish along $C$ and its $\normalfont(n+1)^\text{st}$ derivative, in the directions normal to $C$, projects to a smooth $\bR$-linear bundle homomorphism
    \begin{align}\label{normal-derivative}
        \normalfont\hat K_{n+1}:\text{Sym}^{n+1}_\bR N_C\to\overline{\text{Hom}}_{\hat J}(T_C,N_C) = \Lambda^{0,1}T^*_C\otimes_\bC N_C
    \end{align}
    satisfying
    \begin{align}\label{normal-taylor-variation}
        \int_{C'}\langle\hat\varphi^*(i\hat K_{n+1}),\xi^{\otimes n+1}\otimes\xi'\rangle \ne 0
    \end{align}
    where $i$ denotes multiplication by $\sqrt{-1}$ on the fibers of $N_C$.
\end{Lemma}
\begin{proof}
    We first construct a bundle homomorphism $L$ of the type \eqref{normal-derivative} satisfying \eqref{normal-taylor-variation} with $L$ in place of $\hat K_{n+1}$. Choose a smooth linear bundle homomorphism
    \begin{align}
        L':\text{Sym}_\bR^{n+1}N_C\to\Lambda^{0,1}T^*_C\otimes_\bC N_C
    \end{align}
    which has positive pairing with \eqref{non-linear-petri-analog} at the point $q\in C$ yielded by Lemma \ref{algebraic-lemma}. Multiplying $L'$ by a non-negative cutoff function on $C$ supported near $q$, we can assume that its pairing with \eqref{non-linear-petri-analog} is $\ge 0$ at all points of $C$ and $>0$ at $q$. Taking $L = iL'$ completes the construction. We refer the reader to \cite[Section 3.2]{Mc-Sa} for comparison and further details.
    
    We are left to find $\hat K\in T_{\hat J}\clJ(X,\omega)$ such that it and its first $n$ derivatives vanish along $C$ and $L = \hat K_{n+1}$. To begin, note that $T_{\hat J}\clJ(X,\omega)$ consists of smooth sections of a vector sub-bundle $E\subset\overline{\text{End}}_{\hat J}(T_X)$. The argument from \cite[Lemma 5.27]{Wendl-19} shows that the natural restriction map
    \begin{align}
        E|_C\to\overline{\text{Hom}}_{\hat J}(T_C,N_C) = \Lambda^{0,1}T^*_C\otimes_\bC N_C
    \end{align} 
    is surjective. Thus, we may lift $L$ to a smooth linear bundle homomorphism $\tilde L:\text{Sym}^{n+1}_\bR N_C\to E|_C$. 
    
    Next, choose a tubular neighborhood $\nu:N_C\hookrightarrow X$ of $C$ and an isomorphism $\pi^*(E|_C) = \nu^*E$ of vector bundles on $N_C$ which is the identity on the zero section. Here, $\pi:N_C\to C$ is the vector bundle projection. Viewing $\tilde L$ as a fiberwise (real, homogeneous degree $n+1$) polynomial map $N_C\to E|_C$, we obtain a global smooth section of $\pi^*(E|_C) = \nu^*E$. Multiplying this section of $\nu^*E$ by a cutoff function of compact support which is $\equiv 1$ near the zero section yields a section $\hat K$ of $E$ on $X$. By construction, $\hat K$ and its first $n$ derivatives vanish on $C$. This shows that $\hat K$ has a well-defined $(n+1)^\text{st}$ normal derivative at $C$ (by Remark \ref{sections-vanishing-order}) which, by construction, coincides with $\tilde L$. It now follows that $L = \hat K_{n+1}$.
\end{proof}

Now, fix any point 
\begin{align}\label{fixed-bad-point}
    \hat p = (\hat J,C,\hat\varphi:C'\to C,\hat\iota)\in\clM^d_{{\bf b},G}(\clM^*_g(X,A);{\bf k},{\bf c})_\spadesuit.
\end{align}
Having fixed $\hat p$ and setting $\hat D = D^N_{C,\hat J}$, pick a section $0\ne\xi\in\ker\hat\varphi^*\hat D$ (unique up to scaling). By condition (1) of Definition \ref{bad-stratum-defined}, we can also pick a section, unique up to scaling, $0\ne\xi'\in\ker(\hat\varphi^*\hat D)^\dagger = (\coker\hat\varphi^*\hat D)^*$ which annihilates the image $W_{\hat\varphi}$  of \eqref{branched-cover-leading}. Lemmas \ref{algebraic-lemma} and \ref{analytic-lemma} then produce $1\le n\le d$ and $\hat K$ satisfying the conditions of Lemma \ref{analytic-lemma}. We can now find a smooth path $\hat\gamma:[-1,1]\to\clJ(X,\omega)$, with $\hat J_s:=\hat\gamma(s)$, satisfying the following properties.
\begin{enumerate}
    \item $\hat J_0 = \hat J$ and $\frac{d}{ds} (\hat J_s)|_{s=0} = \hat K$.
    \item The path $\hat\gamma$ fixes the $n$-jet of $\hat J$ along $C$. More precisely, we have $|\hat J_s - \hat J| = O(|\nu_C|^{n+1})$ for all $s$, where $\nu_C(\cdot) = \text{dist}(C,\cdot)$ denotes the distance to $C$ induced by the given Riemannian metric on $X$.
\end{enumerate}
Following the discussion of \textsection\ref{sec_ift}, with $\hat\gamma$ (resp. $\hat\varphi:C'\to C$) taking the place of $\gamma$ (resp. $\varphi:\Sigma'\to\Sigma$), we choose auxiliary data (recall $\{\Delta_{[z]}\}_{[z]\in M/\Gamma}$ and (1)--(4) from \textsection\ref{sec_ift}) to define Banach spaces $\hat\clX_\text{vf}\oplus\hat\clX_\text{acs}\oplus\bR_s = \hat\clX$, $\hat\clY$ and a smooth Fredholm map $\hat\Phi:\hat\clX\to\hat\clY$ (defined near $0$) by the formula
\begin{align}
    \hat\Phi(\hat\xi,\hat y,s) = PT_{\nabla}\left[\hat I_{X,s}^{-1}\circ(\bar\partial_{\hat J_s,j_{\hat y}}\exp_{\hat\varphi}\hat\xi)\circ\hat I_{C',\hat y}\right]
\end{align}
analogous to \eqref{fredholm-map-wc}. Here, we have used the notation 
\begin{align}\label{iso-along-Js-path}
    \hat I_{X,s} = \text{id} + \textstyle\frac12\hat J\hat A_{X,s}:(T_X,\hat J)\xrightarrow{\simeq}(T_X,\hat J_s)
\end{align}
to denote the family of $\bC$-linear isomorphisms with the tensor $\hat A_{X,s}$ defined by the explicit formula
\begin{align}\label{explicit-antlinear-perturb-hat}
    \hat A_{X,s} = -2\hat J(\text{id} + \hat J\hat J_s)(\text{id}-\hat J\hat J_s)^{-1} = -2\hat J(\hat J_s - \hat J)(\hat J_s + \hat J)^{-1}    
\end{align}
analogous to \eqref{explicit-antilinear-perturb}. It follows from \eqref{explicit-antlinear-perturb-hat} that $\hat A_{X,s}$ and its first $n$ derivatives vanish along $C$. Write $\hat\clX_0 = \hat\clX_\text{vf}\oplus\hat\clX_\text{acs}$ and, for each $s$, define the smooth Fredholm map
\begin{align}
    \hat\Phi_s:\hat\clX_0\to\hat\clY
\end{align}
by restricting $\hat\Phi$ to $\hat\clX_0\times\{s\}$, i.e., $\hat\Phi_s(\hat\xi,\hat y) = \hat\Phi(\hat\xi,\hat y,s)$. The next observation makes precise the intuitive claim that the $n$-jet of $\hat\Phi_0$ at $0$ depends only on the $n$-jet of $\hat J$ along $C\subset X$. In particular, it shows that we have
\begin{align}\label{path-on-wall}
    \hat p(s) = (\hat J_s,C,\hat\varphi:C'\to C,\hat\iota)\in\clM^d_{{\bf b},G}(\clM^*_g(X,A);{\bf k},{\bf c})
\end{align}
for each $s$, which defines a lift of the path $\hat\gamma$ along the map \eqref{wall-to-acs}.

\begin{Lemma}\label{truncation-of-fred-map}
    For each $s$, the smooth map 
    \begin{align}
        P_s = \hat\Phi_s - \hat\Phi_0:\hat\clX_0\to\hat\clY.    
    \end{align}
    satisfies $D^k P_s(0) = 0$ for $0\le k\le n$. For any $\tilde\xi\in\hat\clX_0$, let $\hat\xi^N$ denote its image under the projection $\hat\clX_0\to\hat\clX_{\normalfont\text{vf}}$ followed by the projection $\hat\varphi^*T_X\to\hat\varphi^*N_C$. Then, we have the identity
    \begin{align}\label{first-variation-fred-derivative}
        \Pi^N\circ\textstyle\frac{d}{ds}|_{s=0}(D^{n+1}P_s|_0(\tilde\xi,\ldots,\tilde\xi)) = -\frac12\hat\varphi^*(i\hat K_{n+1})(\hat\xi^N,\ldots,\hat\xi^N),
    \end{align}
    where $\Pi^N$ is the map on $\hat\clY$ induced by the projection $\hat\varphi^*T_X\to\hat\varphi^*N_C$.
\end{Lemma}
\begin{proof}
    Let $\tilde\xi = (\hat\xi,\hat y)\in\hat\clX_0$ be given. It suffices to show that the function
    \begin{align}
        \phi_s(\tau) = \hat\Phi_s(\tau\hat\xi,\tau\hat y)
    \end{align}
    of (small) $\tau\in\bR$ has the properties that $D^k\phi_s(0)$ for $0\le k\le n$ is independent of $s$ and that $\Pi^N\circ\frac{d}{ds}|_{s=0}D^{n+1}\phi_s(0)$ is given by the right side of \eqref{first-variation-fred-derivative}. The key idea to compute the derivatives of $\phi_s$ explicitly is to re-interpret them as covariant derivatives. For this purpose, it is convenient to assume that $\hat\xi$ is smooth (the case of general $\hat\xi$ follows by the density of smooth sections in $\hat\clX_\text{vf}$). In the argument that follows, $s$ will always range over $(-1,1)$, while $\tau$ will range over a small interval centered at $0$ (the exact size of which is irrelevant). For notational convenience, we just write $s,\tau\in\bR$ without repeatedly specifying the intervals over which $s,\tau$ range.
    
    Consider the $\bC$-vector bundle on $X\times\bR_s$ given by (the pullback of) $T_X$, where we equip the copy of $T_X$ over $X\times\{s\}$ with the almost complex structure $\hat J_s$. Using the isomorphisms $\hat I_{X,s}$ from \eqref{iso-along-Js-path}, we can regard this bundle as the pullback of $(T_X,\hat J)$ along the projection $X\times\bR_s\to X$ and this endows it with a $\bC$-linear connection by pulling back the already fixed $\bC$-linear connection $\nabla$ on $(T_X,\hat J)$. In particular, contracting with $\partial/\partial s$ annihilates the curvature form of this pullback connection. Similarly, choose a $j_0$-linear connection on $T_{C'}$ and use the isomorphisms $\hat I_{C',\hat y}$ to endow the $\bC$-vector bundle given by (the pullback of) $T_{C'}$ over $C'\times\bR_\tau$, equipped with $j_{\tau\hat y}$ over $C'\times\{\tau\}$, with a $\bC$-linear connection. 
    
    Next, consider the $\bC$-vector bundle $E\to C'\times\bR_s\times\bR_\tau$ whose fiber over a point $(z,s,\tau)$ is the space of $\bC$-anti-linear maps 
    \begin{align}\label{bundle-of-antilinear}
        (T_{C',z},j_{\tau\hat y}(z))\to(T_{X,x},\hat J_s(x)),
    \end{align}
    where $x = (\exp_{\hat\varphi}\tau\hat\xi)(z)$. $E$ has an evident $\bC$-linear connection (which we again denote by $\nabla$) induced by the previous paragraph. Explicitly, we get the bundle $E$ by taking the bundle $\tilde E$ on $C'\times\bR_\tau\times X\times\bR_s$ with fiber over a point $(z,\tau,x,s)$ given by the space of $\bC$-anti-linear maps \eqref{bundle-of-antilinear} and then pulling back under the map
    \begin{align}
        C'\times\bR_s\times\bR_\tau&\to C'\times\bR_\tau\times X\times\bR_s \\
        (z,s,\tau)&\mapsto(z,\tau,(\exp_{\hat\varphi}\tau\hat\xi)(z),s)
    \end{align}
    which sends the vector field $\partial/\partial s$ to the vector field $\partial/\partial s$. From the previous paragraph, it follows that the curvature form of $\tilde E$ (and therefore of $E$) is annihilated by contraction with $\partial/\partial s$. In particular, $[\nabla_s,\nabla_\tau] = 0$ on $E$.
    
    Now, $E$ has a smooth section $\psi$ given by
    \begin{align}
        \psi(z,s,\tau) = \delbar_{\hat J_s,j_{\tau \hat y}}(\exp_{\hat\varphi}\tau\hat\xi)|_z
    \end{align}
    vanishing along $C'\times\bR_s\times\{0\}$.
    Noting that differentiating $\phi_s(\tau)$ corresponds to \emph{covariantly} differentiating $\psi(\cdot,s,\tau)$, we will be done once we show that $\nabla_{s}\nabla^k_{\tau}\psi$ vanishes at $\tau = 0$ for each $0\le k\le n$ and that $\nabla_{s}\nabla^{n+1}_{\tau}\psi$ at $s = \tau = 0$ is given by the right side of \eqref{first-variation-fred-derivative} once we apply $\Pi^N$. Using the fact that $[\nabla_s,\nabla_\tau] = 0$ on $E$, we can instead compute $\nabla^k_\tau\nabla_s\psi$ at $\tau = 0$. Below, we first compute $\nabla_s\psi$ at $s = 0$ and then apply $\nabla_\tau$ to the result (which is now defined on $C'\times\{0\}\times\bR_\tau$) $k$ times and finally set $\tau = 0$.
    
    For an $\bR$-linear map $T$ as in \eqref{bundle-of-antilinear}, denote its $\bC$-linear (resp. $\bC$-anti-linear) part by $T^{1,0}_{s,\tau}$ (resp. $T^{0,1}_{s,\tau}$). A straightforward computation gives
    \begin{align}\label{donaldson-divisor-type-calc}
        T^{0,1}_{s,\tau} = \textstyle\frac12(\text{id} + \hat J_s\hat J)T^{1,0}_{0,\tau} + \frac12(\text{id} - \hat J_s\hat J)T^{0,1}_{0,\tau}.
    \end{align}
    Deducing $\hat I_{X,s}^{-1} = -\frac12(\hat J_s + \hat J)\hat J_s$ from \eqref{iso-along-Js-path}--\eqref{explicit-antlinear-perturb-hat}, we get
    \begin{align}\label{del-delbar-expansion}
        \hat I_{X,s}^{-1}\circ T^{0,1}_{s,\tau} - T^{0,1}_{0,\tau} = -\textstyle\frac14(\hat J_s + \hat J)(\hat J_s - \hat J)T^{1,0}_{0,\tau} - \frac14(\hat J\hat J_s + \hat J_s\hat J + 2\cdot\text{id})T^{0,1}_{0,\tau}.
    \end{align}
    In formulas \eqref{donaldson-divisor-type-calc}--\eqref{del-delbar-expansion}, we have omitted $z$ and $x = (\exp_{\hat\varphi}\tau\hat\xi)(z)$ from the notation for readability. 
    Applying \eqref{del-delbar-expansion} to $T = d(\exp_{\hat\varphi}\tau\hat\xi)|_z$, we obtain
    \begin{align}\label{key-covariant-derivative-computation}
        \nabla_s\psi(\cdot,0,\tau) = -\textstyle\frac12\hat J\hat K\cdot\partial_{\hat J,j_{\tau\hat y}}(\exp_{\hat\varphi}\tau\hat\xi)
    \end{align}
    where we have used the $\hat J$-anti-linearity of $\hat K = \frac d{ds}(\hat J_s)|_{s=0}$ to eliminate the second term of \eqref{del-delbar-expansion}. Since $\hat K$ and its first $n$ derivatives vanish along the image $C\subset X$ of $\hat\varphi$, \eqref{key-covariant-derivative-computation} shows that $\nabla^k_\tau\nabla_s\psi(\cdot,0,0)\equiv 0$ for $0\le k\le n$. The same holds even at a general nonzero $s$ since the section $\frac d{ds}\hat J_s$ and its first $n$ derivatives also vanish along $C\subset X$. Further, \eqref{key-covariant-derivative-computation} shows that any term of $\nabla^{n+1}_\tau\nabla_s\psi(\cdot,0,0)$ involving $\le n$ derivatives of $\hat K$ has to vanish. Thus, the only term in $\nabla^{n+1}_\tau\nabla_s\psi(\cdot,0,0)$ that can survive at $s=\tau=0$ is the one involving the $(n+1)^\text{st}$ derivative of $\hat K$ (in the direction of $\hat\xi$).
    
    This yields \eqref{first-variation-fred-derivative} once we recall the definition of $\hat K_{n+1}$, rewrite $\hat J$ appearing in \eqref{key-covariant-derivative-computation} as $i = \sqrt{-1}$ and recall that $\hat\varphi$ is $\hat J$-holomorphic. In more detail, the reason for why only $\hat\xi^N$ (and not the full $\hat\xi$) appears in \eqref{first-variation-fred-derivative} is because of the second paragraph of Remark \ref{sections-vanishing-order} and the reason for why $\Pi^N$ appears in \eqref{first-variation-fred-derivative} is because $\hat K_{n+1}$ is the image of the full $(n+1)^\text{st}$ derivative of $\hat K$ over $C$ under the projection $\overline{\text{End}}_{\hat J}(T_X|_C)\to\overline{\text{Hom}}_{\hat J}(T_C,N_C)$. Using the notation in the proof of Lemma \ref{analytic-lemma}, the full $(n+1)^\text{st}$ derivative of $\hat K$ over $C$ is the bundle map $\tilde L$, while $\hat K_{n+1} = L$.
\end{proof}

Lemma \ref{truncation-of-fred-map} shows that $D\hat\Phi_s(0):\hat\clX_0\to\hat\clY$ is independent of $s$ and, as a result, we get canonical identifications of kernels and cokernels
\begin{align}
    T_{\hat\varphi}^\text{Zar}\clM_h(X,\hat J_s,dA) &= T_{\hat\varphi}^\text{Zar}\clM_h(X,\hat J,dA) \\
    E_{\hat\varphi}\clM_h(X,\hat J_s,dA) &= E_{\hat\varphi}\clM_h(X,\hat J,dA)
\end{align}
and a $1$-parameter family of Kuranishi maps
\begin{align}
    F_{\hat p(s)}:T_{\hat\varphi}^\text{Zar}\clM_h(X,\hat J,dA)\to E_{\hat\varphi}\clM_h(X,\hat J,dA)
\end{align}
obtained by applying Kuranishi reduction to the maps $\hat\Phi_s$ (using the same choice of splittings of $\ker D\hat\Phi(0)\to\hat\clX_0$ and $\hat\clY\to\coker D\hat\Phi(0)$ for each $s$; see Appendix \ref{kuranishi-appendix} for details on splittings). Lemma \ref{abstract-kuranishi-high-order-perturb} applies (by virtue of Lemma \ref{truncation-of-fred-map}) to show that the maps $D^kF_{\hat p(s)}(0)$ are $s$-independent for $0\le k\le n$ and that the first variation of $D^{n+1}(F_{\hat p(s)}-F_{\hat p})(0)$ can be computed using \eqref{first-variation-fred-derivative}. Since $\hat p = \hat p(0)$ lies on the degeneracy locus, we have a preferred splitting (see the discussion surrounding \eqref{reduced-kuranishi-ker} for details)
\begin{align}\label{degen-canonical-splitting}
    T_{\hat\varphi}^\text{Zar}\clM_h(X,\hat J,dA) = T_{\hat\varphi}\clM_h(C,d)\oplus\ker\hat\varphi^*\hat D
\end{align}
of \eqref{zar-ses}. Thus, we may express the submanifold $\hat Z_s\subset T_{\hat\varphi}^\text{Zar}\clM_h(X,\hat J,dA)$ corresponding to $\clM_h(C,d)$ contained in the zero locus of $F_{\hat p(s)}$ as the graph of a smooth function
\begin{align}
    \hat f_s:T_{\hat\varphi}\clM_h(C,d)\to\ker\hat\varphi^*\hat D
\end{align}
with $\hat f_s(0) = 0$ and $D\hat f_s(0) = 0$. In fact, since the submanifolds $\hat Z_s$ are $s$-independent, so are the functions $\hat f_s$. Indeed, we have a (manifestly $s$-independent) submanifold of $\hat\clX_0$ corresponding to $\clM_h(C,d)$ and each $\hat Z_s$ is just the isomorphic image in $T_{\hat\varphi}^\text{Zar}\clM_h(X,\hat J,dA) = \ker D\hat\Phi(0)$ of this submanifold under the $s$-independent linear projection $\hat\clX_0\to\ker D\hat\Phi(0)$. This leads to a factorization of the Kuranishi map
\begin{align}\label{s-dep-reduced-factorization}
    F_{\hat p(s)} = \langle\hat\varsigma - \hat f,F^\text{red}_{\hat p(s)}\rangle
\end{align}
as explained in the paragraph preceding \eqref{fully-reduced-kuranishi-map}, where we have dropped the subscript from $\hat f_s$ in view of its $s$-independence. The following purely algebraic observation will be useful.

\begin{Lemma}\label{jet-of-reduced-kur}
    The $D^kF^{\normalfont\text{red}}_{\hat p(s)}(0)$ is independent of $s$ for $0\le k\le n-1$. After restricting to $\ker\hat\varphi^*\hat D\subset T_{\hat\varphi}^{\normalfont\text{Zar}}\clM_h(X,\hat J,dA)$, we have
    \begin{align}\label{reduced-v-unreduced-taylor}
        D^{n+1}F_{\hat p(s)}(0)|_{\ker\hat\varphi^*\hat D} = (n+1)\cdot D^n F^{\normalfont\text{red}}_{\hat p(s)}(0)|_{\ker\hat\varphi^*\hat D}
    \end{align}
    for all $s$, where we make the natural identification
    \begin{align}
        \label{trivial-id-1} &\normalfont\text{Hom}((\ker\hat\varphi^*\hat D)^{\otimes n+1},\coker\hat\varphi^*\hat D) = \\  \label{trivial-id-2} &\normalfont\text{Hom}((\ker\hat\varphi^*\hat D)^{\otimes n},\text{Hom}(\ker \hat\varphi^*\hat D,\coker \hat\varphi^*\hat D)).
    \end{align}
\end{Lemma}
\begin{proof}
    The $s$-independence of $D^kF^\text{red}_{\hat p(s)}(0)$ for $0\le k\le n-1$ is a direct consequence of \eqref{s-dep-reduced-factorization} and the $s$-independence of $D^kF_{\hat p(s)}(0)$ for $0\le k\le n$. Using coordinates $(\hat{\bf z},\hat\varsigma)$ on $T_{\hat\varphi}^{\normalfont\text{Zar}}\clM_h(X,\hat J,dA)$, similar to \eqref{kura-std-form} and respecting the direct sum decomposition \eqref{degen-canonical-splitting}, we get
    \begin{align}\label{factorization-at-z0}
        F_{\hat p(s)}({\bf 0},\hat\varsigma) = \hat\varsigma\cdot F^\text{red}_{\hat p(s)}({\bf 0},\hat\varsigma)
    \end{align}
    from \eqref{s-dep-reduced-factorization}. Differentiating \eqref{factorization-at-z0} $n+1$ times at $\hat\varsigma = 0$ yields \eqref{reduced-v-unreduced-taylor}.
\end{proof}

Lemma \ref{jet-of-reduced-kur} shows that, for $2\le k\le n$, we must have $T^{(k)}_{\hat p(s)} \equiv T^{(k)}_{\hat p} = 0$. The next observation computes $T^{(n+1)}_{\hat p(s)}$ in terms of $D^nF^\text{red}_{\hat p(s)}(0)$.

\begin{Lemma}\label{jet-of-fully-reduced-kur}
    The linear map
    \begin{align}\label{restriction-reduced-kur}
        (\ker\hat\varphi^*\hat D)^{\otimes n+1}\to(\coker\hat\varphi^*\hat D)/W_{\hat\varphi}
    \end{align}
    obtained by restricting $D^n(F^{\normalfont\text{red}}_{\hat p(s)}-F^{\normalfont\text{red}}_{\hat p})(0)$ to $\ker\hat\varphi^*\hat D$, using the identifications \eqref{trivial-id-1}--\eqref{trivial-id-2} and then passing to the quotient modulo the subspace $W_{\hat\varphi}\subset\coker\hat\varphi^*\hat D$ coincides with $T^{(n+1)}_{\hat p(s)}$.
\end{Lemma}
\begin{proof}
     Recall that $T^{(n+1)}_{\hat p} = 0$ by assumption. If $n=1$, the result follows by inspection. If $n>1$, we apply Lemma \ref{abstract-kuranishi-high-order-perturb} to the smooth maps $F^\text{red}_{\hat p(s)}$ (which have $s$-independent $(n-1)$-jet at $0$ by Lemma \ref{jet-of-reduced-kur}) to get the result.
\end{proof}

\begin{remark}\label{summarizing-taylor-coeff-computations}
    Lemmas \ref{jet-of-reduced-kur} and \ref{jet-of-fully-reduced-kur} show that the first variation of $T^{(n+1)}_{\hat p(s)}$ is a nonzero multiple of the map of the type \eqref{restriction-reduced-kur} induced by the first variation of $D^{n+1}(F_{\hat p(s)} - F_{\hat p})(0)$ which, in turn, can be computed explicitly from Lemma \ref{truncation-of-fred-map} (with the help of Lemma \ref{abstract-kuranishi-high-order-perturb}).
\end{remark}

We are now in a position to prove our main technical result.

\begin{Lemma}\label{main-technical-result}
    The inclusion \eqref{bad-stratum-inclusion} has codimension $\ge 1$.
\end{Lemma}
\begin{proof}
     Near any point of $\clM^d_{{\bf b},G}(\clM^*_g(X,A);{\bf k},{\bf c})_\spadesuit$, we can choose auxiliary data (as in \textsection\ref{sec_ift}) and define Kuranishi models $F_{\hat p}$ as in \eqref{cover-kuranishi} for all points $\hat p\in\clM^d_{{\bf b},G}(\clM_g^*(X,A);{\bf k},{\bf c})$ in this neighborhood. Thus, in this neighborhood, the inclusion \eqref{bad-stratum-inclusion} is defined by the $d$ equations 
    \begin{align}\label{local-def-eqns}
        T^{(2)}_{\hat p} = \cdots = T^{(d+1)}_{\hat p} = 0.
    \end{align}
    Given any point $\hat p$ satisfying \eqref{local-def-eqns}, we will show that the linearization of this system of $d$ equations has rank $\ge 1$. This will suffice to show that \eqref{bad-stratum-inclusion} is the inclusion of a $C^\infty$ subvariety of codimension $\ge 1$ (in the sense of \cite[Appendix C]{Wendl-19}). 
    
    For the remainder of the proof, we freely use the notation established in the discussion starting from \eqref{fixed-bad-point} and leading up to the statement of the present lemma. In particular, recall that we have $0\ne\xi\in\ker\hat\varphi^*\hat D$, $0\ne\xi'\in\ker(\hat\varphi^*\hat D)^\dagger = (\coker\hat\varphi^*\hat D)^*$ annihilating $W_{\hat\varphi}$, an integer $1\le n\le d$ and a tensor $\hat K$ satisfying the conditions of Lemma \ref{analytic-lemma}. By Remark \ref{summarizing-taylor-coeff-computations}, in order to show that the first variation of $T^{(n+1)}_{\hat p(s)}$ at $s=0$ is nonzero, it suffices to check the that the number
    \begin{align}
        \int_{C'}\langle\hat\varphi^*(i\hat K_{n+1}),\xi^{\otimes n+1}\otimes\xi'\rangle
    \end{align}
    is nonzero. This is exactly the content of Lemma \ref{analytic-lemma} which concludes the proof that \eqref{bad-stratum-inclusion} is of codimension $\ge 1$ near $\hat p$.
\end{proof}
\section{An extension of Taubes' Gromov invariant to dimension $6$}\label{sec_GT}

In this section, we will apply the results of the last two sections to prove Theorem \ref{thm:Gr-invariance}. The resulting invariant can be viewed as a partial extension of Taubes' Gromov invariant (defined in \cite{Taubes-Gr}) to symplectic Calabi--Yau $3$-folds.
\subsection{Linear wall crossing}\label{lwc-development}
The purpose of this subsection is to introduce the linear wall-crossing invariant (Definition \ref{lwc-double}), which will serve as a ``correction term" in the definition of virtual count of embedded curves of class $2A$ in equation \eqref{Gr-def}. 

Fix a smooth closed oriented $2$-manifold $\Sigma$ of genus $g$ and a complex vector bundle $\pi:N\to\Sigma$ of rank $2$ such that $c_1(N)\cdot[\Sigma] = 2g-2$. Such pairs $(\Sigma,N)$ arise by considering the normal bundle of any embedded pseudo-holomorphic curve of genus $g$ in a symplectic Calabi--Yau $3$-fold.
\begin{Definition}[Space of operators]\label{space-of-CR-on-N}
    Let $\clJ(\Sigma)$ denote the space of almost complex structures on $\Sigma$ inducing the given orientation on $\Sigma$. For $\bK\in\{\bR,\bC\}$, define the space
    \begin{align}\label{space-of-CR-on-N-eq}
        \clC\clR_\bK(N)
    \end{align}
    to consist of pairs $(j,D)$, where $j\in\clJ(\Sigma)$ and $D$ is a $\bK$-linear Cauchy--Riemann operator on $N$ with respect to $j$, i.e., 
    \begin{align}
        D:\Omega^0(\Sigma,N)\to\Omega^{0,1}_j(\Sigma,N)
    \end{align}
    is a linear first order differential operator such that for all $s\in\Omega^0(\Sigma,N)$ and $f\in C^\infty(\Sigma,\bK)$, we have $D(fs) = (\delbar_jf)s + fDs$. This is a bundle of affine spaces over the contractible space $\clJ(\Sigma)$ and therefore is also contractible. Every $(j,D)\in\clC\clR_\bK(N)$ satisfies $\ind D = 0$.
\end{Definition}
\begin{Definition}[Ordinary Brill--Noether loci]
    For $\bK\in\{\bR,\bC\}$ and any integer $k\ge 1$, define the subspace
    \begin{align}
        \clW^k_\bK(N)\subset\clC\clR_\bK(N)
    \end{align}
    to consist of those pairs $(j,D)$ for which $\dim_\bK\ker D = k$. Also define
    \begin{align}
        \clW_\bK(N) = \bigcup_{k\ge 1}\clW^k_\bK(N).
    \end{align}
    Finally, define $\clW'(N)\subset\clC\clR_\bR(N)$ to be the set of pairs $(j,D)$ such that there exists a section $0\ne\sigma\in\ker D$ and a point $z\in\Sigma$ such that $\sigma(z) = 0$.
\end{Definition}

We remind the reader of the convention that all (co)dimensions appearing below are to be understood over $\bR$, unless otherwise specified.

\begin{Lemma}\label{triv-wall-codim}
    We have the following codimension estimates on $\clW^k_\bK(N)$.
    \begin{enumerate}[\normalfont(i)]
        \item $\clW^1_\bR(N)$ is a submanifold of $\clC\clR_\bR(N)$ of codimension $1$.
        \item For each $k\ge 2$, $\clW^k_\bR(N)$ is of codimension $\ge 4$ in $\clC\clR_\bR(N)$.
        \item For each $k\ge 1$, $\clW^k_\bC(N)$ is of codimension $\ge 2$ in $\clC\clR_\bC(N)$.
    \end{enumerate}
\end{Lemma}
\begin{proof}
    Given any element $(j,D)\in\clW^k_\bK(N)$, there exists (using \cite[Lemma 1.1.4]{DW-20} for example) a smooth function
    \begin{align}
        \clC\clR_\bK(N)\supset\clU\xrightarrow{\clS}\text{Hom}_\bK(\ker D,\coker D)
    \end{align}
    defined in a neighborhood $\clU$ of $(j,D)$ such that $\clU\cap\clW_\bK^k(N) = \clS^{-1}(0)$. Moreover, its linearization $d\clS|_{(j,D)}$ restricted to $\Omega^{0,1}_j(\Sigma,\text{Hom}_\bK(N,N))$, i.e. the space of variations of $D$ with $j$ held fixed, is given by
    \begin{align}\label{linearize-triv-wall}
        A\mapsto(\kappa\mapsto A\kappa\pmod{\text{im }D})
    \end{align}
    where $\kappa \in \ker D$.
    The statements (i)--(iii) will therefore follow by proving suitable rank estimates on this map. It suffices to prove rank estimates for the ``Petri map" (recall that $D^\dagger$ is the Serre dual of $D$ as in Definition \ref{serre})
    \begin{align}\label{triv-petri}
        \ker D\otimes_\bK\ker D^\dagger\to\Omega^0(C,N\otimes_\bK N^\dagger) \\
        s\otimes_\bK s'\mapsto(z\mapsto s(z)\otimes_\bK s'(z)).
    \end{align}
    Indeed, $A \in \Omega^{0,1}_j(\Sigma,\text{Hom}_\bK(N,N))$ from \eqref{linearize-triv-wall} naturally lies in the dual space of the right side of \eqref{triv-petri} and thus, a rank estimate for \eqref{triv-petri} will imply one for \eqref{linearize-triv-wall} by duality.
    By taking a $0\ne s\in\ker D$ and $0\ne s'\in\ker D^\dagger$ and using unique continuation for solutions of Cauchy--Riemann equations, it follows that the Petri map always has $\bK$-rank $\ge 1$. This proves (i) and (iii). Statement (ii) follows by using \cite[Remark 1.1.14]{DW-20} (which is itself based on \cite[Lemma 4.4]{Eft16}) to conclude that the rank of the Petri map is $\ge 2k$ whenever $k\ge 2$ and $\bK = \bR$.
\end{proof}
\begin{Lemma}\label{ker-with-zero}
    $\clW'(N)$ is of codimension $\ge 3$ in $\clC\clR_\bR(N)$.
\end{Lemma}
\begin{proof}
    Since $\clW'(N)\subset\clW_\bR(N)$, Lemma \ref{triv-wall-codim}(ii) implies that it is enough to show that $\clW'(N)\cap\clW^1_\bR(N)$ is of codimension $\ge 3$ in $\clC\clR_\bR(N)$.
    To prove this, we argue as follows. Define
    \begin{align}
        \widehat\clW'(N)\subset\clC\clR_\bR(N)\times\Sigma
    \end{align}
    to consist of $(j,D,z)$ such that $(j,D)\in\clW^1_\bR(N)$ and the elements of $\ker D$ vanish at $z$. It will be enough to show that $\widehat \clW'(N)$ is a submanifold of codimension $4$ in $\clW^1_\bR(N)\times\Sigma$. Fix any point $(j,D,z)\in\widehat\clW'(N)$. Choose metrics on $\Sigma$, $N$ -- this yields an $L^2$ inner product on the space of sections of $N$. Fix a section $\sigma\in\ker D$ with $\|\sigma\|_{L^2} = 1$ (note that $\sigma$ is unique up to multiplication by $\pm 1$). Then, for every $(j',D')$ in a sufficiently small neighborhood $\clU$ of $(j,D)$ in $\clW^1_\bR(N)$, there is a distinguished element $\sigma_{(j',D')}\in\ker D'$ determined by the condition that its $L^2$ inner product with $\sigma$ is $1$. Note that $\sigma_{(j,D)} = \sigma$ by definition. We now have the smooth map
    \begin{align}
        \clU\times\Sigma&\xrightarrow{\mathcal{E}} N \\
        \label{ev-ker-elt} (j',D',z')&\mapsto\sigma_{(j',D')}(z')
    \end{align}
    which is compatible with the projection to $\Sigma$ and satisfies $\widehat\clW'(N)\cap(\clU\times\Sigma) = \mathcal{E}^{-1}(0)$. It will now suffice to show that the linearization
    \begin{align}\label{eqn:e-linear}
        d\mathcal{E}|_{(j,D,z)}:T_{(j,D)}\clW^1_\bR(N)\oplus T_z\Sigma\to N_z
    \end{align}
    is surjective. We will in fact show the stronger statement that variations in $D$ with $j$ and $z$ held constant suffice to surject onto $N_z$. First, we need to find an explicit expression for this linearization. To this end, let $\Pi:\Omega^{0,1}_j(\Sigma,N)\to\coker D$ be the natural projection and let $Q:\text{im }D\to(\ker D)^\perp$ be the right inverse of $D$ which we extend to $\Omega^{0,1}_j(\Sigma,N)$ by declaring $Q|_{(\text{im }D)^\perp} \equiv 0$.
    
    Let $A\in\Omega^{0,1}_j(\Sigma,\text{Hom}_\bR(N,N))$ be such that $A\in T_{(j,D)}\clW^1_\bR(N)$ (from the proof of Lemma \ref{triv-wall-codim}(i), this means that $\Pi A\sigma = 0$). Then, there exists a smooth path $t\mapsto D_t\in\clW^1_\bR(N)$ such that $\dot D_0 = A$. Denoting $\sigma_t = \sigma_{(j,D_t)} = \sigma + \xi_t$ with $\xi_t\perp\sigma$ and differentiating the identity $D_t\sigma_t \equiv 0$, we find
    \begin{align}
        0 = A\sigma + D\dot\xi_0
    \end{align}
    which gives $\dot\xi_0 = -QA\sigma$. 
    
    We claim that the surjectivity of \eqref{eqn:e-linear}, restricted to the subspace where variations of $j$ and $z$ are zero, will follow once we show the surjectivity of
    \begin{align}
        L:\Omega^{0,1}_j(\Sigma,\text{Hom}_\bR(N,N))\to N_z\oplus\coker D \\
        L(A):= (-(QA\sigma)(z),\Pi A\sigma).
    \end{align}
    Indeed, if $L$ is surjective, then so is its restriction $L^{-1}(N_z \oplus \{0\})\to N_z$. As recalled in the preceding paragraph, $L^{-1}(N_z\oplus\{0\})$ is identified with the subspace of $T_{(j,D)}\clW^1_\bR(N)$ where $j$ is held constant. The claim now follows by observing that $L^{-1}(N_z \oplus \{0\})\to N_z$ is the desired restriction of \eqref{eqn:e-linear}.
    
    Now, if we let $D_z:\Omega^0_z(\Sigma,N)\to\Omega^{0,1}_j(\Sigma,N)$ be the restriction of $D$ to the space of sections of $N$ vanishing at $z$ and $\Pi':\Omega^{0,1}_j(\Sigma,N)\to\coker D_z$ be the corresponding projection, we can construct a map
    \begin{align}\label{eqn:5.16}
        \coker D_z\to N_z\oplus\coker D \\
        \Pi'\eta\mapsto(-(Q\eta)(z),\Pi\eta).
    \end{align}
    It follows from the long exact sequence relating the (co)kernels of $D_z$ and $D$ with $N_z$, that the map \eqref{eqn:5.16} is an isomorphism. Indeed, since we have $\ker D_z = \ker D$ (both are spanned by $\sigma$), the long exact sequence mentioned above yields the short exact sequence
    \begin{align}\label{eqn:ev-les}
        0\to N_z\to\coker D_z\to\coker D\to 0
    \end{align}
    and the map $\Pi'\eta\mapsto (Q\eta)(z)$ splits the connecting map $N_z\to\coker D_z$ in \eqref{eqn:ev-les}.
    Under the isomorphism \eqref{eqn:5.16}, the map $L$ is simply $A\mapsto\Pi'A\sigma$. 
    
    Suppose for the sake of contradiction that $L$ is not surjective. Then, there is a nonzero distributional section $\tau$ of $N^\dagger$ such that $\langle A\sigma,\tau\rangle = 0$ for all smooth $A$ (i.e., $\sigma\otimes\tau\equiv 0$ as a distributional section of $N\otimes_\bR N^\dagger$) and $\tau|_{\text{im }D_z}\equiv 0$. The second condition implies that $D^\dagger\tau$, viewed as a functional on $\Omega^0(\Sigma,N)$, vanishes on the subspace $\Omega^0_z(\Sigma,N)$. As a result, there is a linear functional $\Lambda_\tau:N_z\to\bR$ such that for any $\sigma\in\Omega^0(\Sigma,N)$, we have $\langle D^\dagger\tau,\sigma\rangle = \Lambda_\tau(\sigma(z))$. Thus, $D^\dagger\tau$ is of class $(L_1^p)^*$ for any $2<p<\infty$ (where $L_1^p$ is the space of $L^p$ functions with one derivative in $L^p$ and we are using the Sobolev embedding $L_1^p\subset C^0$). By elliptic regularity, $\tau$ must be of class $(L^p)^* = L^q$ (where $\frac1p+\frac1q=1$) and smooth in the complement of $z$ (since it satisfies $D^\dagger\tau = 0$ here). But now $\sigma\otimes\tau\equiv 0$ with $\sigma\not\equiv 0$ and $\tau\not\equiv 0$ on the complement of $z$ contradicts the fact that $\sigma$ and $\tau$ satisfy unique continuation on the complement of $z$.
\end{proof}
Now, let $h\ge g$ be an integer such that $r:=(2h-2)-2(2g-2)\ge 0$. Then, for any $j\in\clJ(\Sigma)$, the complex orbifold $\clM_h(\Sigma,j,2)$ of genus $h$ holomorphic double covers of $(\Sigma,j)$ branched at $r$ points is non-empty. In fact, the underlying coarse moduli spaces\footnote{In this context, the coarse moduli space is just obtained by forgetting the $\bZ/2\bZ$ stabilizers in the purely ineffective orbifold $\clM_h(\Sigma,j,2)$. In particular, the coarse moduli space has no quotient singularities.} are canonically identified for all $j$ (since the isomorphism class of a cover is determined by purely topological data once $j$ is specified). Call this common coarse moduli space $M_h(\Sigma,2)$; it is naturally a smooth manifold of dimension $2r$.

Now, given $(j,D,[\varphi:\Sigma'\to\Sigma])\in\clC\clR_\bK(N)\times M_h(\Sigma,2)$, the operator $D_\varphi = \varphi^*D$ on $\varphi^*N$ admits a splitting $D_\varphi = D_\varphi^+\oplus D_\varphi^-$ into a direct sum of Cauchy--Riemann operators on isotypic subspaces corresponding to the trivial and sign representations of $\bZ/2\bZ = \text{Aut}(\varphi)$. We have $\ind D_\varphi^+ = \ind D = 0$ and $\ind D_\varphi^- = \ind D_\varphi = -2r$.
\begin{Definition}[Equivariant Brill--Noether loci]
    Given $\bK\in\{\bR,\bC\}$, an integer $k\ge 1$ and integers $h$ and $r$ as above, define the subset
    \begin{align}
        \clW^k_{2,h,\bK}(N)\subset\clC\clR_\bK(N)\times M_h(\Sigma,2)
    \end{align}
    to consist of those $(j,D,[\varphi:\Sigma'\to\Sigma])$ for which $\dim_\bK\ker D_\varphi^- = k$. Also define
    \begin{align}
        \clW_{2,h,\bK}(N) = \bigcup_{k\ge 1}\clW^k_{2,h,\bK}(N).
    \end{align}
\end{Definition}

\begin{Lemma}\label{sign-wall-codim}
    The following codimension estimates on $\clW^k_{2,h,\bK}(N)$ hold.
    \begin{enumerate}[\normalfont(i)]
        \item $\clW^1_{2,h,\bR}$ is a submanifold of $\clC\clR_\bR(N)\times M_h(\Sigma,2)$ of codimension $2r+1$.
        \item For each $k \ge 2$, $\clW^k_{2,h,\bR}$ has codimension $\ge 4r + 4$ in $\clC\clR_\bR(N)\times M_h(\Sigma,2)$.
        \item For each $k\ge 1$, $\clW^k_{2,h,\bC}$ has codimension $\ge 2r+2$ in $\clC\clR_\bC(N)\times M_h(\Sigma,2)$.
    \end{enumerate}
\end{Lemma}
\begin{proof}
    The proof is a slight variation of the argument from Lemma \ref{triv-wall-codim}. As in that argument, given any point $(j,D,[\varphi:\Sigma'\to\Sigma])\in\clW^k_{2,h,\bK}$, we can find a smooth function
    \begin{align}
        \clC\clR_\bK(N)\times M_h(\Sigma,2)\supset\clU\xrightarrow{\clS}\text{Hom}_\bK(\ker D_\varphi^-,\coker D_\varphi^-)
    \end{align}
    defined on a neighborhood $\clU$ of $(j,D,[\varphi:\Sigma'\to\Sigma])$ such that $\clU\cap\clW^k_{2,h,\bK} = \clS^{-1}(0)$. Moreover, the linearization $d\clS|_{(j,D,[\varphi])}$ restricted to $\Omega^{0,1}_j(\Sigma,\text{Hom}_\bK(N,N))$, i.e. the space of variations of $D$ with $j$ and $[\varphi]$ held fixed, is given by
    \begin{align}\label{eqn:sign-wall-linearize}
        A\mapsto (\kappa\mapsto(\varphi^*A)\kappa \pmod{\text{im }D_\varphi^-})
    \end{align}
    where $\kappa \in \ker D_\varphi^-$. Thus, exactly as in Lemma \ref{triv-wall-codim}, the proofs of (i)--(iii) will follow from suitable rank estimates on this map. By duality, it is enough to prove rank estimates for the ``Petri map"
    \begin{align}\label{double-petri}
        \ker D_\varphi^-\otimes_\bK\ker((D_\varphi)^\dagger)^-&\to\Omega^0(\Sigma',\varphi^*N\otimes_\bK (\varphi^*N)^\dagger)^{\bZ/2\bZ} \\
        \sigma\otimes_\bK\sigma'&\mapsto(\zeta\mapsto\sigma(\zeta)\otimes_\bK\sigma'(\zeta))
    \end{align}
    where $(\cdot)^{\bZ/2\bZ}$ denotes the $\bZ/2\bZ = \text{Aut}(\varphi)$-invariant part.
    The key observation is that $\ker D_\varphi^-\otimes_\bK\ker((D_\varphi)^\dagger)^-$ already consists of $(\bZ/2\bZ)$-invariant elements -- indeed, $\bZ/2\bZ$ acts by the sign representation on both $\ker D_\varphi^-$ and $\ker ((D_\varphi)^\dagger)^-$ -- and so, the map \eqref{double-petri} is well-defined. Thus, we are reduced again to the argument of Lemma \ref{triv-wall-codim}(i)--(iii). More precisely, we can show by unique continuation that the Petri map has $\bR$-rank $= 2r+1$ in case of (i) and $\bC$-rank $\ge r+1$ in case of (ii). For statement (iii), we invoke \cite[Remark 1.1.14]{DW-20} (or \cite[Remark 1.3.12]{DW-20}) to conclude that the rank of the Petri map is $\ge 2(2 + 2r)$ in case of (ii).  
\end{proof}
\begin{remark}
The injectivity of Petri maps in the proofs of both Lemma \ref{triv-wall-codim} and Lemma \ref{sign-wall-codim} follows without appealing to Wendl's work on generic injectivity of Petri maps \cite{Wendl-19} because the degrees of the corresponding covers are small enough.
\end{remark}
\begin{Definition}[Co-orientation on $\clW^1_{2,h,\bR}(N)$]\label{natural-coor}
    Given a tuple $(j,D,[\varphi:\Sigma'\to\Sigma])$ lying in $\clW^1_{2,h,\bR}(N)$, its normal space in $\clC\clR_\bR(N)\times M_h(\Sigma,2)$ is canonically identified with
    \begin{align}\label{double-wall-normal}
        \normalfont\text{Hom}(\ker D_\varphi^-,\coker D_\varphi^-).
    \end{align}
    We define an orientation on \eqref{double-wall-normal} in the usual fashion, i.e., we deform $D$ in a $1$-parameter family to a $\bC$-linear operator and use the spectral flow. This assignment of orientations to the normal spaces of the forgetful map $\clW^1_{2,h,\bR}(N)\to\clC\clR_\bR(N)$, at points where this map is an immersion, is called the natural co-orientation on $\clW^1_{2,h,\bR}(N)$. Note that here we are using the natural orientation on $M_h(\Sigma,2)$ coming from the orientation on $\Sigma$.
\end{Definition}
In order to make the next definition, we will need the following well-known fact which is implicit in \cite[Appendix B]{Wendl-19} and \cite[Lemma 3.6]{doan2021castelnuovo}.
\begin{Lemma}\label{CR-op-v-acs}
    Given a complex curve $C$ (not necessarily compact) and a complex vector bundle $\pi:N\to C$ on it, there is a natural bijection $D\mapsto J_D$ between
    \begin{enumerate}[\normalfont (i)]
        \item $\bR$-linear Cauchy--Riemann operators $D:\Omega^0(C,N)\to\Omega^{0,1}(C,N)$ and,
        \item almost complex structures $J$ on the total space of $N$ which
        \begin{enumerate}[\normalfont (a)]
            \item restrict to the standard complex structure on each fiber of $\pi$,
            \item restrict to the given complex structure on $C$ along the zero section,
            \item are invariant under the $\bR^\times$ action $\lambda\cdot(z,v) = (z,\lambda v)$ on $N$ and,
            \item such that the projection $\pi$ is $J$-holomorphic.
        \end{enumerate}
    \end{enumerate}
    Under this bijection, if $C'$ is another complex curve (not necessarily compact), then a smooth map $\hat s: C'\to N$ is $J_D$-holomorphic if and only if $\varphi:=\pi\circ\hat s:C'\to C$ is holomorphic and the corresponding section $s\in\Omega^0(C',\varphi^*N)$ given by
    \begin{align}\label{eqn:curve-section}
        z\in C'\mapsto \hat s(z)\in N_{\varphi(z)}    
    \end{align}
    satisfies $(\varphi^*D)s = 0$.
\end{Lemma}
\begin{proof}
    It is enough to consider the case of the trivial bundle $U\times\bC^r\to U$ where $U\subset\bC$ is an open subset with complex coordinate $z$ and $\bC^r$ is endowed with the complex coordinates $w = (w_1,\ldots,w_r)$. Let $J$ be any almost complex structure on $U\times\bC^r$ which satisfies conditions (a)--(d). Using (a), (b) and (d), we find that the bundle $\Omega^{1,0}_J$ near $U\times 0\subset U\times\bC^r$ has a local frame consisting of $dz$ and
    \begin{align}\label{acs-frame}
        dw_i + A_i(z,w)d\bar z
    \end{align}
    for $1\le i\le r$ where the $A_i$ are smooth functions vanishing when $w_i=0$ for all $1 \leq i \leq r$. Now, (c) implies that the $A_i$ are actually $\bR$-linear in $w$, i.e., they are of the form
    \begin{align}\label{acs-frame-2}
        A_i(z,w) = \hat A_i(z)[w]
    \end{align}
    where $\hat A = (\hat A_i)_{i=1}^r:U\to\text{End}_\bR(\bC^r)$ is a smooth function. Moreover, the expression \eqref{acs-frame} for a frame of $\Omega^{1,0}_J$ is valid on the whole of $U\times\bC^r$. To this almost complex structure, we can associate the $\bR$-linear Cauchy--Riemann operator $D = \delbar + \hat A[\cdot]$ on the bundle $U\times\bC^r\to U$. Conversely, any $\bR$-linear Cauchy--Riemann operator $D$ on $U\times\bC^r$ can be written as $\delbar + \hat A[\cdot]$ for some smooth function $\hat A:U\to\text{End}_\bR(\bC^r)$ and thus, defines an almost complex structure $J_D$ on $N$, satisfying (a)--(d), by taking the $(1,0)$ forms to be spanned by $dz$ along with $dw_i + \hat A_i(z)[w]d\bar z$ for $1\le i\le r$.
    
    It is easy to see that this gives a bijective correspondence $D\mapsto J_D$ and that the bijection is preserved under the natural action of the group $C^\infty(U,GL_r(\bC))$ on both sides. Finally, for any open subset $U'\subset\bC$ and smooth map $\hat s:U'\to U\times\bC^r$ given by $\zeta\mapsto(\varphi(\zeta),s(\zeta))$ is $J_D$-holomorphic if and only if we have
    \begin{align}\label{eqn_redtext}
        \delbar\varphi = 0, \quad\delbar s + \hat A(\varphi)[s]\delbar\bar\varphi = 0
    \end{align}
    as desired. Now, the lemma extends to the general case by naturality.
\end{proof}
Let us now observe that for any $(j,D,[\varphi:\Sigma'\to\Sigma]) \in \clW^1_{2,h,\bR}(N)$, all non-zero elements of $\ker D_\varphi^-$ are non-zero scalar multiples of each other. Thus, for any $0 \neq \sigma \in \ker D_\varphi^-$, the corresponding $J_D$-holomorphic map $\hat\sigma:\Sigma'\to N$ is immersed/injective if and only if the same property holds for $a\cdot\sigma$ where $0\neq a\in \bR$. Therefore, we can make sense of the locus within $\clW^1_{2,h,\bR}$ where this immersion/injectivity property fails. The discussion which follows constructs Baire sets which avoid this locus (i.e. the immersion and injectivity property hold). These Baire sets will be used in the definition of linear wall crossing invariant (Definition \ref{lwc-double}).
\begin{Definition}
    Define the subsets
    \begin{align}
        &\normalfont\clW^1_{2,h,\bR}(N)_\text{imm}, \\
        &\normalfont\clW^1_{2,h,\bR}(N)_\text{inj}
    \end{align}
    of $\clW^1_{2,h,\bR}(N)$ to consist of exactly those $(j,D,[\varphi:\Sigma'\to\Sigma])$ for which the $J_D$-holomorphic map $\hat\sigma:\Sigma'\to N$ corresponding to a nonzero $\sigma\in\ker D_\varphi^-$ fails to be immersed or injective respectively. Also define
    \begin{align}
        \normalfont\clW^1_{2,h,\bR}(N)_\text{emb} = \normalfont\clW^1_{2,h,\bR}(N)_\text{imm}\cup \normalfont\clW^1_{2,h,\bR}(N)_\text{inj}.
    \end{align}
\end{Definition}
\begin{Lemma}\label{emb-multisection-wall}
    We have the following codimension estimates.
    \begin{enumerate}[\normalfont(i)]
        \item $\normalfont\clW^1_{2,h,\bR}(N)_\text{imm}$ has codimension $\ge 2r + 5$ in $\clC\clR_\bR(N)\times M_h(\Sigma,2)$.
        \item $\normalfont\clW^1_{2,h,\bR}(N)_\text{inj}$ has codimension $\ge 2r + 3$ in $\clC\clR_\bR(N)\times M_h(\Sigma,2)$.
    \end{enumerate}
    In particular, $\normalfont\clW^1_{2,h,\bR}(N)_\text{emb}$ has codimension $\ge 2r+3$ in $\clC\clR_\bR(N)\times M_h(\Sigma,2)$.
\end{Lemma}
\begin{proof}
    The proof is just a combination of the ideas from Lemmas \ref{ker-with-zero} and \ref{sign-wall-codim}. More precisely, the proof of (ii) exactly mirrors the proof of Lemma \ref{ker-with-zero} while (i) requires a slightly more involved argument.
    \begin{enumerate}[(i)]
        \item Consider the finite cover $\tilde \clC_\text{imm}\to\clC\clR_\bR(N)\times M_h(\Sigma,2)$ which consists of tuples $(j,D,[\varphi:\Sigma'\to\Sigma])$ along with the choice of a point $\zeta\in\text{Crit }\varphi$. Within $\tilde\clC_\text{imm}$, we can consider the subset $\tilde\clW^1_{2,h,\bR}(N)_\text{imm}$ consisting of $(j,D,[\varphi:\Sigma'\to\Sigma],\zeta)$ such that $(j,D,[\varphi:\Sigma'\to\Sigma])\in\clW^1_{2,h,\bR}(N)$ and $d\hat\sigma|_\zeta = 0$ for all $0\ne\sigma\in\ker D_\varphi^-$ (where $\hat\sigma$ is the $J_D$-holomorphic map corresponding to $\sigma$ via Lemma \ref{CR-op-v-acs}). Since $\hat\sigma$ is $J_D$-holomorphic, its derivative at $\zeta$ is $\bC$-linear. Equivalently, since $\sigma\in\ker D^-_\varphi$, the derivative $(\varphi^*\nabla)\sigma|_\zeta$ (which is well-defined independent of the choice of $\nabla$ on $N$ since $\zeta\in\text{Crit }\varphi$) is $\bC$-linear. In other words, we see that $\tilde\clW^1_{2,h,\bR}(N)_\text{imm}\subset\tilde\clC_\text{imm}$ is locally given by the zeros of the map
        \begin{align}\label{eqn:5.33}
            (j,D,[\varphi:\Sigma'\to\Sigma],\zeta) &\mapsto (\varphi^*\nabla)^{1,0}\sigma_{(j,D)}|_\zeta
        \end{align}
        where $(j,D,[\varphi:\Sigma'\to\Sigma])\mapsto \sigma_{(j,D)}\ne 0$ is a (locally defined) distinguished choice of element in the $1$-dimensional vector space $\ker D^-_{\varphi}$ (specified by an $L^2$ normalization condition as in the definition the analogous map \eqref{ev-ker-elt} in the proof of Lemma \ref{ker-with-zero}).
        
        To prove the desired codimension estimate for $\clW^1_{2,h,\bR}(N)_\text{imm}$, it is sufficient to show that $\tilde\clW^1_{2,h,\bR}(N)_\text{imm}\subset\tilde\clC_\text{imm}$ is a submanifold of codimension $2r+5$. To this end, fix a point $(j,D,[\varphi:\Sigma'\to\Sigma],\zeta)\in\tilde\clW^1_{2,h,\bR}(N)_\text{imm}$. Let $\sigma$ be a nonzero element of the $1$-dimensional vector space $\ker D^-_\varphi$. Denote by $\Omega^0(\Sigma',\varphi^*N)^-_\zeta$ the subspace of $\Omega^0(\Sigma',\varphi^*N)^-$ consisting of sections whose derivative at $\zeta$ is $\bC$-anti-linear. Note that since $\zeta \in \text{Crit } \varphi$, the derivative of elements in $\Omega^0(\Sigma',\varphi^*N)^-$ is well-defined and independent of choice of connection on $N$ (since the difference of two connections on $N$ is a tensor on $\Sigma$ whose pullback along $\varphi$ vanishes at $\zeta$). Write $(D_\varphi^-)_\zeta: \Omega^0(\Sigma',\varphi^*N)^-_\zeta \to \Omega^{0,1}_j(\Sigma',\varphi^*N)^-$ for the restriction of $D_\varphi^-$ to $\Omega^0(\Sigma',\varphi^*N)^-_\zeta$ and let $\Pi'$ be the projection from $\Omega^{0,1}_j(\Sigma',\varphi^*N)^-$ to $\coker (D_\varphi^-)_\zeta$. Arguing as the proof of Lemma \ref{ker-with-zero}, we are reduced to showing that
        \begin{align}
            L:\Omega^{0,1}_j(\Sigma,\text{Hom}_\bR(N,N))\to\coker (D_\varphi^-)_\zeta \\
            L(A):=\Pi'((\varphi^*A)\sigma)
        \end{align}
        is surjective. We remark that the map \eqref{eqn:5.33} naturally leads us to consider the space $\Omega^0(\Sigma',\varphi^*N)^-_\zeta$, analogous to how the space $\Omega^0_z(\Sigma,N)$ naturally arises from the map \eqref{ev-ker-elt} in the proof of Lemma \ref{ker-with-zero}. Thus, we see that $\Omega^0(\Sigma',\varphi^*N)^-_\zeta$ plays a similar role in this proof as $\Omega^0_z(\Sigma,N)$ did in the proof of Lemma \ref{ker-with-zero}.
        
        Assume, for the sake of contradiction, that $L$ is not surjective. Then, arguing as in the proof of Lemma \ref{ker-with-zero}, we deduce the existence of a nonzero distributional section $\tau$ of $(\varphi^*N)^\dagger$ which
        \begin{enumerate}[(a)]
            \item is $\bZ/2\bZ = \text{Aut}(\varphi)$-anti-invariant,
            \item vanishes on the image of $(D_\varphi^-)_\zeta$ and,
            \item is such that $\sigma\otimes\tau\equiv 0$ as a distributional section of $\varphi^*N\otimes_\bR(\varphi^*N)^\dagger$ (for this last condition, we are using the observation that $\sigma\otimes\tau$ is invariant under $\bZ/2\bZ$ --  see the parallel argument in the proof of Lemma \ref{sign-wall-codim}).
        \end{enumerate}
        Now, note that (a) and (b) imply that $(D_\varphi)^\dagger\tau$, viewed as a functional on $\Omega^0(\Sigma',\varphi^*N)$, vanishes on the subspace $\Omega^0(\Sigma',\varphi^*N)^+\oplus\Omega^0(\Sigma',\varphi^*N)^-_\zeta$ and therefore factors through the projection 
        \begin{align}\label{c-linear-derivative}
            \Omega^0(\Sigma',\varphi^*N)&\to\text{Hom}_\bC(T_\zeta\Sigma',N_{\varphi(\zeta)})\\
            s&\mapsto(\varphi^*\nabla)^{1,0}s|_\zeta.
        \end{align}
        Thus, $(D_\varphi)^\dagger\tau$ is of class $(L_2^p)^*$ for any $2<p<\infty$ (where $L_2^p$ is the space of $L^p$ functions with two derivatives in $L^p$ and we are using the Sobolev embedding $L_2^p\subset C^1$). By elliptic regularity it now follows that $\tau$ is of class $(L_1^p)^*$ and smooth in the complement of $\zeta$ (since it satisfies $(D_\varphi)^\dagger\tau = 0$ here). Now, if the support of $\tau$ is not just the point $\zeta$, then by unique continuation for $\sigma$ and $\tau$ (on the complement of $\zeta$), we must have $\sigma\otimes\tau\not\equiv 0$. Thus, we must have $\text{supp }\tau = \{\zeta\}$. It follows that there is an integer $m\ge 0$ such that $\tau$ applied to any element of $\Omega^{0,1}_j(\Sigma',\varphi^*N)$ depends only on the $m$-jet of this element at $\zeta$. Since $\tau$ is known to be of class $(L_1^p)^*$ and $L_1^p\not\subset C^1$, it follows that $m=0$ and as a result $\tau$ acts on the space $\Omega^{0,1}_j(\Sigma',\varphi^*N)$ via evaluation at $\zeta$ followed by a linear functional
        \begin{align}
            M:(\Lambda^{0,1}T^*_{\Sigma',j}\otimes_\bC\varphi^*N)_\zeta\to\bR.    
        \end{align}
        But now, it is easy to produce a smooth section $s$ of $\varphi^*N$ such that 
        \begin{align}
            -\langle (D_\varphi)^\dagger\tau,s\rangle = \langle\tau,D_\varphi s\rangle = M(D_\varphi s|_\zeta) = M((\varphi^*\nabla)^{0,1}s|_\zeta)\ne 0
        \end{align}
        and $(\varphi^*\nabla)^{1,0}s|_\zeta = 0$. This contradicts the fact that $(D_\varphi)^\dagger\tau$ factors through the map \eqref{c-linear-derivative}.
        
        \item The proof is similar to (i) but easier and so, we only indicate the setup. Define $\tilde\clC_\text{inj}$ to be the subset of $\clC\clR_\bR(N)\times M_h(\Sigma,2)\times\Sigma$ consisting of tuples $(j,D,[\varphi:\Sigma'\to\Sigma],z)$ such that $z\not\in\varphi(\text{Crit }\varphi)$. We can then define the subset $\tilde\clW^1_{2,h,\bR}(N)_\text{inj}\subset\tilde\clC_\text{inj}$ to consist of those elements $(j,D,[\varphi:\Sigma'\to\Sigma],z)$ for which $(j,D,[\varphi])\in\clW^1_{2,h,\bR}(N)$ and $\sigma|_{\varphi^{-1}(z)}\equiv 0$ for all $0\ne\sigma\in\ker D_\varphi^-$. It can now be shown as in (i) that $\tilde\clW^1_{2,h,\bR}(N)_\text{inj}\subset\tilde\clC_\text{inj}$ is a submanifold of codimension $2r+5$. Noting that $\clW^1_{2,h,\bR}(N)_\text{inj}$ is the image of $\tilde\clW^1_{2,h,\bR}(N)_\text{inj}$ under the natural projection gives the desired result.
    \end{enumerate}
    The statement concerning $\clW^1_{2,h,\bR}(N)_\text{emb}$ follows from (i) and (ii).
\end{proof}

\begin{remark}
The codimension estimates from Lemma \ref{emb-multisection-wall} assert that the holomorphic curve associated to the kernel of a generic element in $\normalfont\clW^1_{2,h,\bR}(N)$ is embedded. The injectivity of the relevant Petri map is not hard to establish because our discussion just involves $(\bZ/2\bZ)$-equivariant Cauchy--Riemann operators from $2$-fold covers. In fact, it is possible to extend the statements in Lemma \ref{emb-multisection-wall} to non-trivial elementary walls as in Definition \ref{elem-wall}. The fact that $\dim_{\bR} \mathbf{k}^H = 1$, which follows from the definition of non-trivial elementary wall type, is necessary in order to define the holomorphic sections uniquely up to rescaling. However, one needs to adapt the proof of genetric injectivity of Petri maps from \cite{Wendl-19} to this context. 
\end{remark}
\begin{remark}\label{rmk_mult_slag}
Any element $0 \neq \sigma \in \ker D_\varphi^-$ defines a \emph{multi-valued section} of the bundle $N \to \Sigma$. The conclusion of Lemma \ref{emb-multisection-wall} can therefore be phrased as the ``generic embeddedness" of multi-valued sections. As pointed out in Remark \ref{rmk_slag}, it may be possible to use similar arguments to establish similar generic behavior for $2$-valued harmonic forms relevant to special Lagrangian geometry.
\end{remark}

The next definition introduces the notion of ``$2$-rigid" operators, paths between them and fixed endpoint homotopies between such paths. Such operators arise naturally as normal deformation operators of embedded pseudo-holomorphic curves with respect to any super-rigid almost complex structure. 
\begin{Definition}[2-rigidity]
    \begin{enumerate}[\normalfont(a)]\label{defn:2-rig}
        \item We say that $\alpha = (j,D)\in\clC\clR_\bR(N)$ is \emph{$2$-pre-rigid} if it misses the images of $\clW'(N),\clW^k_{2,h,\bR}(N)$ and $\normalfont\clW^1_{2,h,\bR}(N)_\text{emb}$ in $\clC\clR_\bR(N)$ for all $h\ge g$ and $k\ge 2$. We say that $\alpha$ is $2$-rigid if, in addition, it also misses the image of the map
        \begin{align}\label{double-wall-to-ops}
            \clW^1_{2,h,\bR}(N)\to\clC\clR_\bR(N)
        \end{align}
        for all $h\ge g$. Denote the set of $2$-rigid elements by $\clC\clR_\bR(N)_2$ and its intersection with $\clC\clR_\bC(N)\setminus\clW_\bC(N)$ by $\clC\clR_\bC(N)_2$.
        \item Given $\alpha_\pm\in\clC\clR_\bR(N)_2$, define $\clP(\alpha_-,\alpha_+)$ to be the space of smooth paths in $\clC\clR_\bR(N)$ from $\alpha_-$ to $\alpha_+$. Define
        \begin{align}
            \clP^\pitchfork(\alpha_-,\alpha_+)\subset\clP(\alpha_-,\alpha_+)
        \end{align}
        to be the set of paths which consist entirely of $2$-pre-rigid elements and are transverse to \eqref{double-wall-to-ops}.
        \item Given $\gamma_\pm\in \clP^\pitchfork(\alpha_-,\alpha_+)$, define $\clP(\gamma_-,\gamma_+)$ to be the space of smooth paths in $\clP(\alpha_-,\alpha_+)$ from $\gamma_-$ to $\gamma_+$, i.e., smooth homotopies from $\gamma_-$ to $\gamma_+$ with fixed endpoints. Define
        \begin{align}
            \clP^\pitchfork(\gamma_-,\gamma_+)\subset\clP(\gamma_-,\gamma_+)
        \end{align}
        to be the set of homotopies which consist entirely of $2$-pre-rigid elements and are transverse to \eqref{double-wall-to-ops}.
    \end{enumerate}
\end{Definition}
For the convenience of the reader, we now summarize the consequences of Lemmas \ref{ker-with-zero}, \ref{sign-wall-codim} and \ref{emb-multisection-wall} which are relevant for the next lemma: $\clW'(N)$ has codimension $\ge 3$ in $\clC\clR_\bR(N)$, the image of $\clW^k_{2,h,\bR}(N)$ in $\clC\clR_\bR(N)$ has codimension $\ge 4$ for $k\ge 2$ and the image of $\clW^1_{2,h,\bR}(N)_\text{emb}$ has codimension $\ge 3$ in $\clC\clR_\bR(N)$. Moreover, both $\clW^k_\bC(N)$ and the image of $\clW^k_{2,h,\bC}(N)$ in $\clC\clR_\bC(N)$ have codimension $\ge 2$ for all $k\ge 1$.
\begin{Lemma}\label{double-op-generic}
    We have the following genericity statements.
    \begin{enumerate}[\normalfont(i)]
        \item $\clC\clR_\bR(N)_2\subset\clC\clR_\bR(N)$ is a Baire subset.
        \item $\clC\clR_\bC(N)_2\subset\clC\clR_\bC(N)$ is a path connected Baire subset.
        \item For $\alpha_\pm\in\clC\clR_\bR(N)_2$, $\clP^\pitchfork(\alpha_-,\alpha_+)\subset\clP(\alpha_-,\alpha_+)$ is a Baire subset.
        \item For $\gamma_\pm\in P^\pitchfork(\alpha_-,\alpha_+)$, $\clP^\pitchfork(\gamma_-,\gamma_+)\subset\clP(\gamma_-,\gamma_+)$ is a Baire subset.
    \end{enumerate}
\end{Lemma}
\begin{proof}
    This is a straightforward consequence of the Sard--Smale theorem combined with the codimension formulas/estimates from Lemmas \ref{ker-with-zero}, \ref{sign-wall-codim} and \ref{emb-multisection-wall}.
\end{proof}

\begin{Lemma}\label{compactness-for-double-cover-op}
    Fix $\alpha_\pm\in\clC\clR_\bR(N)_2$, $\gamma_\pm\in\clP^\pitchfork(\alpha_-,\alpha_+)$ and $\tilde\gamma\in\clP^\pitchfork(\gamma_-,\gamma_+)$. Consider $\tilde\gamma$ as smooth map
    \begin{align}\label{generic-homotopy}
        [-1,1]\times[-1,1]\to\clC\clR_\bR(N)
    \end{align}
    such that $\tilde\gamma(\cdot,\pm1)\equiv\gamma_\pm(\cdot)$ and $\tilde\gamma(\pm1,\cdot)\equiv\alpha_\pm$. Then, for any fixed $h\ge g$, the fiber product of \eqref{generic-homotopy} and \eqref{double-wall-to-ops} is a naturally oriented smooth $1$-manifold with boundary which is proper over $[-1,1]\times[-1,1]$ and has boundary exactly over $[-1,1]\times\pm1$.
\end{Lemma}

\begin{proof}
    The fact that the fiber product in question is a manifold with boundary (with boundary only over $[-1,1]\times\pm 1$) follows from the transversality assumption in the definition of $\clP^\pitchfork(\gamma_-,\gamma_+)$. The orientation comes from the standard orientation on $[-1,1]\times[-1,1]$ and the natural co-orientation from Definition \ref{natural-coor}. Thus, what remains to be proved is properness over $[-1,1]\times[-1,1]$ and this will be a consequence of the fact that all elements in the image of $\tilde\gamma$ are $2$-pre-rigid. To establish properness, let $p_n\to p$ be a convergent sequence in $[-1,1]\times[-1,1]$ and let $\tilde\alpha_n = (j_n,D_n,[\varphi_n:\Sigma'_n\to\Sigma])\in\clW^1_{2,h,\bR}(N)$ be a sequence such that we have $\tilde\gamma(p_n) = \alpha_n:= (j_n,D_n)$ for all $n\ge 1$, with limit $\tilde\gamma(p) = \alpha := (j,D)$. We will show that $\tilde\alpha_n$ has a subsequence converging to an element $\tilde\alpha\in\clW^1_{2,h,\bR}(N)$.
    In other words, we will use the assumption $\tilde\gamma\in\clP^\pitchfork(\gamma_-,\gamma_+)$ to argue that there is a limiting $(\bZ/2\bZ)$-anti-invariant multi-valued section associated to $\alpha = (j,D)$ which is defined over a branched cover of $\Sigma$ with smooth domain.
    
    Fix a Hermitian inner product on $N$ and let $\omega_0$ be a symplectic form on the unit ball $B_1(N)$ taming the almost complex structure $J_D$ (this exists by \cite[Lemma 3.6]{doan2021castelnuovo}). Since $\alpha_n\to\alpha$, it follows that $\omega_0$ tames $J_{D_n}$ also for $n\gg 1$.
    
    Let $\hat\sigma_n:\Sigma'_n\to N$ be the simple $J_{D_n}$-holomorphic map representing the class $2[\Sigma]$ corresponding to a section $\sigma_n\in\ker(D_n)_{\varphi_n}^-$ with $\|\sigma_n\|_{L^\infty} = \frac12$. By Gromov compactness, we can pass to a subsequence to assume that we have a limiting $J_D$-holomorphic stable map $\hat\sigma:\Sigma'\to N$ representing class $2[\Sigma]$ with $\Sigma'$ being of arithmetic genus $h$. Note that the image of each $\hat\sigma_n$ (and therefore of $\hat\sigma$) is invariant under the automorphism $\clI:(z,v)\mapsto(z,-v)$ of $N$. Define $\varphi = \pi\circ\hat\sigma:\Sigma'\to\Sigma$ and note that the corresponding section $\sigma\in\ker\varphi^*D$ satisfies $\|\sigma\|_{L^\infty} = \frac12$. Each effective component of $\Sigma'$ (i.e. an irreducible component on which $\hat\sigma$ is non-constant) must map with positive degree onto $\Sigma$ via $\varphi$. Thus, there are at most two effective components in $\Sigma'$. We now consider two cases.
    \begin{enumerate}[(a)]
        \item (There are two effective components $C_1,C_2\subset\Sigma'$.) Since $\varphi|_{C_i}:C_i\to\Sigma$ are both of degree $1$, they must be isomorphisms. Thus, $\hat s_1 := \hat\sigma|_{C_1}$ and $\hat s_2:=\hat\sigma|_{C_2}$ give us two sections $s_1,s_2\in\ker D$ at least one of which must be nonzero. If $\hat s_1\not\equiv 0$, then since the image of $\hat\sigma$ is connected invariant under $\clI$, it follows that $s_1$ must vanish at a point $z\in C_1$. This contradicts the $2$-pre-rigidity of $\alpha = (j,D)$.
        \item (There is a unique effective component $C\subset\Sigma'$.) Let $\psi:\tilde C\to C$ be the normalization. Since the image of $\hat\sigma\circ\varphi:\tilde C\to N$ is invariant under $\clI$, it follows that $0\ne\psi^*\sigma\in\ker D_{\varphi\circ\psi}^-$. Since $\alpha = (j,D)$ is $2$-pre-rigid, it follows that $\dim\ker D^-_{\varphi\circ\psi} = 1$ and the $J_D$-holomorphic map $\hat\sigma\circ\psi$ is an embedding because $\alpha$ necessarily lies in the complement of $\normalfont\clW^1_{2,h,\bR}(N)_\text{emb}$. This implies that $\psi$ is an isomorphism, $C$ is smooth and $\hat\sigma|_C$ is an embedding representing class $2[\Sigma]$. Since $\hat\sigma:\Sigma'\to N$ is the Gromov limit of the simple (in fact, embedded by $2$-pre-rigidity) $J_{D_n}$-curves $\hat\sigma_n:\Sigma'_n\to N$, Remark \ref{compactness-remark}(b) shows that $\Sigma'$ cannot have any ghost components. Thus, $\Sigma'$ is smooth and coincides with $C$. Moreover, $0\ne\sigma\in\ker D_\varphi^-$ and $\dim\ker D_\varphi^- = 1$.
    \end{enumerate}
    We have now shown that $\Sigma'$ is necessarily smooth. We conclude that the limit $\tilde\alpha = (j,D,[\varphi:\Sigma'\to\Sigma])$ lies in $\clW^1_{2,h,\bR}(N)$ as desired.
\end{proof}
\begin{Definition}[Linear wall crossing invariant]\label{lwc-double}
    Let $\alpha\in\clC\clR_\bR(N)_2$ be $2$-rigid and let $h\ge g$ be given. We then define the invariant
    \begin{align}
        \normalfont\text{w}_{2,h}(\alpha)\in\bZ
    \end{align}
    as follows. Set $r = (2h-2) - 2(2g-2)$. If $r<0$, we define $\normalfont\text{w}_{2,h}(\alpha) = 0$. 
    
    If $r\ge 0$, choose a point $\beta\in\clC\clR_\bC(N)_2$ and choose a path $\gamma\in\clP^\pitchfork(\alpha,\beta)$. Then, using Lemma \ref{compactness-for-double-cover-op} and the natural co-orientation from Definition \ref{natural-coor}, the fiber product of $\gamma$ with \eqref{double-wall-to-ops} is a compact oriented $0$-manifold and we define $\normalfont\text{w}_{2,h}(\alpha)$ to be the signed count of points in this fiber product. 
\end{Definition}
The following lemma shows that $\normalfont\text{w}_{2,h}(\alpha)$ does not depend on the auxiliary data therefore it is well-defined. 
\begin{Lemma}\label{lwc-double-well-def}
    $\normalfont\text{w}_{2,h}(\alpha)$ in Definition \ref{lwc-double} is independent of the choice of $\beta$ and $\gamma$.
\end{Lemma}
\begin{proof}
    It is enough to consider the case $r\ge 0$. Let us temporarily denote the quantity from Definition \ref{lwc-double} as $\text{w}_{2,h}(\alpha,\beta,\gamma)$ to emphasize the dependence on the choices $\beta$ and $\gamma$. 
    \begin{enumerate}[(i)]
        \item Fix the choice of $\beta$ and consider $\gamma,\gamma'\in\clP^\pitchfork(\alpha,\beta)$. Choose $\tilde\gamma\in\clP^\pitchfork(\gamma,\gamma')$ and apply Lemma \ref{compactness-for-double-cover-op} (combined with the fact that an oriented $1$-manifold with boundary has $0$ boundary points when counted with sign) to get
        \begin{align}
            \text{w}_{2,h}(\alpha,\beta,\gamma) = \text{w}_{2,h}(\alpha,\beta,\gamma').
        \end{align}
        Thus, we may write $\text{w}_{2,h}(\alpha,\beta,\gamma) = \text{w}_{2,h}(\alpha,\beta)$.
        \item Now, suppose $\beta,\beta'\in\clC\clR_\bC(N)_2$ are given. Choose a path $\delta$ in $\clC\clR_\bC(N)_2$ from $\beta$ to $\beta'$ -- such a path exists by Lemma \ref{double-op-generic} and it is, moreover, automatically in $\clP^\pitchfork(\beta,\beta')$. Now, given any $\gamma\in\clP^\pitchfork(\alpha,\beta)$, we can set $\gamma'\in\clP^\pitchfork(\alpha,\beta')$ to be the concatenation of $\gamma$ and $\delta$. Since there are no intersections of $\delta$ with \eqref{double-wall-to-ops}, it follows that
        \begin{align}
            \text{w}_{2,h}(\alpha,\beta) = \text{w}_{2,h}(\alpha,\beta,\gamma) = \text{w}_{2,h}(\alpha,\beta',\gamma') = \text{w}_{2,h}(\alpha,\beta').
        \end{align}
    \end{enumerate}
    Combining (i) and (ii), we conclude that $\text{w}_{2,h}(\alpha,\beta,\gamma)$ is independent of the choices of $\beta$ and $\gamma$.
\end{proof}

\subsection{Proof of Theorem \ref{thm:Gr-invariance}}\label{Gr-invariance-proof}
The proof of Theorem \ref{thm:Gr-invariance} is an application of the arguments of \textsection\ref{lwc-development} combined with Theorems \ref{compactness-result} and \ref{main-bifurcation}. 

The primitive class $A\in \overline{H}_2(X,\bZ)$ remains fixed for the whole argument. As explained after Definition \ref{emb-wall}, as long as $J$ is away from the subset $\normalfont\clW_\text{emb}\subset\clJ(X,\omega)$ of codimension $2$, every $J$-holomorphic stable map to $X$ is either an embedded $J$-curve or a multiple cover\footnote{In a slight abuse of terminology, we will be referring to any $J$-holomorphic stable map which is not an embedded $J$-curve as a ``multiple cover" (even if it maps with degree $1$ onto its image but has a ghost component).} (with possibly nodal domain) of an embedded $J$-curve. For such $J$, we record the following observation which will be used throughout the proof below. Since $A$ is primitive, any stable $J$-holomorphic map of class $dA$ with $d\le 2$ either represents an embedded $J$-curve of class $dA$ or a (possibly nodal) multiple cover $\varphi:C'\to C$ (with $d'=\deg(\varphi)\le d$) of an embedded $J$-curve $C$ such that $d'[C] = dA$ and $\varphi$ is not an isomorphism. We now describe the non-trivial multiple covers $\varphi$ which can appear. When $d'=1$ (in this case, $d$ is either $1$ or $2$), the domain $C'$ must have exactly one effective, i.e., non-constant component (which necessarily maps isomorphically onto $C$ under $\varphi$) and every other component is a ghost, i.e., non-effective component. Moreover, $C'$ has at least one ghost component since $\varphi$ is not an isomorphism. When $d'=2$ (in this case, $d = 2$), $C'$ has at most two effective components and $[C] = A$. Further, if $C'$ has only one effective component, then restricting $\varphi$ to (the normalization of) this component yields a (possibly branched) double cover of $C$. If, instead, $C'$ has two effective components, then $\varphi$ maps each of these isomorphically onto $C$.

Let $J_\pm\in\clJ(X,\omega)$ be super-rigid almost complex structures. We need to show that for all $h\ge 0$ we have
\begin{align}
    \text{Gr}_{2A,h}(X,\omega,J_-) = \text{Gr}_{2A,h}(X,\omega,J_+)
\end{align}
where $\text{Gr}_{2A,h}(X,\omega,J_\pm)$ is the right side of \eqref{Gr-def} for $J = J_\pm$. We will accomplish this by studying the bifurcations that occur along a suitably chosen path $\gamma:[-1,1]\to\clP(J_-,J_+)$. Before this, we need some preparatory facts analogous to Lemmas \ref{ker-with-zero}, \ref{emb-multisection-wall}, \ref{double-op-generic} and \ref{compactness-for-double-cover-op}.

\begin{Definition}
    Let $h\ge g\ge 0$ be integers.
    \begin{enumerate}[\normalfont(i)]
        \item Define $\clJ'_{A,g}\subset\clJ(X,\omega)$ to be the set of all $\omega$-compatible almost complex structures $J$ on $X$ for which there exists an embedded $J$-curve $C\subset X$ of genus $g$ in class $A$, a point $z\in C$ and $0\ne\sigma\in\ker D^N_{C,J}$ with $\sigma(z)= 0$.
        \item Define $\clW_{2,h}(\clJ_{A,g})\subset\clJ(X,\omega)$ to be the set of all $\omega$-compatible almost complex structures $J$ on $X$ for which there exists an embedded $J$-curve $C\subset X$ of genus $g$ in class $A$, a holomorphic branched double cover $\varphi:C'\to C$ with $C'$ being smooth of genus $h$ and $0\ne\sigma\in\ker(D^N_{C,J})^-_{\varphi}$ such that the corresponding $J_{D^N_{C,J}}$-holomorphic map $\hat\sigma:C'\to N_C$ defined using \eqref{eqn:curve-section} is not an embedding.
    \end{enumerate}
\end{Definition}
\begin{Lemma}[Analogue of Lemmas \ref{ker-with-zero} and \ref{emb-multisection-wall}]\label{nl-emb-multisection-wall}
    We have the following codimension estimates.
    \begin{enumerate}[\normalfont(i)]
        \item $\clJ'_{A,g}\subset\clJ(X,\omega)$ has codimension $\ge 3$ for all $g\ge 0$.
        \item $\clW_{2,h}(\clJ_{A,g})\subset\clJ(X,\omega)$ has codimension $\ge 3$ for all $h\ge g\ge 0$.
    \end{enumerate}
\end{Lemma}
\begin{proof}
    We first define the relevant universal moduli spaces and reduce the problem of estimating the codimension to proving a rank lower bound for the linearization of the defining equations of these universal moduli spaces. The proofs of the rank estimates from Lemmas \ref{ker-with-zero} and \ref{emb-multisection-wall} (which prove the analogous results for Cauchy--Riemann operators) apply verbatim once we combine them with \cite[Lemma 5.27]{Wendl-19} which shows that deformations in the space of Cauchy--Riemann operators can be realized by deformations in the space of $\omega$-compatible almost complex structures.
\end{proof}
\begin{Lemma}[Analogue of Lemma \ref{double-op-generic}]
    The subset $\clB^*\subset\clP(J_-,J_+)$ of paths $\gamma:[-1,1]\to\clJ(X,\omega)$ with $\gamma(\pm1) = J_\pm$ which are disjoint from $\clJ'_{A,g}$ and $\clW_{2,h}(\clJ_{A,g})$ for all $h\ge g\ge 0$ is a Baire subset.
\end{Lemma}
\begin{proof}
    This is a direct consequence of a Sard--Smale argument applied to the codimension estimates from Lemma \ref{nl-emb-multisection-wall}.
\end{proof}
For the statement of the next lemma, recall the Baire set $\clB'$ from Lemma \ref{baire-transverse}.
\begin{Lemma}[Analogue of Lemma \ref{compactness-for-double-cover-op}]\label{nl-compactness-for-double-covers}
    Fix a path $\gamma\in\clB^*\cap\clB'\subset\clP(J_-,J_+)$ and integers $0\le g\le h$. 
    \begin{enumerate}[\normalfont(i)]
        \item Consider all tuples of the form $(t,\Sigma,\varphi:\Sigma'\to\Sigma)$ where $t\in[-1,1]$,  $\Sigma\subset X$ is an embedded $\gamma(t)$-curve of genus $g$ in class $A$ and $(\Sigma',\varphi)\in\overline\clM_h(\Sigma,2)$ with $\Sigma'$ being smooth. Among these, the tuples satisfying
        \begin{align}\label{eqn:discrete-set}
            \dim\ker D^N_{\Sigma,\gamma(t)} = 0,\quad\dim\ker\varphi^*D^N_{\Sigma,\gamma(t)} = 1
        \end{align}
        form a finite set.
        \item Let $t_n\to t$ in $[-1,1]$ be a convergent sequence and suppose $\Sigma'_n\subset X$ is a sequence of embedded $\gamma(t_n)$-curves of genus $h$ in class $2A$ which converge to a $\gamma(t)$-holomorphic stable map $\varphi:\Sigma'\to\Sigma\subset X$ in the Gromov topology with $\Sigma$ being an embedded $\gamma(t)$-curve of genus $g$ in class $2A/d$ for some $d\in\{1,2\}$ and $(\Sigma',\varphi)\in\overline\clM_h(\Sigma,d)$. 
    
        Then, $\Sigma'$ is necessarily smooth. If $d=1$, then $\varphi$ is an isomorphism. If $d=2$, then the operator $D^N_{\Sigma,\gamma(t)}$ is an isomorphism and we have 
        \begin{align}
            \dim\ker \varphi^*D^N_{\Sigma,\gamma(t)} = 1.
        \end{align}
    \end{enumerate}
\end{Lemma}
\begin{proof}
    The key ingredients of the proof are Theorem \ref{compactness-result} and Remark \ref{compactness-remark}(b).
    \begin{enumerate}[(i)]
        \item Since $\gamma\in\clB'$, it follows that the set \eqref{eqn:discrete-set} is discrete. It will therefore suffice to prove compactness. To this end, suppose that $(t_n,\Sigma_n,\varphi_n:\Sigma'_n\to\Sigma_n)$ is a sequence of points in this set with $t_n\to t$. As observed at the beginning of \textsection\ref{Gr-invariance-proof}, the sequence $\Sigma_n$ has (after passing to a subsequence) a Gromov limit $\Sigma$ which is either an embedded $\gamma(t)$-curve of the same genus $g$ or an embedded $\gamma(t)$-curve of lower genus with ghost components attached.
        By Remark \ref{compactness-remark}(b), the latter situation is impossible. Now, by identifying the normal bundles $N_{\Sigma_n}\to\Sigma_n$ for $n\gg 1$ with the fixed smooth vector bundle $N_\Sigma\to\Sigma$, we can repeat the argument of Lemma \ref{compactness-for-double-cover-op} to conclude the proof (the condition $\gamma\in\clB'\cap\clB^*$ is the analogue of the $2$-pre-rigidity condition used in that proof).
        \item First, suppose $\Sigma$ itself is of class $2A$. Then, Remark \ref{compactness-remark}(b) shows that $d=1$, $\varphi$ is an isomorphism and, in particular, that $\Sigma'$ is smooth. For the remainder suppose that $\Sigma$ is of class $A$ and $d = 2$. Theorem \ref{compactness-result} now yields a section $s\in\ker\varphi^*D^N_{\Sigma,\gamma(t)}$ which is not pulled back from $\Sigma$. Let $\hat s:\Sigma'\to N_\Sigma$ be the associated $J_{D^N_{\Sigma,\gamma(t)}}$-holomorphic map.
        
        In accordance with the observation made at the beginning of \textsection\ref{Gr-invariance-proof}, since $(\Sigma',\varphi)\in\overline\clM_h(\Sigma,2)$, it follows that $\Sigma'$ has at most two effective components (i.e. irreducible components of $\Sigma'$ on which $\varphi$ is non-constant). We now split into two cases.
        \begin{enumerate}[(a)]
            \item (There are two effective components $C_1,C_2\subset\Sigma'$.) Since $\varphi|_{C_i}$ has degree $1$, it must be an isomorphism for $i=1,2$. Thus, $s|_{C_i}$ yields a section $s_i\in\ker D^N_{\Sigma,\gamma(t)}$ for $i=1,2$. Since $s$ is not pulled back from $\Sigma$, it follows that $s_1\not\equiv s_2$. Since $C_1$ can be connected to $C_2$ within $\Sigma'$ by a (possibly empty) chain of ghost components, it also follows that there must be a point $z\in\Sigma$ such that $s_1(z) = s_2(z)$. The section $s_1-s_2\in\ker D^N_{\Sigma,\gamma(t)}$ now shows that $\gamma(t)\in\clJ'_{A,g}$. This contradicts $\gamma\in\clB^*$.
            \item (There is a unique effective component $C\subset\Sigma'$.) Let $\psi:\tilde C\to C$ be the normalization and set $\varphi' = \varphi\circ\psi:\tilde C\to\Sigma$. Since $\gamma(t)\not\in\clW_{2,h}(\clJ_{A,g'})$ for $g' = \text{genus}(\tilde C)$, we deduce that $\hat s\circ\psi:\tilde C\to N_\Sigma$ is an embedding. This implies that $\psi$ must be injective and therefore an isomorphism. Thus, $C$ is smooth and $\hat s|_C$ is an embedding. Since $\ker (D^N_{\Sigma,\gamma(t)})_{\varphi'}^-\ne 0$, it follows from $\gamma\in\clB'$ that this vector space must in fact be $1$-dimensional and that $D^N_{\Sigma,\gamma(t)}$ is an isomorphism. Using the inverse function theorem, we can now obtain a sequence $\Sigma_n\subset X$ of embedded $\gamma(t_n)$-curves of genus $g$ in class $A$ converging to $\Sigma$. Arguing as in \cite[Appendix B]{Wendl-19}, we can now identify tubular neighborhoods of $\Sigma_n$ for $n\gg 1$ with the normal bundle of $\Sigma$ equipped with the pullback of the almost complex structure $\gamma(t_n)$. Rescaling the images of the embedded curves $\Sigma'_n$ in $N_\Sigma$ to have distance $\frac12$ from the zero section gives a new sequence of embedded curves which are holomorphic with respect to the rescaled versions of $\gamma(t_n)$. The rescalings of the $\gamma(t_n)$ converge to $J_{D^N_{\Sigma,\gamma(t)}}$ and moreover, all these almost complex structures are tamed by a symplectic form for $n\gg 1$ (\cite[Lemma 3.6]{doan2021castelnuovo}). Therefore, after passing to a subsequence, we can extract a Gromov limit given by a $J_{D^N_{\Sigma,\gamma(t)}}$-holomorphic map $\hat s':\Sigma'\to N_\Sigma$ whose image doesn't lie in the zero section. Since $C$ is the only effective component of $\hat s'$ (since a holomorphic mapping into any normal fiber must be constant), it follows that $\hat s'$ a scalar multiple of $\hat s$. We have therefore exhibited $\hat s$ as a limit of embedded pseudo-holomorphic curves of genus $h$ and class $2[\Sigma]$. Recalling that $\hat s|_C$ is an embedding and applying Remark \ref{compactness-remark}(b) shows that we must have $C = \Sigma'$.
        \end{enumerate}
        We have now succeeded in showing that $\Sigma'$ is smooth. The remaining facts concerning the operator $D^N_{\Sigma,\gamma(t)}$ have already been proved during the course of (b) above.
    \end{enumerate}
\end{proof}

\begin{remark}
Although the proof of Lemma \ref{nl-compactness-for-double-covers} is formally parallel to the proof of Lemma \ref{compactness-for-double-cover-op}, it is crucial to note that in case (ii) we use Theorem \ref{compactness-result} to construct the multi-section $s$ over $\Sigma$ because the rescaling argument from \cite[Appendix B]{Wendl-19} or Remark \ref{compactness-remark}(a) cannot be used unless we already know that $D^N_{\Sigma,\gamma(t)}$ is an isomorphism.
\end{remark}

We can now complete the proof of Theorem \ref{thm:Gr-invariance}. Choose a path $\gamma$ from the Baire subset $\clB^*\cap\clB'\cap\clB''\subset\clP(J_-,J_+)$. (Recall the definition of the Baire set $\clB''$ from Lemma \ref{baire-taylor}.) Define the moduli space
\begin{align}\label{2A-h-mod}
    \clM_h^*(X,\gamma,2A)\to[-1,1]
\end{align}
consisting of pairs $(t,\Sigma')$, where $t\in[-1,1]$ and $\Sigma'\subset X$ is an embedded $\gamma(t)$-curve of genus $h$ in class $2A$. For each $0\le g\le h$, also consider the moduli spaces
\begin{align}\label{A-g-mod}
    \clM_g(X,\gamma,A)\to[-1,1]
\end{align}
consisting of pairs $(t,\Sigma)$, where $t\in[-1,1]$ and $\Sigma\subset X$ is an embedded $\gamma(t)$-curve curve of genus $g$ in class $A$. Both $\clM_h^*(X,\gamma,2A)$ and $\clM_g(X,\gamma,A)$ are oriented $1$-manifolds with boundary over $t = \pm 1$ since $\gamma\in\clB'$. Moreover, since $A$ is primitive and the path $\gamma$ misses $\clW_{\text{emb}}$, the moduli space $\clM_g(X,\gamma,A)$ is compact by Remark \ref{compactness-remark}(b). Since $\gamma\in\clB'\cap\clB^*$, Lemma \ref{nl-compactness-for-double-covers} shows that the set of points in $(t,\Sigma,\varphi:\Sigma'\to\Sigma)$ with $(t,\Sigma)\in\clM_g(X,\gamma,A)$ for some $g\le h$, $(\Sigma',\varphi)\in\clM_h(\Sigma,2)$ with $\Sigma'$ smooth and
\begin{align}\label{wall-points-along-path}
    \ker D^N_{\Sigma,\gamma(t)} = 0,\quad \ker\varphi^*D^N_{\Sigma,\gamma(t)} = 1
\end{align}
forms a finite set. Moreover, $\gamma$ is transverse to the map
\begin{align}\label{relevant-wall}
    \clM^2_{{\bf b}^r,\bZ/2\bZ}(\clM_0^*(X,A);\theta,\theta^{\oplus 1+2r})\to\clJ(X,\omega)    
\end{align}
at any such point $(t,\Sigma,\varphi:\Sigma'\to\Sigma)$, where we have set $r = 2h-2-2(2g-2)\ge 0$, ${\bf b}^r = ((2),\ldots,(2))$ and $\theta = $ the sign representation of $\bZ/2\bZ$.

Enumerate the times $t$ for which there exists $\varphi:\Sigma'\to\Sigma$ satisfying \eqref{wall-points-along-path} as
\begin{align}
    -1 < t_1 < \cdots < t_k < 1
\end{align}
and, for each $1\le i\le k$, enumerate the intersections with the maps \eqref{relevant-wall} as
\begin{align}
    p_{i,j} = (t_i,\Sigma_{i,j}\subset X,\varphi_{i,j}:\Sigma'_{i,j}\to\Sigma_{i,j})
\end{align}
for $1\le j\le \ell_i$, where $\ell_i$ is a positive integer depending on $i$. Theorem \ref{main-bifurcation} now provides numbers $\epsilon_i>0$ and neighborhoods $\clO_{i,j}$ of $\varphi_{i,j}:\Sigma'_{i,j}\to\Sigma_{i,j}\subset X$ and by shrinking these further we can assume that the products
\begin{align}
    \clU_{i,j} = (t_i-\epsilon_i,t_i+\epsilon_i)\times\clO_{i,j}
\end{align}
are all pairwise disjoint. We now make a series of claims.
\begin{enumerate}
    \item The $\bZ$-valued function
    \begin{align}
        (t,\Sigma)\mapsto\text{w}_{2,h}(D^N_{\Sigma,\gamma(t)})
    \end{align}
    defined on $\clM_g(X,\gamma,A)$ away from $t = t_1,\ldots,t_k$ is locally constant.
    \begin{proof}
        Indeed, fix any point $(t,\Sigma)$ such that $t\ne t_i$ for all $1\le i\le k$. Then, $D^N_{\Sigma,\gamma(t)}$ is $2$-rigid and the same is true for all nearby points $(t',\Sigma')$. Thus, any nearby $(t',\Sigma')$ can be connected to $(t,\Sigma)$ by a path of curves with $2$-rigid normal deformation operators and so, $\text{w}_{2,h}(D^N_{\Sigma,\gamma(t)}) = \text{w}_{2,h}(D^N_{\Sigma',\gamma(t')})$.
    \end{proof}
    \item The moduli space $\clM_h^*(X,\gamma,2A)\setminus\bigcup_{i,j}\clU_{i,j}$ is compact. Moreover, the restriction of \eqref{2A-h-mod} over the complement of $\{t_i\}_{1\le i\le k}$ is proper.
    \begin{proof}
        Both assertions follow directly from Lemma \ref{nl-compactness-for-double-covers}.
    \end{proof}
    \item The expression \eqref{Gr-def} is well-defined with $J = \gamma(t)$ provided $t\in[-1,1]\setminus\{t_1,\ldots,t_k\}$ is a regular value of \eqref{2A-h-mod} and each of the maps \eqref{A-g-mod}. If $t_i<t'\le t''< t_{i+1}$ are two such regular values (for some $1\le i<k$), then we have the equality
    \begin{align}\label{loc-const-Gr}
        \text{Gr}_{2A,h}(X,\omega,\gamma(t')) = \text{Gr}_{2A,h}(X,\omega,\gamma(t'')).
    \end{align}
    \begin{proof}
        Claim (2) shows that the first sum in the expression \eqref{Gr-def} is well-defined for such a $t$. The remaining sums in \eqref{Gr-def} are well-defined by Claim (1). For regular values $t'<t''$ in $(t_i,t_{i+1})$, we can now use the cobordisms given by restricting \eqref{2A-h-mod} and \eqref{A-g-mod} to $[t',t'']$ and the local constancy statement from Claim (1) to deduce \eqref{loc-const-Gr}.
    \end{proof}
    \item In the previous claim, if we instead have $t_i-\epsilon_i<t'<t_i<t''<t_i+\epsilon_i$ for some $1\le i\le k$, then \eqref{loc-const-Gr} still holds.
    \begin{proof}
        By Claim (3), we can assume $t',t''$ are arbitrarily close to $t_i$. For small $|t-t_i|$ and $\Sigma\in\{\Sigma_{i,j}:1\le j\le \ell_i\}$, define $\Sigma(t)$ to be the unique embedded $\gamma(t)$-curve close to $\Sigma$ (obtained by noting that $\ker D^N_{\Sigma,\gamma(t_i)} = 0$ and applying the inverse function theorem). For such a $\Sigma$, define the set $I_\Sigma(i) = \{1\le j\le\ell_i:\Sigma = \Sigma_{i,j}\}$, and note that we have
            \begin{align}\label{net-change}
                \sum_{j\in I_\Sigma(i)}\text{sgn}(\gamma,t_i;p_{i,j})
                 =\text{sgn}(D^N_{\Sigma,\gamma(t_i)})\cdot(\text{w}_{2,h}(D^N_{\Sigma(t''),\gamma(t'')}) - \text{w}_{2,h}(D^N_{\Sigma(t'),\gamma(t')}))
            \end{align}
            where we check orientations by using the identification
            \begin{align}
                \det(D^N_{\Sigma,\gamma(t_i)})\otimes\det(D^N_{\Sigma,\gamma(t_i)})_{\varphi_{i,j}}^- = \det(\varphi_{i,j}^*D^N_{\Sigma,\gamma(t_i)}) 
            \end{align}
            for each $j\in I_\Sigma(i)$. The quantities $\text{sgn}(\gamma,t_i;p_{i,j})$ are as in Definition \ref{wc-sign-def}. Now recall that the neighborhoods $\clU_{i,j}$ were constructed by applying Theorem \ref{main-bifurcation} to each point $p_{i,j}\in\clU_{i,j}$ and note that the isotropy group $\Gamma$ is $\bZ/2\bZ$ in all cases. Restrict \eqref{2A-h-mod} to $[t',t'']$ and note that this gives a compact $1$-manifold with boundary over the complement of $\bigcup_{i,j}\clU_{i,j}$ by Claim (2). Thus, as in Claim (3), this portion contributes equally to both sides of \eqref{loc-const-Gr}.
            
            For the remaining contribution (coming from $\bigcup_{i,j}\clU_{i,j}$), we see from Theorem \ref{main-bifurcation} that the first sum in the definition of $\text{Gr}_{2A,h}$ changes by the negative of the left side of \eqref{net-change} as we go from $t'$ to $t''$, while the second sum in the definition of $\text{Gr}_{2A,h}$ changes by the right side of \eqref{net-change} when we go from $t'$ to $t''$. These two changes cancel each other out and so, we again have \eqref{loc-const-Gr} as desired.
    \end{proof}
\end{enumerate}
Combining claims (3) and (4) and noting that $\pm1$ are regular values of \eqref{2A-h-mod} as well as each of the maps \eqref{A-g-mod} completes the proof of Theorem \ref{thm:Gr-invariance}.
\section{Application to the Gopakumar--Vafa invariants}\label{sec_GV}

In this section, we will show that the invariant from Theorem \ref{thm:Gr-invariance} coincides with the Gopakumar--Vafa invariant for genus $0$. The comparison is carried out in \textsection\ref{comparison-to-GV} and is an application of the discussion of obstruction bundles in \textsection\ref{sec_euler}. The main result of \textsection\ref{sec_euler} (Theorem \ref{euler-change-thm}) is stated in greater generality than is needed for the application in \textsection\ref{comparison-to-GV} as it may be of independent interest.
\subsection{Obstruction bundles and their Euler classes}\label{sec_euler}

Let $C$ be a smooth curve of genus $g$ and let $h\ge g$ and $d\ge 2$ be integers such that the integer
\begin{align}
    r = (2h-2)-d(2g-2)    
\end{align}
is non-negative.
The moduli stack $\overline\clM_h(C,d)$, consisting of stable holomorphic maps $f:C'\to C$ with the arithmetic genus of $C'$ being $h$ and $f_*[C'] = d[C]$, has complex virtual dimension given by $r$. Moreover, the non-empty open substack $\clM_h(C,d)\subset\overline\clM_h(C,d)$, consisting of those $[f:C'\to C]$ for which the domain $C'$ is smooth, is a complex orbifold of complex dimension $r$.
\begin{Definition}[Virtual fundamental class]
    Define the class 
    \begin{align}
        \normalfont[\overline\clM_h(C,d)]^\text{vir} \in \check{H}_{2r}(\overline\clM_h(C,d),\bQ)
    \end{align}
    to be the virtual fundamental class of this moduli space as in \cite[Section 9]{Pardon-VFC}. Note that $\overline\clM_h(C,d)$ is algebraic and so, we do not actually need to distinguish between its \v{C}ech and singular (co)homologies.
\end{Definition}
Now, let $N$ be a rank $2$ complex vector bundle on $C$ with $\deg(N) = 2g - 2$. We are interested in $\bR$-linear Cauchy--Riemann operators on $N$ because $N$ naturally arises as the normal bundle of $J$-curves of genus $g$ in a symplectic Calabi--Yau $3$-fold. Then, for any $\bR$-linear Cauchy--Riemann operator $D$ on $N$, we have
\begin{align}
    \text{ind}(D) =2\deg(N) + (2-2g)\cdot\text{rank}(N) = 0
\end{align}
by the Riemann--Roch theorem. For the Cauchy--Riemann operator $D$ and a point $[f:C'\to C]\in\overline\clM_h(C,d)$, we can consider the vector spaces
    \begin{align}
        \label{new-ker}\ker(f^*D)' &:= \coker(\ker D\xrightarrow{f^*}\ker(f^*D)) \\
        \label{new-coker}\coker(f^*D)' &:= \coker(\coker D\xrightarrow{f^*}\coker(f^*D)).
    \end{align}
Note that these spaces can also be defined as subspaces of $\ker(f^*D)$ and $\coker(f^*D)$, when $C'$ is smooth, using Remark \ref{invariant-sections} and Lemma \ref{operator-decomp}.
    
\begin{Definition}[Cokernel/obstruction bundle]\label{coker-bundle-def}
The operator $D$ is said to be \emph{$(d,h)$-rigid} if and only if for all $[f:C'\to C]\in\overline\clM_h(C,d)$, we have
    \begin{align}\label{no-new-ker}
        \ker(f^*D)' = 0.
    \end{align}
In this case, the assignment
    \begin{align}\label{new-coker-bundle}
        [f:C'\to C]\mapsto (\clN^{d,h}_{C,D})_f := \coker(f^*D)'
    \end{align}
defines a real orbi-bundle\footnote{For a proof that $\clN^{d,h}_{C,D}$ is an orbi-bundle, see Appendix \ref{obstruction-bundle-construction-details}.} $\clN_{C,D}^{d,h}\to\overline\clM_h(C,d)$ of rank $2r$, which is called the \emph{cokernel/obstruction bundle}.
\end{Definition}

\begin{remark}
    It is important to note that the condition of $(d,h)$-rigidity \emph{does not} require $\ker(f^*D) = 0$, and thus, it could be satisfied even if the operator $D$ has a non-trivial kernel on $C$. This extra generality is important for the proof of Theorem \ref{thm_GWGV} (comparison to Gopakumar--Vafa invariants) as it allows the obstruction bundle to remain well-defined even when a birth-death bifurcation occurs. 
\end{remark} 

Next, define the canonically oriented \emph{determinant line} of $(f^*D)'$ to be
    \begin{align}
        \label{eqn:modified-orientation-line}
        \det(f^*D)' &= \det\ker(f^*D)'\otimes\left(\det\coker(f^*D)'\right)^*\\
        \label{eqn:orientation-decouple}&=\det(f^*D)\otimes\det(D)^*
    \end{align}
where $\det(D)$ and $\det(f^*D)$ are the determinant lines of the Cauchy--Riemann operators $D$ and $f^*D$, which are canonically oriented by considering the spectral flow of a homotopy between $D$ and a $\bC$-linear Cauchy--Riemann operator. If $D$ is $(d,h)$--rigid, this gives rise to an orientation of the corresponding cokernel bundle. Accordingly, we have a well-defined \emph{virtual Euler number} of the cokernel bundle
    \begin{align}\label{virt-eu-num}
        e_{d,h}(C,D) = \langle \normalfont[\overline\clM_h(C,d)]^\text{vir}, e(\clN^{d,h}_{C,D})\rangle\in\bQ.
    \end{align}
More generally, suppose $P$ is a topological space and $\clD = \{D_\lambda\}_{\lambda\in P}$ is a family of real Cauchy--Riemann operators
\begin{equation}
    D_\lambda: \Omega^0(C, N) \to \Omega^{0,1}(C, N)
\end{equation}
parametrized by $P$. The family $\clD = \{D_\lambda\}_{\lambda\in P}$ is called \emph{continuous} if the coefficients of $\{D_\lambda\}_{\lambda\in P}$ (and all their derivatives in the $C$ direction) are continuous as functions on $P\times C$ once we identify the space of real Cauchy--Riemann operators with $\Omega^{0,1}(C, \text{Hom}_{\mathbb{R}}(N,N))$ by choosing a reference operator. The notion of a \emph{smooth} family $\clD = \{D_\lambda\}_{\lambda\in P}$ is defined similarly when $P$ is a smooth manifold. It is easy to see that this notion of continuity/smoothness is well-defined.
For a continuous family $\clD$, let $\hat{\Delta}\subset\overline\clM_h(C,d)\times P$ be the closed substack over which \eqref{no-new-ker} fails, with $\clD$ in place of $D$. Then, the assignment \eqref{new-coker-bundle}, with $\clD$ in place of $D$, gives a well-defined oriented rank $2r$ (orbi-)vector bundle
    \begin{align}\label{families-version-cokernel-bundle}
        \clN^{d,h}_{C,\clD}\to(\overline\clM_h(C,d)\times P)\setminus\hat{\Delta}
    \end{align}
which has an Euler class $e(\clN^{d,h}_{C,\clD})\in \check{H}^{2r}((\overline\clM_h(C,d)\times P)\setminus\hat{\Delta},\bQ)$. For a proof that $\hat\Delta$ is closed and $\clN^{d,h}_{C,\clD}$ is an orbi-bundle, see Appendix \ref{obstruction-bundle-construction-details}.

\begin{remark}\label{further-orientation-explanation}
    The determinant line defined in \eqref{eqn:modified-orientation-line}--\eqref{eqn:orientation-decouple} deserves some more explanation. We have the relation
    \begin{equation}
        \det(f^*D) = \det(f^*D)' \otimes \det(D),
    \end{equation}
    which means that to compute the Gromov--Witten contribution of a super-rigid curve $C$ with normal deformation operator $D$, one should multiply \eqref{virt-eu-num} by the sign \eqref{count-sign-def} of $C$.
\end{remark}

\begin{Definition}[Simple wall crossing]\label{simple-linear-wall-crossing}
    Let $\clD = \{D_t\}_{t\in [-1,1]}$ be a smooth family of real Cauchy--Riemann operators on $N$. We say that $\clD$ is a \emph{simple wall crossing} of type $(d,h)$ at $t = 0$ if $\ker D_0 = 0$ and there is a finite set $\Delta\subset\clM_h(C,d)$ such that 
\begin{enumerate}[\normalfont(i)]
    \item condition $\eqref{no-new-ker}$ holds on $(\overline\clM_h(C,d)\times[-1,1])\setminus(\Delta\times 0)$, and so, the vector bundle $\clN^{d,h}_{C,\clD}$ is well-defined here;
    \item for each $[f:C'\to C]\in\Delta$, the composition
    \begin{equation}\label{linear-wc-iso}
    \begin{aligned}
        T_f\clM_h(C,d)\oplus\bR_t&\to\normalfont\text{Hom}(\ker(f^*D_0),\coker(f^*D_0))\\ &\xrightarrow{\simeq}\normalfont\text{Hom}(\ker(f^*D_0)',\coker(f^*D_0)')
    \end{aligned}
    \end{equation}
    is an isomorphism, where the first map is a linearized map as in Definition \ref{linearize-family-op} while second isomorphism comes from the assumption that $\ker(D_0) = 0$, which also implies $\coker(D_0) = 0$ for index reasons.
\end{enumerate}
\end{Definition}
\begin{remark}
    Note that the set $\Delta$ in Definition \ref{simple-linear-wall-crossing} is assumed to be a subset of the locus $\clM_h(C,d)$ parametrizing holomorphic maps with smooth domains. This is consistent with equation \eqref{eqn_intro_delta} in Theorem \ref{thm_int_2}. This assumption is sufficient for the application in Section \ref{comparison-to-GV}.
\end{remark}
    
    By counting dimensions and indices, it follows that we have
    \begin{align}
        \dim_\bR\ker(f^*D_0)' &= 1 \\ 
        \dim_\bR\coker(f^*D_0)' &= 1 + 2r
    \end{align}
    over the finite set $\Delta$.
    Both sides of \eqref{linear-wc-iso} are canonically oriented by combining the discussion surrounding equation \eqref{hom-space-orientation} with the orientation on $\det(f^*D)'$. Thus, to each $[f:C'\to C]\in\Delta$, we can assign a well-defined number
    \begin{align}
        \normalfont\text{sgn}(\clD,[f:C'\to C])\in\{+1,-1\}.
    \end{align}
    which is the sign of the isomorphism \eqref{linear-wc-iso}.

We want to compute the change in the virtual Euler number $e_{d,h}(C,D_t)$ as we go from $t<0$ to $t>0$ for a simple wall crossing $\clD$ in terms of isotropy groups of the points in $\Delta$ and the signs $\{\text{sgn}(\clD,p)\}_{p\in\Delta}$. Our main result is the following.
\begin{Theorem}[Change in Euler number]\label{euler-change-thm}
    In the situation of Definition \ref{simple-linear-wall-crossing},
    \begin{align}
        [-1,1]\setminus\{0\}&\to\bQ \\
        t&\mapsto e_{d,h}(C,D_t)
    \end{align}
    is a locally constant function. Let $e_{d,h}(C,D_\pm)$ denote the value $e_{d,h}(C,D_t)$ for $\pm t>0$. Then, we have the identity
    \begin{align}\label{euler-change-formula}
        \normalfont e_{d,h}(C,D_+) - e_{d,h}(C,D_-) = \sum_{p\in\Delta}\frac{2\cdot\text{sgn}(\clD,p)}{|\text{Aut}(p)|}.
    \end{align}
\end{Theorem}
The local constancy assertion is obvious. The remaining assertion will be proved in the next two lemmas. We continue to use the notation from Definition \ref{simple-linear-wall-crossing}.
\begin{Lemma}[Local to global]\label{reduction-to-local-contributions}
    For each point $p\in\Delta$, choose an open neighborhood \begin{align}
        V_p\Subset\clM_h(C,d)\times(-1,1)
    \end{align}
    whose closure\footnote{$A\Subset B$ denotes the assertion that $A\subset B$ and the closure of $A$ in $B$ is compact.} has smooth orbifold boundary $\partial V_p$. Endowing $\partial V_p$ with the boundary orientation from $V_p$ and assuming that the closures of $\{V_p\}_{p\in\Delta}$ are pairwise disjoint, we have the identity
    \begin{align}
        e_{d,h}(C,D_+) - e_{d,h}(C,D_-) = \sum_{p\in\Delta}\langle[\partial V_p], e(\clN^{d,h}_{C,\clD})\rangle
    \end{align}
    where $[\partial V_p]$ denotes the standard orbifold fundamental class for each $p\in\Delta$.
\end{Lemma}
\begin{proof}
    Introduce the notation $\clM_{I} := \overline\clM_h(C,d)\times[-1,1]$. It will be enough to show that the identity
    \begin{align}
        [\overline\clM_h(C,d)]^\text{vir}\times\{+1\} - [\overline\clM_h(C,d)]^\text{vir}\times\{-1\} = \sum_{p\in\Delta}[\partial V_p]
    \end{align}
    holds in $H_*(\clM_{I}\setminus\bigcup_{p\in\Delta}V_p;\bQ)$. This will follow from \cite[Lemmas 5.2.4, 5.2.5 and 5.2.6]{Pardon-VFC} applied to the space
    \begin{align}\label{aux-cobordism}
        \clM_{I}\setminus\bigcup_{p\in\Delta}V_p
    \end{align}
    provided we show that it has an implicit atlas with boundary which, when restricted to the boundary, gives (a) the usual atlas (from section \cite[Section 9]{Pardon-VFC}) on $[\overline\clM_h(C,d)]^\text{vir}\times\{\pm1\}$ and (b) the standard orbifold atlas on each $V_p$. Constructing such an atlas is straightforward. Indeed, choose open neighborhoods
    \begin{align}
        V_p\Subset W_p \Subset W'_p\Subset \clM_h(C,d)\times(-1,1)
    \end{align}
    with $\{W'_p\}_{p\in\Delta}$ still being pairwise disjoint. Endow each $W'_p\setminus V_p$ with the natural orbifold-with-boundary atlas (with trivial obstruction spaces) as in \cite[Section 2.1.2]{Pardon-VFC} and endow the open subset
    \begin{align}
        \clM_{I}\setminus\bigcup_{p\in\Delta}\overline W_p
    \end{align}
    with the restriction of the product atlas (see \cite[Definition 6.3.1]{Pardon-VFC}) on $\clM_{I}$. These $1+|\Delta|$ open sets cover the moduli space in question. Over the open overlaps $W'_p\setminus\overline W_p\subset\clM_h(C,d)\times(-1,1)$ -- note that there are no triple overlaps -- both implicit atlases are equivalent (in the sense of \cite[Remark 2.1.5]{Pardon-VFC}). This is because the standard orbifold atlas on $\clM_h(C,d)$ is equivalent to the restriction of the implicit atlas on $\overline\clM_h(C,d)$. Thus, we can combine this cover into an implicit atlas on the space \eqref{aux-cobordism} with the desired properties.
\end{proof}
\begin{Lemma}[Local contribution]\label{local-contribution}
    In the situation of Lemma \ref{reduction-to-local-contributions}, the neighborhoods $V_p$ can be chosen such that their closures are pairwise disjoint and we have
    \begin{align}
        \normalfont\langle[\partial V_p], e(\clN^{d,h}_{C,\clD})\rangle = \frac{2\cdot\text{sgn}(\clD,p)}{|\text{Aut}(p)|}
    \end{align}
    for each $p\in\Delta$.
\end{Lemma}
\begin{proof}
    Fix $p = [f:C'\to C]\in\Delta\subset\clM_h(C,d)$. Since $\clM_h(C,d)$ is a smooth orbifold, we can identify a neighborhood $U_p$ of $p$ diffeomorphically with the orbifold $[\bV/\Gamma]$, where $\Gamma = \text{Aut}(C,f)$ is the isotropy group of $p$ with its natural complex representation on $\bV = T_f\clM_h(C,d)$ and the point $p$ is identified with $[0]$. Without loss of generality, we can assume that this neighborhood meets no other points of $\Delta$.
    
    Assume, for the sake of concreteness, that $\text{sgn}(\clD,p) = +1$. This can be achieved by reparametrizing the interval $[-1,1]$ in reverse if necessary. The local chart $[\bV/\Gamma]$ gives a map $\bV\xrightarrow{\nu}\clM_h(C,d)$ which, in turn, yields the following $\Gamma$-equivariant commutative diagram (with the square being a fiber square).
    \begin{center}
    \begin{tikzcd}
        C' \arrow[d] \arrow[r, hook, "j"] \arrow[rr, bend left, "f"] & \clC' \arrow[d, "\pi"] \arrow[r, "\text{ev}"] & C \\
        *\arrow[r, hook, "0"] & \bV
    \end{tikzcd}
    \end{center}
    Explicitly, $\pi:\clC'\to\bV$ is obtained via pulling back the universal curve over $\clM_h(C,d)$ to $\bV$ along the local \'etale chart $\nu$ and defines a $\Gamma$-equivariant smooth family of curves over $\bV$. The map $\text{ev}:\clC'\to C$ is the pullback along $\nu$ of the evaluation map to $C$ defined on the universal curve over $\clM_h(C,d)$, and it is a $\Gamma$-invariant smooth map which is holomorphic on each fiber of $\pi$. The maps $\pi$ and $\text{ev}$ together with the $\Gamma$-action realize the open embedding $[\bV/\Gamma]\hookrightarrow\clM_h(C,d)$. We can then define $V_p = [B_\bV/\Gamma]$, where $B_\bV$ is the ball bounded by $S_\bV$, and conclude by Lemma \ref{lem_sphere} and noting that $\langle[S_\bV],e(TS_\bV)\rangle = \chi(S_\bV) = 2$. It is useful to note that, for this last argument to work, we \emph{do not} require the isomorphism \eqref{tgt-bundle-iso} to be $\Gamma$-equivariant. 
\end{proof}
\begin{Lemma}\label{lem_sphere}
Use the same notations as in Lemma \ref{local-contribution}. For any sufficiently small sphere $S_\bV\subset\bV\times[-1,1]$ centered at $0$ (defined with respect to the orthogonal direct sum of a $\Gamma$-invariant Hermitian inner product on $\bV$ and the standard inner product on $\bR$), we have an isomorphism
    \begin{align}\label{tgt-bundle-iso}
        (\tilde\nu^*\clN^{d,h}_{C,\clD})|_{S_\bV}\simeq TS_\bV
    \end{align}
    of oriented vector bundles on $S_\bV$, where $\tilde\nu:\bV\times[-1,1]\to\clM_h(C,d)\times[-1,1]$ is given by $\tilde\nu = \nu\times{\normalfont\text{id}}$.
\end{Lemma}
\begin{proof}   
    Let us first make the following auxiliary choices.
    \begin{enumerate}
        \item A Hermitian metric $(\cdot,\cdot)$ and a compatible $\bC$-linear connection $\nabla$ on $N$.
        \item A Riemannian metric (and therefore, an area form) on $C'$ whose conformal class coincides with the one on $C'$ induced by $j$.
        \item A smooth trivialization $\clC' = C'\times\bV$ of the bundle $\pi$. Let $\tilde f:C'\times\bV\to C$ denote the map induced by $\text{ev}$, and set $f_v:=\tilde f(\cdot,v):C'\to C$ for $v\in\bV$. Also, let $\tilde j = (j_v)_{v\in V}$ denote the induced family of almost complex structures on $C'$.
        \item A smooth family of isomorphisms
        \begin{align}
            I_v:(T_{C'},j_0)\xrightarrow{\simeq}(T_{C'},j_v).
        \end{align}
    \end{enumerate}
    We now have the following smooth family of Cauchy--Riemann operators on $N$ parametrized by $(v,t)\in\bV\times[-1,1]$
    \begin{align}
        D_{v,t}:L^2_1(C',f^*N)\to L^2(C',\Lambda^{0,1}T^*_{C',j_0}\otimes_\bC f^*N)
    \end{align}
    defined by pulling back the operator $f_v^* D_t$ on $f_v^*N$ to an operator on $f^*N$ (recall: $f = f_0$) using the $\nabla$-parallel transport along the paths $\tau\in[0,1]\mapsto f_{\tau v}(z)\in C$ for each $z\in C'$ and the isomorphism $I_v^{-1}$. Obviously, we have $D_{0,0} = f^*D_0$ on smooth sections. Since $\ker D_0 = \coker D_0 = 0$ (and this will continue to hold for small $t$), we see from \eqref{new-coker} that $\coker D_{v,t} = \coker D_{v,t} / \coker D_t = \coker(f_v^* D_t)'$ for small $t$. Moreover, for small $(v,t) \neq (0,0)$, we have $\ker D_{v,t} = 0$ (even though $\ker D_{0,0}$ is $1$-dimensional) and so, $\coker D_{v,t}$ forms a vector bundle for small $(v,t)\ne(0,0)$.
    Therefore, we have an isomorphism
    \begin{align}
        \coker D_{v,t}\simeq (\tilde\nu^*\clN^{d,h}_{C,\clD})_{v,t}
    \end{align}
    for small $(v,t)\ne (0,0)$. Orient $A = \ker(f^*D_0)'$ and $B = \coker(f^*D_0)'$ so that the induced orientation on $\det(f^*D_0)'$ is the canonical one from Definition \ref{coker-bundle-def}. Define the bounded linear map 
    \begin{align}
        R:\text{im}\,D_{0,0}\to(\ker D_{0,0})^\perp
    \end{align} 
    to be the inverse of $D_{0,0}|_{(\ker D_{0,0})^\perp}$, with $(\ker D_{0,0})^\perp\subset L^2_1$ denoting the $L^2$ orthogonal complement of $\ker D_{0,0}$. Let $\Pi_\text{im}:L^2\to\text{im}\,D_{0,0}$ be the orthogonal projection and let $\Pi_\coker$ be the quotient projection onto the cokernel. For small $(v,t)$, the quantity $\|1 - \Pi_\text{im}D_{v,t}R\|$ is small on $\text{im}\,D_{0,0}$ and so, we can define the map 
    \begin{align}
        R_{v,t} = R(\Pi_\text{im}D_{v,t}R)^{-1}:\text{im}\,D_{0,0}\to(\ker D_{0,0})^\perp    
    \end{align}
    which is the inverse of $\Pi_\text{im}D_{v,t}|_{(\ker D_{0,0})^\perp}$. It is immediate that $\ker( \Pi_\text{im}D_{v,t})$ is given by the graph of the map $- R_{v,t}\Pi_\text{im}D_{v,t}:\ker D_{0,0}\to(\ker D_{0,0})^\perp$. Define the smooth family of linear maps
    \begin{align}
        L_{v,t}&:\ker D_{0,0}\to \coker D_{0,0}\\
        L_{v,t} &= \Pi_\coker D_{v,t}(1 - R_{v,t}\Pi_\text{im}D_{v,t})
    \end{align}
    for small $(v,t)$. We now observe that that the natural maps $\ker L_{v,t} \xrightarrow{\simeq} \ker D_{v,t}$ and $\coker L_{v,t} \xrightarrow{\simeq} \coker D_{v,t}$ are isomorphisms. It follows (using the identifications $A = \ker D_{0,0}$ and $B = \coker D_{0,0}$ which hold because $\ker D_0 = \coker D_0 = 0$) that the oriented vector bundle $\tilde\nu^*\clN^{d,h}_{C,\clD}$ is given, on a punctured neighborhood of $(0,0)$, by the assignment
    \begin{align}\label{eqn:bundle-coker}
        (v,t)\mapsto\coker(L_{v,t}:A\to B)
    \end{align}
    with the bundle orientation derived from the orientations already chosen on $A$ and $B$. Note that $L_{0,0} = 0$ and that the first order Taylor expansion of $L_{v,t}$ about $(0,0)$ is given by the positively oriented linear isomorphism 
    \begin{align}\label{eqn:orientation-spehere}
        \bV\oplus\bR_t\xrightarrow{\simeq}\text{Hom}(A,B)    
    \end{align}
    from \eqref{linear-wc-iso}, which we recall is obtained from Definition \ref{linearize-family-op}. Fixing a unitary $\Gamma$-invariant metric on $\bV$ and restricting to any sufficiently small sphere $S_\bV\subset\bV\times[-1,1]$ centred at the origin, we can discard (i.e. homotope away) the higher order terms from the Taylor expansion of $L_{v,t}$. It follows that the fiber of the vector bundle \eqref{eqn:bundle-coker} at a point $(v_0,t_0)\in S_\bV$ is identified with the quotient of the oriented vector space $\bV\oplus\bR_t$ by the oriented $1$-dimensional space spanned by $(v_0,t_0)$. But this is simply the standard description of the (oriented) tangent bundle of the sphere $S_\bV\subset\bV\oplus\bR_t$, once we replace ``quotient" by ``orthogonal complement" and \eqref{tgt-bundle-iso} now follows.
\end{proof}
Note that Lemma \ref{lem_sphere} holds for any $d$ and $h$. It is ultimately a consequence of the isomorphism \eqref{linear-wc-iso} via the finite dimensional reduction argument presented above.
\begin{remark}
    A coordinate invariant way of summarizing the above proof is the following. The bundle $\clN^{d,h}_{C,\clD}\to(\overline\clM_h(C,d)\times[-1,1])\setminus(\Delta\times 0)$ extends to the real oriented blow-up of $\overline\clM_h(C,d)\times[-1,1]$ along the smooth points $\Delta\times 0$ as a vector bundle $\overline\clN^{d,h}_{C,\clD}$. Moreover, for each point $p\in\Delta$, we have a natural vector bundle isomorphism with orientation $\normalfont\text{sgn}(\clD,p)$
    \begin{align}
        \overline\clN^{d,h}_{C,\clD}|_{\bS_p} = T\bS_p
    \end{align}
    over $\bS_p$, the exceptional boundary component created by blowing up at $p$. Now, a direct cobordism argument as in Lemma \ref{reduction-to-local-contributions}, applied to the real oriented blow-up $\normalfont\text{Bl}^{+}_{\Delta\times 0}(\overline\clM_h(C,d)\times[-1,1])$ gives \eqref{euler-change-formula}.
\end{remark}

\subsection{Comparision to Gopakumar--Vafa invariants}\label{comparison-to-GV}

In this subsection, we identify Taubes' Gromov invariant in dimension $6$ (Equation \eqref{Gr-def}) with the Gopakumar--Vafa invariants in the case of genus $0$ and class $2A$ for $A\in H_2(X,\bZ)$ primitive. We omit the underlying Calabi--Yau $3$-fold $(X,\omega)$ from the notation.

\begin{Theorem}\label{thm_GWGV}
    The invariant $\normalfont\text{Gr}_{2A,0}$ satisfies the Gopakumar--Vafa formula. More precisely, define, for super-rigid $J$,
    \begin{align}
        \label{eqn:GW-A}
        \normalfont\text{GW}_{A,0} &:= \normalfont\text{Gr}_{A,0} := \sum_{C :\,A,0} \text{sgn}(D^N_{C,J}) \\
        \label{eqn:GW-2A}\normalfont \text{GW}_{2A,0} &:= \normalfont \sum_{C':\,2A,0}\text{sgn}(D^N_{C',J}) + \sum_{C:\,A,0}\text{sgn}(D^N_{C,J})e_{2,0}(C,D^N_{C,J})
    \end{align}
    where $e_{2,0}$ denotes the virtual Euler number as defined in \eqref{virt-eu-num} and the summations follow the notation of \eqref{Gr-def}.\footnote{For super-rigid $J$, this definition indeed recovers the Gromov--Witten invariants, see \cite[Theorem 1.2]{zinger2011comparison}.} Then, we have the following relation between $\normalfont\text{GW}$ and $\normalfont\text{Gr}$.
    \begin{align}\label{eqn:2A-GV-GW}
        \normalfont \text{GW}_{2A,0} =  \text{Gr}_{2A,0} + \frac1{2^3}\cdot\normalfont \text{Gr}_{A,0}.
    \end{align}
\end{Theorem}
\begin{proof}
    From \eqref{eqn:GW-A}--\eqref{eqn:GW-2A} and \eqref{Gr-def}, it is clear that \eqref{eqn:2A-GV-GW} will follow once we show that
    \begin{align}\label{eqn:euler-vs-lwc}
        e_{2,0}(C,D^N_{C,J}) - \frac1{2^3} = \text{w}_{2,0}(D^N_{C,J}).
    \end{align}
    Let $\bar\partial$ denote an operator on $N_C$ which is conjugate to the standard operator on $\clO(-1)^{\oplus2}\to\bP^1$. The Aspinwall--Morrison formula (see, for example \cite{Voisin-AM}) shows that we have the identity
    \begin{align}\label{eqn:AM}
        e_{2,0}(C,\bar\partial) = \frac1{2^3}.
    \end{align}
    Thus, we are left to prove that the difference of Euler numbers is given by the linear wall crossing invariant $\text{w}_{2,0}(D^N_{C,J})$. To compute $\text{w}_{2,0}(D^N_{C,J})$, choose a smooth path $\gamma$ lying in the Baire set $\clP^\pitchfork(D^N_{C,J},\delbar)$ from Definition \ref{defn:2-rig}. Let $-1< t_1<\cdots<t_m< 1$ be the values for which $\gamma(t_i)$ lies in the image of $\clW^1_{2,0,\bR}(N_C)\to\clC\clR_\bR(N_C)$. Using Lemma \ref{triv-wall-codim}, we can ensure that $\gamma$ has only finitely many transverse intersections (at times $-1<t'_1<\cdots<t'_n<1$) with $\clW^1_\bR(N_C)$ and, moreover, that $\gamma$ misses $\clW^k_\bR(N_C)$ for all $k\ge 2$. Moreover, we can also ensure (by Lemma \ref{lem:GV-triv-sign-ker-codim} below) that $t_i\ne t'_j$ for $1\le i\le m$ and $1\le j\le n$. Clearly, $t\mapsto e_{2,0}(C,D_t)$ is a locally constant function away from $\{t_i\}_{1\le i\le m}$ and $\{t'_j\}_{1\le j\le n}$.
    
    For any $1\le j\le n$, the operator $D_{t'_j}$ has a $1$-dimensional kernel. However, since $t'_j\ne t_i$ for all $1\le i\le m$, it follows that $D_{t'_j}$ is $2$-rigid (in the sense of Definition \ref{defn:2-rig}) and, therefore, also $(2,0)$-rigid (in the sense of Definition \ref{coker-bundle-def}). It follows that $\text{w}_{2,0}(D_t)$ remains constant for $t$ near $t'_j$. It also follows that the cokernel bundle $\clN^{d,h}_{C,D_t}$ is well-defined for all $t$ near $t'_j$ (see Appendix \ref{obstruction-bundle-construction-details}) and thus its Euler number $e_{2,0}(D_t)$ also remains constant for $t$ near $t'_j$. The reader may refer to Remark \ref{further-orientation-explanation} to recall the orientation convention and to check that $e_{2,0}(D_t)$ is indeed unchanged near $t_j'$.
    
    For any $1\le i\le m$, we have $\ker D_{t_i} = 0$ (since $t_i\ne t'_j$ for all $1\le j\le m$). For any point lying on the (necessarily transverse) fiber product of $\gamma$ and $\clW^1_{2,0,\bR}(N_C)\to\clC\clR_\bR(N_C)$, i.e., a branched double cover $\varphi:C'\to C$ (with $C'$ smooth of genus $0$) such that $\ker(D_{t_i})^-_{\varphi} = \ker\varphi^*D_{t_i} = \ker(\varphi^*D_{t_i})'$ is $1$-dimensional, we get an isomorphism
    \begin{align}\label{eqn:GV-wall-transverse}
        T_{t_i}\gamma = \bR_t\xrightarrow{\simeq}\text{Hom}(\ker (D_{t_i})^-_{\varphi},\coker (D_{t_i})^-_{\varphi})/T_{\varphi}\clM_0(C,2)
    \end{align}
    whose sign is the amount that $\varphi:C'\to C$ contributes to the change in $\text{w}_{2,0}(D_t)$ as $t$ goes from $t<t_i$ to $t>t_i$. Recall that the right side of \eqref{eqn:GV-wall-transverse} is oriented using the natural co-orientation of $\clW^1_{2,0,\bR}(N_C)$ from Definition \ref{natural-coor}. From \eqref{eqn:GV-wall-transverse}, arguing as in Lemma \ref{lem:transverse-wall-lift}, we get that the analogue
    \begin{align}\label{eqn:GV-simple-wall}
        \bR_t\oplus T_\varphi\clM_0(C,2)\xrightarrow{\simeq}\text{Hom}(\ker (\varphi^*D_{t_i})',\coker (\varphi^*D_{t_i})')
    \end{align}
    of the map \eqref{linear-wc-iso} is also an isomorphism. Thus, the family $D_t$ undergoes a simple wall crossing (in the sense of Definition) at $t = t_i$ and Theorem \ref{euler-change-thm} now shows that the contribution of $\varphi:C'\to C$ to the change in $e_{2,0}(C,D_t)$ as $t$ goes from $t<t_i$ to $t>t_i$ is precisely the sign of \eqref{eqn:GV-simple-wall}. Moreover, \eqref{eqn:GV-simple-wall} induces \eqref{eqn:GV-wall-transverse} when we pass to the quotient by $T_\varphi\clM_0(C,2)$ and so, these two maps have the same sign. Here, we are using the identification $\ker(D_{t_i})^-_{\varphi} = \ker\varphi^*D_{t_i} = \ker(\varphi^*D_{t_i})'$, and similarly for cokernels, which comes from recalling that $\ker D_{t_i} = 0$. To summarize, $\text{w}_{2,0}(C,D_t)$ and $e_{2,0}(C,D_t)$ change by the same amount as $t$ goes from $t<t_i$ to $t>t_i$.
    
    The last two paragraphs show that $e_{2,0}(D_t) - \text{w}_{2,0}(D_t)$ is a constant function in $t$ and this finishes the proof of \eqref{eqn:euler-vs-lwc} once we recall the identity \eqref{eqn:AM} and that $\text{w}_{2,0}(\delbar) = 0$.
\end{proof}

\begin{Lemma}\label{lem:GV-triv-sign-ker-codim}
    Using the notation in the proof of Theorem \ref{thm_GWGV}, the subset
    \begin{align}
        \clW_{2,0,\bR}^{1,1}(N_C)\subset\clC\clR_\bR(N_C)\times M_0(C,2)
    \end{align} 
    consisting of $(j,D,[\varphi:C'\to C])$ with
    \begin{align}
        \label{triv-ker-GV}\dim_\bR\ker D &= 1 \\
        \label{sign-ker-GV}\dim_\bR\ker D_\varphi^- &= 1
    \end{align}
    is a submanifold of real codimension $6$.
\end{Lemma}
\begin{proof}
    Let $(j,D,[\varphi:C'\to C])\in\clW^{1,1}_{2,0,\bR}(N_C)$. Using the argument which appears repeatedly in \textsection 5 (e.g. see the proofs of Lemmas \ref{triv-wall-codim} and \ref{sign-wall-codim}), it suffices to show that the linear map
    \begin{align}\label{triv-sign-ker-GV-linearization}
        \Omega^{0,1}_j(C,\text{Hom}_\bR(N,N))&\to\text{Hom}_\bR^{\bZ/2\bZ}(\ker\varphi^*D,\coker\varphi^*D) \\
        A&\mapsto(\kappa\mapsto(\varphi^*A)\kappa\pmod{\text{im }\varphi^*D})
    \end{align}
    analogous to \eqref{eqn:sign-wall-linearize} is surjective, since \eqref{triv-ker-GV}--\eqref{sign-ker-GV} imply that the right side of \eqref{triv-sign-ker-GV-linearization} is 6-dimensional over $\bR$. Exactly as in the proof of Lemma \ref{sign-wall-codim}, this is equivalent to the injectivity of the ``Petri map"
    \begin{align}\label{eqn:GV-double-petri}
        (\ker\varphi^*D\otimes_\bR\ker(\varphi^*D)^\dagger)^{\bZ/2\bZ}&\to\Omega^0(C',(\varphi^*N)\otimes_{\bR}(\varphi^*N)^\dagger)^{\bZ/2\bZ} \\
        \sigma\otimes_\bK\sigma'&\mapsto(\zeta\mapsto\sigma(\zeta)\otimes_\bK\sigma'(\zeta))
    \end{align}
    dual to \eqref{triv-sign-ker-GV-linearization} and the latter follows from, e.g., \cite[Remark 1.3.12]{DW-20}.
\end{proof}
\appendix
\setcounter{section}{0}
\section{Kuranishi reduction}\label{kuranishi-appendix}

In this appendix, we collect some useful elementary observations about Kuranishi reduction (i.e., application of the implicit function theorem in Fredholm, but not necessarily transverse, situations) for which we are not aware of appropriate references in the literature.

Let $\clX,\clY$ be (possibly finite dimensional) Banach spaces and let
\begin{align}
    \Phi:\clX\to\clY
\end{align}
be a smooth Fredholm map, defined on an open neighborhood of $0\in\clX$, such that we have $\Phi(0) = 0\in\clY$. Consider the linearization
\begin{align}
    \tilde D = D\Phi(0):\clX\to\clY
\end{align}
and define the (necessarily finite dimensional) vector spaces $T:=\ker\tilde D$ and $E:=\coker\tilde D$. Split the inclusion $T\to\clX$ and the projection $\clY\to E$ by choosing a Banach space complement $\clX'\subset\clX$ of $T$ and a linear subspace lift $E\subset\clY$. Applying the implicit function theorem now shows that $\Phi^{-1}(E)\subset\clX$ is identified (near $0\in\clX$) with the graph of a smooth map $\Psi:T\to\clX'$ defined on a neighborhood of $0\in T$. This map satisfies $\Psi(0) = 0\in\clX'$ and $D\Psi(0) = 0$. The associated \emph{Kuranishi map}
\begin{align}\label{abstract-kuranishi-map}
    F:T\to E,
\end{align}
defined near $0\in T$, is given by $F(\kappa) := \Phi(\kappa + \Psi(\kappa))$ for $\kappa\in T$. It satisfies $F(0) = 0$ and $DF(0) = 0$ by construction. We refer to the process of constructing $F$ from $\Phi$ as \emph{Kuranishi reduction} and note that it depends on the choices of splittings of the maps $T\to\clX$ and $\clY\to E$. We first investigate this dependence.

\begin{Lemma}[Splitting dependence of Kuranishi map]\label{abstract-kuranishi-invariance}
    Let $\tilde F:T\to E$ be a Kuranishi map obtained using a different choice of splittings than in \eqref{abstract-kuranishi-map}. Then, there exists a local diffeomorphism
    \begin{align}
        \psi:T\to T
    \end{align}
    with $\psi(0) = 0$ and $D\psi(0) = \normalfont\text{id}_T$ and a smooth map
    \begin{align}
        \tilde\psi:T\to \normalfont\text{End}(E) 
    \end{align}
    with $\tilde\psi(0) = \normalfont\text{id}_E$ such that, for all $\kappa\in T$ near $0$, we have
    \begin{align}\label{kuranishi-invariance}
        \tilde F(\kappa) = \tilde\psi(\kappa)F(\psi(\kappa)).
    \end{align}
\end{Lemma}
\begin{proof}
    Via the splitting used to define \eqref{abstract-kuranishi-map}, we get direct sum decompositions $\clX = \clX'\oplus T$ and $\clY = \text{im}(\tilde D)\oplus E$.
    
    Consider first the case when $\tilde F$ is obtained by the changing the splitting of $T\to\clX$ but using the same splitting of $\clY\to E$ as in \eqref{abstract-kuranishi-map}. Denote the resulting complement of $T$ by $\tilde\clX'\subset\clX$. 
    Using the implicit function theorem, we write $\Phi^{-1}(E)$ as the graph of a smooth map $\tilde\Psi:T\to\tilde\clX'$ near $0\in\clX$, with $\tilde\Psi(0) = 0$ and $D\tilde\Psi(0) = 0$. Since $\Phi^{-1}(E)$ is locally also the graph of the smooth map $\Psi:T\to\clX'$, we must have
    \begin{align}
        \Pi_{\clX'}\tilde\Psi(\kappa) &= \Psi(\kappa + \Pi_T\tilde\Psi(\kappa))
    \end{align}
    for all $\kappa\in T$ near $0$, where $\Pi_{\clX'}$ and $\Pi_T$ denote the projections onto the summands of the decomposition $\clX = \clX'\oplus T$. This immediately yields $\tilde F(\kappa) \equiv F(\psi(\kappa))$ with $\psi:T\to T$ defined by
    \begin{align}
        \psi(\kappa) = \kappa + \Pi_T\tilde\Psi(\kappa).    
    \end{align}
    Since $\tilde\Psi(0) = 0$ and $D\tilde\Psi(0) = 0$, we see that $\psi$ satisfies $\psi(0) = 0$ and $D\psi(0) = \text{id}_T$ and is thus a local diffeomorphism.
    
    Consider next the case when $\tilde F$ is obtained by keeping the same splitting of $T\to\clX$ as in \eqref{abstract-kuranishi-map} but changing the splitting of the map $\clY\to E$. Denote the resulting complement $\text{im}(\tilde D)$ by $\tilde E\subset\clY$. Consider the map
    \begin{align}
        \Phi_E:\clX\oplus E\to\clY
    \end{align}
    given by $\Phi_E(x,e) = \Phi(x) - e$. Using the implicit function theorem, we may write $\Phi_E^{-1}(\tilde E)$, near $0$, as the graph of a smooth map $\tilde\Psi_E:T\oplus E\to\clX'$. Using $\tilde\Psi_E$, we define the smooth map $\tilde F_E:T\oplus E\to E$ by the formula
    \begin{align}
        \tilde F_E(\kappa,e) = \Pi_E(\Phi(\kappa + \tilde\Psi_E(\kappa,e)) - e)
    \end{align}
    where $\Pi_E:\clY\to E$ denotes the projection. Restricting $\tilde\Psi_E$ to $T\oplus \{0\}$ recovers the map $\tilde\Psi:T\to\clX'$ whose graph coincides with $\Phi^{-1}(\tilde E)$ near $0\in\clX$ (which also shows that $\tilde F_E(\kappa,0) = \tilde F(\kappa)$ for $\kappa\in T$). Therefore, as before, we have $\tilde\Psi(0) = 0$ and $D\tilde\Psi(0) = 0$ which shows that $D\tilde F_E(0,0)$ is given by (the negative of) the projection onto $E$. Now, if $(\kappa,e)$ is a point where $\tilde F_E$ vanishes, then since $\Pi_E|_{\tilde E}:\tilde E\to E$ is an isomorphism it follows that $\Phi(\kappa + \tilde\Psi_E(\kappa,e)) - e = 0$. This, in turn, implies that we must in fact have $\tilde\Psi_E(\kappa,e) = \Psi(\kappa)$ and $e = F(\kappa)$. Conversely, if $e = F(\kappa)$, then $\Phi_E(\kappa + \Psi(\kappa),e) = \Phi(\kappa + \Psi(\kappa)) - e = 0\in\tilde E$ and thus, $\tilde\Psi_E(\kappa,e) = \Psi(\kappa)$ and $\tilde F_E(\kappa,e) = 0$. In summary, $\tilde F_E$ is a defining equation, near $(0,0)$, for the submanifold of $T\oplus E$ given by the graph of the map $F:T\to E$, i.e., the equation $\tilde F_E\ = 0$ transversely cuts out the graph of $F$ near $(0,0)$. Since $F_E(\kappa,e) :=  F(\kappa) - e$ is also a defining equation for the graph of $F$, and $DF_E(0,0) = D\tilde F_E(0,0)$, it follows that there exists a smooth map $\tilde\psi_E:T\oplus E\to\text{End}(E)$ with $\tilde\psi_E(0,0) = \text{id}_E$ such that
    \begin{align}
        \tilde F_E(\kappa,e) = \tilde\psi_E(\kappa,e)\cdot F_E(\kappa,e)
    \end{align}
    for all $(\kappa,e)\in T\oplus E$ near $(0,0)$. Restricting to $e = 0$ and defining the map $\tilde\psi:T\to GL(E)$ by $\tilde\psi(\kappa) = \tilde\psi_E(\kappa,0)$, we obtain $\tilde F(\kappa) \equiv \tilde\psi(\kappa)F(\kappa)$.
    
    Combining the arguments of the preceding two paragraphs produces $\psi,\tilde\psi$ and the identity \eqref{kuranishi-invariance} in the general case.
\end{proof}

Lemma \ref{abstract-kuranishi-invariance} has the following important consequence.

\begin{Corollary}[Invariance of leading order term]\label{abstract-kuranishi-leading-order}
    If $n\ge 1$ is an integer with $D^jF(0) = 0$ for $0\le j\le n$, then we also have $D^j\tilde F(0)$ for $0\le j\le n$. Moreover, we have the equality
    \begin{align}
        D^{n+1}F(0) = D^{n+1}\tilde F(0)\in\normalfont\text{Hom}(\text{Sym}^{n+1}T,E).
    \end{align}
\end{Corollary}

\begin{remark}\label{changing-infinite-dim-trivializations}
    It is important to note that, in the setting of Corollary \ref{abstract-kuranishi-leading-order}, the leading term $D^{n+1}F(0)$ is also invariant under replacing the map $\Phi$ by the map $x\mapsto\Phi_{\chi,\tilde\chi}(x) = \tilde\chi(x)\Phi(\chi(x))$, where $\chi:\clX\to\clX$ is any local diffeomorphism (defined near $0\in\clX$) with $\chi(0) = 0$ and $D\chi(0) = \text{id}_{\clX}$ and $\tilde\chi:\clX\to\clL(\clY)$ is any smooth map to the space of bounded linear operators on $\clY$ (defined near $0\in\clX$) with $\tilde\chi(0) = \text{id}_{\clY}$. The proof is analogous to that of Lemma \ref{abstract-kuranishi-invariance}. Note that we have $\tilde D = D\Phi(0) = D\Phi_{\chi,\tilde\chi}(0)$.
\end{remark}

We next investigate how a perturbation of $\Phi$ (vanishing to high order) affects the Kuranishi map $F$. More precisely, let $n\ge 1$ be an integer and let $P:\clX\to\clY$ be a smooth map with the property that $D^jP(0) = 0$ for $0\le j\le n$. Observe that $D^{n+1}P(0)$ defines an element of $\text{Hom}(\text{Sym}^{n+1}T,E)$ by restriction along $T\to\clX$ and projection along $\clY\to E$.

Define $\hat\Phi = \Phi + P : \clX\to\clY$. Then, $D\hat\Phi(0) = D\Phi(0) = \tilde D$ and thus, $\hat\Phi$ is still Fredholm in a neighborhood of $0\in\clX$. Thus, we may perform Kuranishi reduction on $\hat\Phi$ using the same choices of splittings of $T\to\clX$ and $\clY\to E$ as in \eqref{abstract-kuranishi-map} to obtain the Kuranishi map
\begin{align}
    \hat F:T\to E.
\end{align}
Write $Q = \hat F - F$. In this situation, we have the following useful fact.
\begin{Lemma}[Effect of high order perturbation]\label{abstract-kuranishi-high-order-perturb}
    $D^jQ(0) = 0$ for $0\le j\le n$. Moreover, we have the equality
    \begin{align}
        D^{n+1}Q(0) \equiv D^{n+1}P(0)
    \end{align}
    as elements of $\normalfont\text{Hom}(\text{Sym}^{n+1}T,E)$.
\end{Lemma}
\begin{proof}
    Let $\hat\Psi:T\to\clX'$ be the map whose graph coincides with $\hat\Phi^{-1}(E)\subset\clX$ near $0$. Since $D^j\Phi(0) = D^j\hat\Phi(0)$ for $0\le j\le n$, it follows by using the chain rule that the we also have $D^j\Psi(0) = D^j\hat\Psi(0)$ for $0\le j\le n$. This readily yields $D^jF(0) = D^j\hat F(0)$ for $0\le j\le n$. Next, we compute
    \begin{align}
        D^{n+1}Q(0) &= D^{n+1}\hat F(0) - D^{n+1}F(0) \\
        &= D^{n+1}(\hat\Phi\circ(\text{id}_T + \hat\Psi))(0) - D^{n+1}(\Phi\circ(\text{id}_T + \Psi))(0).
    \end{align}
    We break the last sum into the two following parts:
    \begin{align}
        \label{first-part-high-derivative}
        &D^{n+1}(P\circ(\text{id}_T + \hat\Psi))(0) \\
        \label{second-part-high-derivative} D^{n+1}(\Phi\circ(\text{id}_T + \hat\Psi))(0) - &D^{n+1}(\Phi\circ(\text{id}_T + \Psi))(0).
    \end{align}
    Simplifying \eqref{first-part-high-derivative} using the chain rule just yields $D^{n+1}P(0)$ since all lower derivatives of $P$ vanish at $0$ and $D\hat\Psi(0) = 0$. On the other hand, simplifying \eqref{second-part-high-derivative} using the chain rule just yields (a term proportional to) the term
    \begin{align}\label{higher-derivative-error}
        D\Phi(0)\circ D^{n+1}(\hat\Psi - \Psi)(0).
    \end{align}
    Indeed, all other terms cancel since $D^j\hat\Psi(0) = D^j\Psi(0)$ for $0\le j\le n$. But now we observe that \eqref{higher-derivative-error} lies in $\text{Hom}(\text{Sym}^{n+1}T,\text{im}(\tilde D))$ and, as a result, has zero image in $\text{Hom}(\text{Sym}^{n+1}T,E)$. This shows that we must have $D^{n+1}Q(0)\equiv D^{n+1}P(0)$ in $\text{Hom}(\text{Sym}^{n+1}T,E)$ as claimed.
\end{proof}

\section{Obstruction bundles}\label{obstruction-bundle-construction-details}

In this appendix, we show how to construct obstruction/cokernel bundles under the (somewhat non-standard) assumption of $(d,h)$-rigidity in Definition \ref{coker-bundle-def}. The proof will be given in the generality needed to establish that the condition \eqref{no-new-ker} is open and that \eqref{families-version-cokernel-bundle} defines an orbi-bundle. The argument will make use of ``thickenings" as in \textsection\ref{compactness-section} (and we use very similar notation to facilitate the comparison). For the convenience of the reader, we begin by briefly recalling the relevant setup from \textsection\ref{sec_euler}.

We fix a smooth genus $g$ curve $C$ and integers $h\ge g$ and $d\ge 2$ such that the moduli stack $\overline\clM_h(C,d)$ has (real) virtual dimension $2r\ge 0$. We also fix a complex vector bundle $N\to C$ of rank $2$ and degree $2g-2$ along with a continuous family $\clD = \{D_\lambda\}_{\lambda\in P}$ of real Cauchy--Riemann operators on $N$, parametrized by a topological space $P$. Let $\hat\Delta\subset\overline\clM_h(C,d)\times P$ consist of exactly those pairs $([\varphi:C'\to C],\lambda)$ for which the condition
\begin{align}\label{no-new-ker-appx}
    \ker(\varphi^*D_\lambda)' := \coker(\ker D_\lambda\xrightarrow{\varphi^*}\ker (\varphi^*D_\lambda))  = 0
\end{align}
fails. We will show that $\hat\Delta\subset\overline\clM_h(C,d)\times P$ is closed and that the assignment, to $([\varphi:C'\to C],\lambda)$, of the real vector space
\begin{align}\label{new-coker-appx}
    \coker(\varphi^*D_\lambda)':= \coker(\coker D_\lambda\xrightarrow{\varphi^*}\coker (\varphi^*D_\lambda))
\end{align}
defines an orbi-bundle $\clN^{d,h}_{C,\clD}$ (of rank $2r$) over the open complement of $\hat\Delta$.

To this end, fix any point $([\varphi:C'\to C],\lambda)\in\overline\clM_h(C,d)\times P$. Choose points $p_1,\ldots,p_m\in C$, which are not the images of critical points of $\varphi$ or nodal points of $C'$, such that $(C,p_1,\ldots,p_m)$ is a stable curve. For $1\le i\le m$, choose a numbering $\varphi^{-1}(p_i) = \{q_i^1,\ldots,q_i^d\}$ and observe that $(C',\{q_i^j\})$ gives a point of $\overline\clM_{h,dm}$. Choose a complex manifold $\sM'$ (equipped with a basepoint $\bullet$), a finite group $\Gamma'$ (acting on $\sM'$ holomorphically and fixing $\bullet$) and an open embedding of complex orbifolds
\begin{align}\label{local-orbi-chart-appx}
    [\sM'/\Gamma']\hookrightarrow[\overline\clM_{h,dm}/(S_d)^m]
\end{align}
mapping the image of $\bullet$ to the image of $(C',\{q_i^j\})$, where the $i^\text{th}$ factor in the $m$-fold product of $S_d$ acts on $\overline\clM_{h,dm}$ by permuting the marked points $\{q_i^j\}_{j=1}^d$. By pulling back the universal curve over the right side of \eqref{local-orbi-chart-appx} to $\sM'$, we obtain a $\Gamma'$-equivariant (flat) complex analytic family $\pi':\sC'\to\sM'$ of (possibly nodal) arithmetic genus $h$ curves with sections $\tau_i^j$ (corresponding to $dm$ marked points which stabilize the fibers). For $1\le i\le m$, the sections $\tau_i^j$ for $1\le j\le d$ are permuted among themselves by the action of $\Gamma'$. For later use, we also fix an isomorphism
\begin{align}
    i':(C',\{q_i^j\})\xrightarrow{\simeq}(\sC'_{\bullet},\{\sigma_i^j(\bullet)\}).
\end{align}
    
Define $\sM'(C)$ to be the space consisting of pairs $(s'\in\sM',\psi:\sC'_s\to C)$ with $\psi$ being a holomorphic map satisfying $\psi_*[\sC_{s'}] = d[C]$ and $\psi(\tau_i^j(s')) = p_i$ for $1\le i\le m$ and $1\le j\le d$. The space $\sM'(C)$ carries an evident action of $\Gamma'$ and the quotient $[\sM'(C)/\Gamma']$ now provides an orbispace chart for a neighborhood of $[\varphi:C'\to C]$ in $\overline\clM_h(C,d)$.\footnote{Note that the purpose of taking the quotient by $(S_d)^m$ in \eqref{local-orbi-chart-appx} is to describe $\overline{\clM}_h(C,d)$ locally near $\varphi:C'\to C$ as the quotient stack $[\sM'(C)/\Gamma']$.} We will construct the orbi-bundle with fibers \eqref{new-coker-appx} by constructing a $\Gamma'$-equivariant vector bundle on (an open subset of) $\sM'(C)\times P$, which then descends to the quotient.

Choose a finite dimensional vector space $E$ and a linear map
\begin{align}\label{downstairs-datum-appx}
    \mu:E\to\Omega^{0,1}(C,N)
\end{align}
such that the induced map $E\to\coker D_\lambda$ is surjective. Next, choose a finite dimensional $\Gamma'$-representation $E'$ and a $\Gamma'$-equivariant map
\begin{align}\label{upstairs-datum-appx}
    \mu':E'\to C^\infty(\sC'^\circ\times C,\Omega^{0,1}_{\sC'^\circ/\sM'}\boxtimes_\bC N)
\end{align}
such that its restriction along $i'\times\varphi:C'\hookrightarrow\sC'\times C$ induces a surjective map $E'\to\coker(\varphi^*D_\lambda)'$. Here, $\sC'^\circ\subset\sC'$ denotes the complement of the (fiberwise) nodal locus of $\pi'$ and the sections in the image of $\mu'$ are required to have proper support over $\sM'$. For more details on how to construct \eqref{upstairs-datum-appx}, the reader is referred to the proof of \cite[Lemma 9.2.9]{Pardon-VFC}.

Now, consider the space $K$ of $(t\in P,\sigma\in C^\infty(C,N),e\in E)$ satisfying
\begin{align}\label{downstairs-thick-appx}
    D_{t}\sigma = \mu(e)
\end{align}
with its natural projection $K\to P$. By the choice of $\mu$, \eqref{downstairs-thick-appx} describes the kernel of a family of surjective (thickened) Cauchy--Riemann operators (near $\lambda$). Thus, using a standard implicit function theorem argument, we see that $K$ defines a vector bundle over $P$ near $\lambda$ and there is a (forgetful) linear bundle map from $K\to\underline{E}$, where $\underline{E}$ is the trivial bundle with fiber $E$. Similarly, consider the space $K'$ of tuples of the form $((s',\psi)\in\sM'(C),t\in P,\sigma'\in C^\infty(\sC'_{s'},\psi^*N),e\in E,e'\in E')$ satisfying
\begin{align}\label{upstairs-thick-appx}
    (\psi^*D_{t})\sigma' = \psi^*(\mu(e)) + \mu'(e')(\cdot,\psi(\cdot))
\end{align}
with its natural projection $K'\to\sM'(C)\times P$. By the choice of $\mu$ and $\mu'$, \eqref{upstairs-thick-appx} describes the kernel of a family of surjective (thickened) Cauchy--Riemann operators. Thus, standard arguments involving the implicit function theorem and (linear) gluing analysis (e.g. following \cite[Appendix B]{Pardon-VFC} or \cite[Theorem 1.2(2)]{zinger2011comparison}) show that $K'$ defines a vector bundle over $\sM'(C)\times P$ near $((\bullet,\varphi\circ i'^{-1}),\lambda)$ and there is a (forgetful) linear bundle map $K'\to\underline{E}\oplus\underline{E'}$. Moreover, $K'$ and this bundle map respect the action of $\Gamma'$.

Denoting the projection $\sM'(C)\times P\to P$ by $\pi$, we now have a continuous fiberwise linear map of two-term of complexes of vector bundles
\begin{center}
\begin{equation}\label{quotient-diagram-appx}
\begin{tikzcd}
    \pi^*K \arrow[r] \arrow[d] & {\underline{E}} \arrow[d] \\
    K' \arrow[r] & {\underline{E}\oplus\underline{E'}}
\end{tikzcd}
\end{equation}
\end{center}
described as follows. The horizontal arrows are the linear bundle maps defined earlier, while the right vertical arrow is the natural inclusion. The left vertical arrow is defined by setting $e' = 0$ and pulling back $\sigma\in C^\infty(C,N)$ to $\sigma' = \psi^*\sigma\in C^\infty(\sC_{s'},\psi^*N)$. Both vertical arrows in \eqref{quotient-diagram-appx} are injective maps of vector bundles and passing to the quotient yields
\begin{align}
    L:K'/\pi^*K\to\underline {E'}.
\end{align}
$L$ is a $\Gamma'$-equivariant linear bundle map whose fiberwise kernel and cokernel are identified with the left sides of \eqref{no-new-ker-appx} and \eqref{new-coker-appx} respectively. We deduce that the complement of (the pre-image of) $\hat\Delta$ is open (since the condition defining it amounts to $L$ being injective) and, over this open set, $\coker L$ forms a $\Gamma'$-equivariant vector bundle. This provides an orbi-bundle structure on $\clN^{d,h}_{C,\clD}$ locally near any point $([\varphi:C'\to C],\lambda)$ of $\overline{\clM}_h(C,d)\times P$ at which \eqref{no-new-ker-appx} is satisfied. The rank of this orbi-bundle can be computed by noting that $\text{ind}(D_\lambda) = 0$ and $\text{ind}(\varphi^*D_\lambda) = -2r$.

To compare the local orbi-bundle presentations of $\clN^{d,h}_{C,\clD}$ obtained by this method for two different sets of choices, we embed both into the ``doubly thickened" versions of \eqref{downstairs-thick-appx}--\eqref{upstairs-thick-appx}. This allows us to construct the local orbi-bundle isomorphisms needed to patch $\clN^{d,h}_{C,\clD}$ globally over the locus where \eqref{no-new-ker-appx} is satisfied. For more details on how to define a common thickening to dominate two given thickenings, the reader is referred to \cite[Definition 9.2.3]{Pardon-VFC}. This completes the proof that $\hat\Delta$ is closed and that $\clN^{d,h}_{C,\clD}$ is a well-defined rank $2r$ orbi-bundle over its complement.

\begin{remark}
    Since we appealed to (linear) gluing analysis following \cite[Appendix B]{Pardon-VFC}, we only obtain a continuous orbi-bundle structure on $\clN^{d,h}_{C,\clD}$. Nevertheless, this construction suffices for our topological arguments. However, over the locus $(\clM_h(C,d)\times P)\setminus\hat\Delta$ involving only stable maps to $C$ with smooth domains it is easy to see from the above argument (which, in this case, only requires the implicit function theorem and no gluing analysis) that $\clN^{d,h}_{C,\clD}$ actually has a natural smooth orbi-bundle structure provided $\clD$ is a smooth family parametrized by a smooth manifold $P$.
\end{remark}

\bibliography{references}

\end{document}